\documentclass[11pt]{amsart}

\usepackage[vcentermath]{youngtab}

\usepackage{tikz,tikz-3dplot}
\usetikzlibrary{arrows, patterns, arrows.meta}
\usetikzlibrary{decorations.markings}
\usetikzlibrary{decorations.pathreplacing}
\usetikzlibrary{decorations.pathmorphing}
\usetikzlibrary{calc}
\usetikzlibrary{intersections}

\usepackage{mathtools}
\usepackage{stmaryrd, blkarray}

\def\sep{\discretionary{}{}{}}

\newtheorem{innercustomgeneric}{\customgenericname}
\providecommand{\customgenericname}{}
\newcommand{\newcustomtheorem}[2]{\newenvironment{#1}[1]
  {\renewcommand\customgenericname{#2}\renewcommand\theinnercustomgeneric{##1}\innercustomgeneric
  }
  {\endinnercustomgeneric}
}
\makeatletter
\def\newrepeatedtheorem#1[#2]#3{\newenvironment{#1}[1]
  {\expandafter\global\expandafter\advance\csname c@#2\endcsname\@ne
   \protected\def\customreference{\ref{##1}}\renewcommand\customgenericname{#3}\renewcommand\theinnercustomgeneric{\customreference}\innercustomgeneric
  }
  {\endinnercustomgeneric}
}
\makeatother

\makeatletter

\def\DeclareStructure#1#2{\expandafter\def\csname Declare#1\endcsname##1{\expandafter\def\csname structure:#1:\detokenize{##1}\endcsname
	}\protected\def#2##1{\ifcsname structure:#1:\detokenize{##1}\endcsname
\csname structure:#1:\detokenize{##1}\expandafter\endcsname
		\else
\def\next{\csname structure-generic:#1\endcsname{##1}}\expandafter\next
		\fi
	}\@ifnextchar:{\expandafter\def\csname structure-generic:#1\expandafter\endcsname\@gobble}{\expandafter\let\csname structure-generic:#1\expandafter\endcsname\@firstofone}}

\DeclareStructure{Set}{\set}{\mathrm}
\DeclareStructure{Category}{\cat}:#1{\cat@generate{#1}}
\DeclareStructure{Bicategory}{\ccat}{\mathbf}

\newtoks\cat@name
\newtoks\cat@upper

\def\cat@generate#1{\begingroup
\countdef\parser@state=0
	\countdef\parser@newstate=1
	\parser@state=0		\cat@name={}\cat@upper={}\cat@parse#1\relax
\edef\next{\unexpanded{\DeclareCategory{#1}}{\the\cat@name}}\next
	\expandafter
\endgroup\the\cat@name}

\def\cat@parse#1{\ifx\relax#1\relax
		\parser@newstate=4
	\else\ifcat\noexpand#1\relax\relax
		\parser@newstate=1
	\else\ifnum`#1>`Z\relax
		\parser@newstate=3
	\else\ifnum`#1<`A\relax
		\parser@newstate=3
	\else
		\parser@newstate=2
	\fi\fi\fi\fi
\ifnum\parser@state=\parser@newstate\else
		\ifcase\parser@state
\or
\ifnum\parser@newstate<4
				\edef\next{\cat@name={\the\cat@name\/}}\next
			\fi
		\or
\edef\next{\cat@name={\the\cat@name\noexpand\mathcal{\the\cat@upper}}}\next
			\cat@upper{}\fi
		\parser@state\parser@newstate
	\fi
\ifnum\parser@state=2
		\edef\next{\cat@upper={\the\cat@upper#1}}\next
	\else
		\edef\next{\cat@name={\the\cat@name\noexpand#1}}\next
	\fi
	\ifnum\parser@state<4	\expandafter\cat@parse\fi
}

\makeatother

\newcommand{\colormembrane}{violet}
\newcommand{\colorgamma}{orange!70!black}

\makeatletter
\DeclareCategory{Web}{\mathcal{W}\mkern-2mu\mathit{eb}}
\DeclareCategory{Foam}{\mathcal{F}\mkern-2mu\mathit{oam}}
\DeclareCategory{Web+}{\cat{Web}_+}
\DeclareCategory{Foam+}{\cat{Foam}_+}

\DeclareCategory{AWeb}{\mathcal A\cat{Web}}
\DeclareCategory{AFoam}{\mathcal A\cat{Foam}}
\DeclareCategory{AFoamB}{\mathcal A\cat{Foam}^\basepoint}
\DeclareCategory{AFoamElB}{\mathcal A\cat{Foam}^{\mathit{\leqslant 2}\basepoint}}
\DeclareCategory{AWeb+}{\cat{AWeb}_+}
\DeclareCategory{AFoam+}{\cat{AFoam}_+}
\DeclareCategory{qAFoam}{\mathcal A\cat{Foam}_q}
\DeclareCategory{qAFoamB}{\mathcal A\cat{Foam}_q^\basepoint}

\DeclareBicategory{FoamEl}{\ccat{Foam}^{\mathit{\leqslant 2}}}
\DeclareCategory{FoamEl}{\cat{Foam}^{\mathit{\leqslant 2}}}
\DeclareCategory{AFoamEl}{\mathcal A\cat{Foam}^{\mathit{\leqslant 2}}}
\DeclareCategory{qAFoamEl}{\mathcal A\cat{Foam}^{\mathit{\leqslant 2}}_q}

\DeclareSet{Soergel}{B}
\DeclareSet{SoergelRed}{\overline{B}}
\DeclareSet{qSoergel}{B_q}
\DeclareSet{qSoergelRed}{\overline{B}_q}
\DeclareSet{decoration}{D}
\def\symm{\mathfrak S}
\def\Sym{\mathit{Sym}}

\def\hTr{\mathrm{hTr}}
\def\qhTr{\mathrm{hTr}_q}

\def\qTwist{\@ifnextchar[\qTwist@{\qTwist@[\qnt]}}
\def\qTwist@[#1]#2{{}_{#1}#2}

\def\HH@symbol#1{\@ifnextchar_{\HH@symbolSub{\mathit{#1}}}{\mathit{#1}}}
\def\HH@symbolSub#1_#2{#1_{\!#2}}

\def\HH{\HH@symbol{HH}}
\def\HHH{\HH@symbol{HHH}}
\def\qHH{\HH@symbol{qHH}}

\def\CFK@hatOrMinus#1{\@ifnextchar-{\mathit{#1}^}{\widehat{\mathit{#1}}}}
\def\CFK{\CFK@hatOrMinus{CFK}}
\def\HFK{\CFK@hatOrMinus{HFK}}
\def\tCFK{\CFK@hatOrMinus{\underline{CFK}}}
\def\tHFK{\CFK@hatOrMinus{\underline{HFK}}}
\def\CFS{\CFK@hatOrMinus{CFS}}
\def\tCFS{\CFK@hatOrMinus{\underline{CFS}}}

\def\OSTwist{P^{\mathrm{OS}}}
\def\trTwist{P^{\mathit{tr}}}

\def\scalars{\Bbbk}
\def\web{\omega}
\def\foam{F}
\def\qnt{q}

\def\LieGL{\mathfrak{gl}}
\def\LieSL{\mathfrak{sl}}
\let\gll\LieGL
\let\sll\LieSL

\def\qvar{\mathsf{q}}
\def\avar{\mathsf{a}}
\def\tvar{\mathsf{t}}

\def\gr{\mathrm{gr}}

\newcommand\quotient[2]{\raisebox{0.3\baselineskip}{$#1$}\!\Big/\!\raisebox{-0.3\baselineskip}{$#2$}}

\def\from{\@ifnextchar^\from@\longleftarrow}
\def\from@^#1{\xleftarrow{\ #1\ }}

\def\bracket#1{\langle #1 \rangle}
\def\Bracket#1{\llbracket #1 \rrbracket}

\def\nth{\@ifnextchar[\nth@{\nth@[th]}}
\def\nth@[#1]#2{$#2$--#1}

\def\torus{\mathbb T}\def\colalpha{\mathbf\alpha}\def\colbeta{\mathbf\beta}\def\markX{\mathbb X}\def\markXX{\mathbb{X\!X}}\def\markO{\mathbb O}\def\xgen{\mathbf x}\def\ygen{\mathbf y}\def\zgen{\mathbf z}

\def\qdeg{\mathrm{qdeg}\mskip 2mu}
\def\hdeg{\mathrm{hdeg}\mskip 2mu}

\def\CFbasis{\mathfrak G}

\def\moduli{\mathcal M}
\def\moduliR{\widehat{\mathcal M}}

\def\nmoduli{\mathcal N}
\def\nmoduliR{\widehat{\mathcal N}}

\DeclareMathOperator{\supp}{supp}

\makeatother

\newcommand{\imagesfolder}{../images}
\newcommand{\NB}[1]{\ensuremath{\vcenter{\hbox{#1}}}}

\def\co{\colon\thinspace}
\tikzset{->-/.style={decoration={markings, mark=at position .5 with {\arrow{>}}},postaction={decorate}}}
\tikzset{-<-/.style={decoration={markings, mark=at position .5 with {\arrow{<}}},postaction={decorate}}}

\newcommand{\circles}{\ensuremath{\mathbb{S}}}

\newcommand{\laurent}{\ensuremath{\mathbb{L}}}

\newcommand{\basepoint}{\ensuremath{{\textcolor{red}\star}}}
\newcommand{\braid}{\ensuremath{\beta}}
\newcommand{\bindex}{\ensuremath{k}}
\newcommand{\nlrel}{\ensuremath{\mathit{NL}}}
\newcommand{\bdiagram}{{\ensuremath{\widehat{\beta}}}}

\newcommand{\setcrossings}{\ensuremath{\mathfrak{X}}}
\newcommand{\resolution}{\ensuremath{I}}

\newcommand{\rdiagram}[1][\resolution]{\bdiagram_{#1}}

\newcommand{\gilmoreA}{\ensuremath{\mathcal{A}}}
\newcommand{\gilmoreAext}{\widetilde{\ensuremath{\mathcal{A}}}}

\newcommand{\nonlocalrelations}{\ensuremath{\mathcal{N}}}

\newcommand{\weight}{\ensuremath{w}}

\newcommand{\ZZ}{\ensuremath{\mathbb{Z}}}
\newcommand{\QQ}{\ensuremath{\mathbb{Q}}}
\newcommand{\NN}{\ensuremath{\mathbb{N}}}
\newcommand{\Ztwo}{\ensuremath{\mathbb{F}}}
\newcommand{\aring}{\ensuremath{\scalars}}
\newcommand{\complexgilmore}{\ensuremath{C^{A\!G}}}
\newcommand{\complexglone}{\ensuremath{C^{\gll_1}}}
\newcommand{\homologyglone}{\ensuremath{H^{\gll_1}}}

\newcommand{\glone}{\ensuremath{\mathcal{S}_1}}
\newcommand{\glinfty}{\ensuremath{\mathcal{S}_\infty}}
\newcommand{\complexglzero}{\ensuremath{C^{\gll_0}}}
\newcommand{\homologyglzero}{\ensuremath{H^{\gll_0}}}

\newcommand{\HHHred}{\ensuremath{\HHH^{\mathrm{red}}}}
\newcommand{\image}{\operatorname{im}}
\newcommand{\glzero}{\ensuremath{\mathcal{S}_0}}
\newcommand{\ourQ}{\ensuremath{\mathit{qAG}}}
\newcommand{\ourQH}{\ensuremath{\mathit{qAGH}}}
\newcommand{\ourQone}{\ensuremath{\mathit{AG}}}
\newcommand{\ourQHone}{\ensuremath{\mathit{AGH}}}
\newcommand{\lra}{\ensuremath{\longrightarrow}}

\newcommand{\bigon}{\ensuremath{\mathrm{cup}}}\newcommand{\nogib}{\ensuremath{\mathrm{cap}}}\newcommand{\zip}{\ensuremath{\mathrm{zip}}}
\newcommand{\piz}{\ensuremath{\mathrm{unzip}}}
\newcommand{\assoc}{\ensuremath{\mathrm{as}}}
\newcommand{\coassoc}{\ensuremath{\mathrm{coas}}}

\newcommand{\shiftvertex}[1]{\ensuremath{\mathrm{gr}(#1)}}

\newcommand{\myunderliney}{\ensuremath{Y}}
\newcommand{\myunderlinez}{\ensuremath{Z}}

\newcommand{\mycrossing}[1][]{\!\!
  \NB{  \begin{tikzpicture}[#1]
      \draw[densely dotted] (0,0) circle (0.5);
    \draw[->] (-135:0.5) --  (45:0.5); 
    \fill[white] (0,0) circle (0.05);
    \draw[->] ( 135:0.5) -- (-45:0.5); 
  \end{tikzpicture}}
}

\newcommand{\mydumbellintro}[1][]{\!\!
  \NB{  \begin{tikzpicture}[#1]
      \draw[densely dotted] (0,0) circle (0.5);
    \draw[>-] (-135:0.5) .. controls (-135:0.3) and +(0,0) ..( 180:0.1);
    \draw[>-] ( 135:0.5) .. controls ( 135:0.3) and +(0,0) ..( 180:0.1);
    \draw[<-] (  45:0.5) .. controls (  45:0.3) and +(0,0) ..(   0:0.1);
    \draw[<-] ( -45:0.5) .. controls ( -45:0.3) and +(0,0) ..(   0:0.1);
    \draw[thick] (180:0.1) -- (0:0.1);
  \end{tikzpicture}}
}

\newcommand{\mysmoothing}[1][]{\!\!
  \NB{
    \begin{tikzpicture}[#1]
      \draw[densely dotted] (0,0) circle (0.5);
    \draw[->] (-135:0.5) .. controls +(0.2,0.2) and +(-0.2, 0.2) ..  (-45:0.5); 
    \draw[->] ( 135:0.5) .. controls +(0.2, -0.2) and +(-0.2, -0.2) ..  (45:0.5); 
  \end{tikzpicture}}
}

\newcommand{\nicedumbell}[4][<-]{
  \begin{scope}[xshift=#3cm, yshift=#4 cm]
    \coordinate(#2LT) at (0,0);
    \coordinate(#2LB) at (0, -0.5);
    \coordinate(#2RT) at (1,0);
    \coordinate(#2RB) at (1,-0.5);
  \draw (#2LT) .. controls +(0.2, 0) and +(0,0) .. ( 0.3, -0.25)
  coordinate[pos=0.5] (#2lt);
  \draw (#2LB) .. controls +(0.2, 0) and +(0,0) .. (0.3, -0.25)  coordinate[pos=0.5] (#2lb);
  \draw (0.3, -0.25) -- (0.7, -0.25) node [pos=0.5, above, font=\tiny] {$2$};
  \draw[#1] (#2RT) .. controls +(-0.2, 0) and +(0,0) .. ( 0.7, -0.25) coordinate[pos=0.5] (#2rt);
  \draw[#1] (#2RB) .. controls +(-0.2, 0) and +(0,0) .. (0.7, -0.25) coordinate[pos=0.5] (#2rb);
\end{scope}}
\newcommand{\nicesmoothing}[4][->-]{
  \begin{scope}[xshift= #3 cm, yshift=#4 cm]
    \coordinate(#2LT) at (0,0);
    \coordinate(#2LB) at (0, -0.5);
    \coordinate(#2RT) at (1,0);
    \coordinate(#2RB) at (1,-0.5);
    \draw[#1] (#2LT) -- (#2RT);
    \draw[#1] (#2LB) -- (#2RB);
\end{scope}}

\def\YDbox{\mathrm{box}}

\def\qshift{\mathbf{q}}	\def\tshift{\mathbf{t}}  

\def\inp{\mathrm{in}}\def\outp{\mathrm{out}}

\def\R{\mathbb R}\def\IxR{[0,1]\times\mathbb R}\def\SxR{\mathbb S^1\times\mathbb R}\makeatletter
\def\S{\@ifnextchar^{\mathbb{S}}{\mathbb{S}^1}}\makeatother
\def\sfce{\Sigma}\def\facet{f}\def\foam{W}

\def\qbinom#1#2{\begin{bmatrix}#1\\#2\end{bmatrix}}
\def\topesym_#1{e_{\mathrm{top}}(X_{#1})}

\def\Res{\mathrm{Res}}

\def\Aut{\mathrm{Aut}}
\def\id{\mathit{id}}

\def\HHSymb{\mathit{HH}}
\newcommand*{\HoHom}{\HHSymb}
\newcommand*{\CHoHom}{\mathit{CH}}

\def\qHHSymb{\mathit{qHH}}
\newcommand*{\qHoHom}{\qHHSymb}
\newcommand*{\qCHoHom}{\mathit{qCH}}

\makeatletter

\def\diagMarkO(#1){\draw[line width=1pt] (#1) circle[radius=0.25]}
\def\diagMarkX(#1){\draw[line width=1pt] (#1) ++(0.2,0.25) -- ++(-0.4,-0.5) ++(0.4,0) -- ++(-0.4,0.5)}
\def\diagMarkXX(#1){\draw[line width=1pt]
	(#1) ++(-1pt,0.25) -- ++(-0.4,-0.5) ++( 0.4,0) -- ++(-0.4,0.5)
	(#1) ++( 1pt,0.25) -- ++( 0.4,-0.5) ++(-0.4,0) -- ++( 0.4,0.5)
}

\let\diagmarkinglevel=0
\def\diagMarkings[#1]#2{\ifnum\diagmarkinglevel=#1\relax
		\fill[color=gray] foreach \x/\y in {#2} { (\x, \y) circle[radius=2pt] };
	\fi
}

\def\tikzOmark(#1){\@ifnextchar[{\tikzOmarkLabel(#1)}{\tikzOmarkNoLabel(#1)}}
\def\tikzOmarkNoLabel(#1){\draw[line width=1pt] (#1) circle[radius=6pt]}
\def\tikzOmarkLabel(#1)[#2]{\draw[line width=1pt] (#1) circle[radius=6pt] node {$\scriptstyle #2$}}
\def\tikzXmark(#1){\draw[line width=1pt] (#1)
	++(-4pt,-5pt) -- ++(8pt,10pt)
	++(-8pt,0pt) -- ++(8pt,-10pt)}
\def\tikzXXmark(#1){\draw[line width=1pt] (#1)
	++(-9pt,-5pt) -- ++(8pt,10pt) ++(2pt, 0pt) -- ++(8pt,-10pt)
	++(-18pt,10pt) -- ++(8pt,-10pt) ++(2pt, 0pt) -- ++(8pt,10pt)}

\makeatother

\def\HClabel#1@#2;{\path (#2:1.75em) ++(0,1pt) -- ++(0,-2pt) node[midway,above,sloped,rotate=#2] {$\scriptstyle #1$};
}

\makeatletter
\def\HochschildChain{\@ifnextchar[{\HochschildChainDraw}{\HochschildChainDraw[]}}\makeatother

\def\HochschildChainDraw[#1]#2#3{\vcenter{\hbox{\begin{tikzpicture}
		\useasboundingbox (-3.5em,-3em) rectangle (3.5em,3em);
		\draw (0,0) circle[radius=1.75em];
\foreach \angle in {#2} do
			\draw (\angle:1.4em) -- (\angle:2.1em);
\ifx\relax#1\relax\else
			\draw[color=\colormembrane,line width=4pt] (#1:1.3em) -- (#1:2.2em);
		\fi
\foreach \label/\angle in {#3} do
			\path (\angle:1.75em) ++(0,1pt) -- ++(0,-2pt)
				node[midway,above,sloped,rotate=\angle] {$\scriptstyle\label$};
	\end{tikzpicture}}}}

\newcommand*\ric{\sb{}\kern-\scriptspace }

\usepackage[divide={1in,*,1in}]{geometry}

\usepackage{amssymb,amsmath,amsthm,amsfonts}

\newtheorem{introtheorem}{Theorem}
\newtheorem{introproposition}[introtheorem]{Proposition}
\newtheorem{introcorollary}[introtheorem]{Corollary}
\newtheorem{introconjecture}{Conjecture}

\newtheorem{theorem}{Theorem}
\newtheorem*{thm}{Theorem}
\numberwithin{theorem}{section}
\newtheorem{lemma}[theorem]{Lemma}

\newtheorem{corollary}[theorem]{Corollary}
\newtheorem{proposition}[theorem]{Proposition}

\newcustomtheorem{customthm}{Theorem}
\newcustomtheorem{customcor}{Corollary}
\newcustomtheorem{customconj}{Conjecture}
\newcustomtheorem{customprop}{Proposition}
\newrepeatedtheorem{reptheorem}[introtheorem]{Theorem}
\newrepeatedtheorem{repcor}[introcorollary]{Corollary}

\theoremstyle{definition}
\newtheorem{definition}[theorem]{Definition}
\newtheorem{example}[theorem]{Example}
\newtheorem{notation}[theorem]{Notation}

\theoremstyle{remark}
\newtheorem{remark}[theorem]{Remark}

\usepackage[hyphens]{url}
\usepackage[colorlinks]{hyperref}
\usepackage[hyphenbreaks]{breakurl}
\usepackage[all]{hypcap}

\definecolor{internalLink}{rgb}{0,0,0.5}
\definecolor{citeLink}{rgb}{0,0.5,0}
\definecolor{urlLink}{rgb}{0,0.5,0.5}

\hypersetup{linkcolor=internalLink,
	citecolor=citeLink,
	urlcolor=urlLink,
	breaklinks=true
}

\allowdisplaybreaks

\tikzset{thickweb/.style={
		line width=1pt
	},
	hd/alpha/.style={
		line width=0.5pt,
		dashed,
		blue!50!black
	},
	hd/alpha/back/.style={
		style=hd/alpha,
		dotted
	},
	hd/beta/.style={
		line width=0.5pt,
		solid,
		red!70!black
	},
	hd/surface/.style={
		line width=0.3pt,
		black!60
	}
} 
\renewcommand{\imagesfolder}{images}

\begin{document}

\title[A proof of Dunfield--Gukov--Rasmussen Conjecture]{On the Dunfield--Gukov--Rasmussen Conjecture}

\author{Anna Beliakova}
\address{Universit\"at Z\"urich,
	Institute of Mathematics,
	Winterthurerstrasse 190,
	CH-8057 Z\"urich,
	Switzerland}
\email{\href{mailto:anna@math.uzh.ch}{anna@math.uzh.ch}}

\author{Krzysztof K.\ Putyra}
\address{Universit\"at Z\"urich,
	Institute of Mathematics,
	Winterthurerstrasse 190,
	CH-8057 Z\"urich,
	Switzerland}
\email{\href{mailto:krzysztof.putyra@math.uzh.ch}{krzysztof.putyra@math.uzh.ch}}

\author{Louis-Hadrien Robert}
\address{Universit\'e Clermont Auvergne, LMBP, Campus des C\'ezeaux,
  3 place Vasarely, CS 60026, 63178 Aubi\`ere cedex, France}
\email{\href{mailto:louis-hadrien.robert@uca.fr}{louis-hadrien.robert@uca.fr}}

\author{Emmanuel Wagner}
\address{Universit\'e de Paris,
	IMJ-PRG, UMR 7586 CNRS,
	8 Place Aur\'elie Nemours,
	F-75013, Paris,
	France}
\email{\href{mailto:emmanuel.wagner@imj-prg.fr}{emmanuel.wagner@imj-prg.fr}}

\date\today

\begin{abstract}

In 2005 Dunfield, Gukov and Rasmussen conjectured an existence of a differential from the reduced triply graded Khovanov--Rozansky homology of a
knot to its knot Floer homology defined by Ozsv\'ath and Szab\'o. The main
result of this paper is a proof of a suitably updated version of their conjecture: we show that the reduced triply graded homology is related to knot
Floer homology by two spectral sequences, going through the intermediate
~$\mathfrak{gl}_0$ homology constructed by the last two authors. The~$\mathfrak{gl}_0$ homology comes
equipped with a spectral sequence from the reduced triply graded homology,
and here we construct the other spectral sequence, from the~$\mathfrak{gl}_0$ homology to
knot Floer homology. The new spectral sequence is of Bockstein type
 and arises from a~subtle manipulation
	of coefficients.
	The~main tools are quantum traces of foams and of singular Soergel bimodules
	and a~$\ZZ$-valued cube of resolutions model for knot Floer homology,
	originally constructed by Ozsv\'ath--Szab\'o over the~field of two elements.
	As an~application we deduce that both the~$\mathfrak{gl}_0$ homology
	and the~reduced triply graded Khovanov--Rozansky homology detect the~unknot,
	the~two trefoils, the~figure eight knot and the~cinquefoil.
\end{abstract}

\maketitle
 
\tableofcontents

\section{Introduction}
\label{sec:intro}
The~discovery of the~Alexander polynomial $\Delta_K(\qvar)$ in 1929
marked the~birth of knot theory,
mani\-fes\-ted in the~transition from conjectures to proofs.
In the~1970s Conway found a~first diagrammatic algorithm to compute
$\Delta_K(\qvar)$  using the~so-called skein relation:
\begin{equation}\label{eq:skein-alex}
	\mycrossing \quad - \quad \mycrossing[yscale=-1]
		\quad=\quad
	(\qvar - \qvar^{-1})\,\mysmoothing\,,
\end{equation}
where the~three pictures represent link diagrams that coincide
outside of the~small regions depicted above.

In the~1980s the~second big player in knot theory was introduced by
Jones and later extended to the~two variable HOMFLY-PT polynomial
$P_K(\avar, \qvar)$ with the~skein relation
\begin{equation}\label{eq:skein-HOMFLYPT}
  \avar\, \mycrossing
  	\quad-\quad
  \avar^{-1}\, \mycrossing[yscale=-1]
  	\quad=\quad
  (\qvar - \qvar^{-1})\, \mysmoothing\,.
\end{equation}
It specializes to the~Alexander polynomial for $\avar=1$
and to the~Jones polynomial for $\avar=\qvar^2$.
Setting $\avar=\qvar^N$ recovers the~$\sll_N$ polynomial of the~knot $K$.
Introducing \emph{webs}, oriented  planar trivalent graphs,
we can rewrite \eqref{eq:skein-HOMFLYPT} as

\begin{equation}\label{SSR}
  \avar\,  \mycrossing = \mydumbellintro - \qvar^{-1}\, \mysmoothing\qquad \quad
    \text{and} \qquad \quad
  \avar^{-1}\,  \mycrossing[yscale=-1] = \mydumbellintro - \qvar \,\mysmoothing \,,
\end{equation}

\noindent
where the~second diagram in both equations represents a~\emph{singular} crossing.
 
At the~beginning of this century Jones and HOMFLY--PT polynomials were
moved one categorical level higher by Khovanov and Khovanov--Rozansky
\cite{MR1740682, MR2391017, MR2421131}. These new theories associate with a~link
diagram graded chain complexes, the~homologies of which yield new powerful
link invariants. The~polynomials can be reconstructed by taking the~graded
Euler characteristics of these chain complexes. One important feature of the categorified invariants is their functoriality with respect to 
  link cobordisms, i.e. surfaces bounded by links induce maps on homology.
This does not hold for the~Euler characteristics.

After presenting a~knot $K$ as a~closure of a~braid $\braid$ with $n$ crossings,
the~Khovanov--Rozansky chain complex is defined by resolving each crossing of
$\braid$ in two ways as suggested by \eqref{SSR} and then by assigning to each
such resolution, which is a~web, a~Soergel bimodule.
Webs are then organized as vertices of an~$n$-dimensional cube.
With the~edges of the~cube we associate differentials  given by bimodule maps
induced by singular 2-dimensional cobordisms called \emph{foams}.
This construction is  based on a~functor of bicategories
\begin{equation}\label{eq:bifunctor}
  \set{Soergel}\co\ccat{Foam} \to \ccat{sSBim}
\end{equation}
discussed in Section~\ref{sec:ap}.
Closing up the~braid is achieved by taking the~horizontal trace of $\set{Soergel}$,
which  assigns to a~closed web the~Hochschild homology of the~associated
Soergel bimodule.
The~homology $\HHH$\/ of the~resulting cube complex is a~triply graded link
invariant that categorifies $P_K(\avar,\qvar)$.
By putting a~basepoint on the~diagram and killing the~corresponding
variable in the~Soergel bimodule, we obtain the~so-called
\emph{reduced} homology $\HHH^{\mathrm{red}}$,
which in case of knots does not depend on the~position of the~basepoint.
The~\emph{quantum} horizontal trace of \eqref{eq:bifunctor},
the~\emph{quantum} Hochschild homology as well as
the~\emph{quantum} annular link homology were constructed in \cite{BPW}.
Note that the~quantum horizontal trace of $\ccat{Foam}$ is the~category of
\emph{quantum annular foams} constituted by annular foams together with a~membrane,
subject to additional relations involving the~membrane and a~quantum parameter $q$.
All mentioned constructions admit algorithmic computations.

Parallel to these developments, the~Alexander polynomial was
categorified by Ozsv\'{a}th and Szab\'{o} using completely different,
geometric techniques.
Here chain complexes are generated by Lagrangian intersections
in a~symmetric product of (pointed) Heegaard diagrams
and the~differential counts holomorphic discs.
The resulting homology, known as \emph{Heegaard Floer knot} homology
or simply knot Floer homology, is denoted by $\HFK$.
This homology has important topological applications:
it detects  the~genus and fiberedness of a~knot \cite{YiNi}.
The list of knots detected by $\HFK$ is constantly growing.
However, this theory is essentially non-local and hard to compute in general. 

During many years the relationship between the algebraic
Khovanov--Rozansky and the geometric Heegaard Floer knot homologies remained mysterious.
In 2005
Dunfield, Gukov and Rasmussen made a series of influential conjectures \cite{DGR}.
They predicted an existence of the differentials $\partial_N$ and $\partial_0$ on $\HHH$,  such that
$H_{\!\bullet} (\HHH, \partial_N)$ is the $\mathfrak{sl}_N$
 homology
and $H_{\!\bullet} (\HHH, \partial_0)$ is $\HFK$, lifting the specializations of
$P_K(\avar, \qvar)$ at $\avar=q^N$ and $\avar=1$, respectively. 
A version of this conjecture was
established by Rasmussen in \cite{MR3447099}
for $N\geq 1$,
who
 constructed  a 
differential $d_N$ on $\HHH$ that leads to a
spectral sequence to $\mathfrak{sl}_N$ homology. It is  an open problem whether this spectral sequence collapses immediately.

In this paper we prove a suitable version of the second part of
Dunfield--Gukov--Rasmussen conjecture by connecting algebraic $\HHH^{\mathrm{red}}$ and geometric $\HFK$ knot homologies
with two spectral sequences. 
We will  refer to this result as DGR Conjecture.
An~important ingredient of our proof is
a~new knot homology theory 
 $\homologyglzero$, that was constructed by
the~last two authors of the~present paper in \cite{RW3}.
The $\gll_0$ homology categorifies the~Alexander polynomial and 
comes equipped with a~spectral sequence from
$\HHH^{\mathrm{red}}$ over $\mathbb Q$.
In this paper we construct the second spectral sequence
starting with $\homologyglzero$ and converging to   $\HFK$.
Let us stress that even if $\homologyglzero$ and $\HFK$
categorify the~same polynomial invariant,
there is a~priori no reason to assume that these homologies are isomorphic.
In fact, we will prove here that $\homologyglzero$ and $\HFK$ are different.

Our  construction of a~spectral sequence from $\homologyglzero$ to $\HFK$ 
uses the~cube of resolutions model for $\HFK$ developed
by Ozsv\'{a}th and Szab\'{o}
in \cite{OSCube} and 
reformulated in algebraic terms by Gilmore \cite{Gilmore}.
This model requires twisted 
 coefficients  $\Ztwo[q^{-1},q]]$, which is
a~ring of power series in $q$ over  the~field of two elements $\Ztwo$.
 This ring of coefficients is needed to kill  higher
differentials in the cube.
The problem is however that $q-1$ is invertible in $\Ztwo[q^{-1},q]]$, and hence,
this coefficient ring does not admit a~specialization at $q=1$
that is needed to connect with Soergel bimodules and
 $\homologyglzero$.  This motivated Manolescu 
in  \cite{ManolescuCube} to untwist the~cube of resolutions model
for $\HFK$.
As a result, he was able
  to reduce DGR conjecture to a computation of certain Tor-groups.

 In this paper we choose a different approach and work with twisted coefficients.
In  \cite{Gilmore}
the state spaces associated with vertices of the cube were 
identified with quotients of a twisted
version of Soergel bimodules.
This motivated us to search for a more conceptual understanding
of these spaces
in terms of the {\it quantum} Hochschild homology of Soergel bimodules.  

Let us finally 
mention that the cube of resolutions model was
 extensively studied by Dowlin in  two unpublished papers
 \cite{DowlinHHH, DowlinKh},
where   the~cube construction
was assumed to work with integral coefficients.

\subsection{Main results}
We start by promoting the coefficients from $\Ztwo$ to $\ZZ$, meaning that we
 provide an algebraic model
that computes $\HFK$ over $\ZZ[q^{-1},q]]$ via the~cube of resolutions.
For this we represent a~knot $K$ as a~braid closure $\bdiagram$ and
associate with it a~complex $\complexgilmore(\bdiagram)$ of $\ZZ[q,q^{-1}]$-modules,
such that
\begin{equation}\label{OSG}
	H_\bullet (\complexgilmore(\bdiagram)\otimes \ZZ[q^{-1},q]])
		\cong
	\HFK(K)\otimes \ZZ[q^{-1},q]].
\end{equation}
The first complex satisfying \eqref{OSG} with $\Ztwo$ instead of $\ZZ$
was constructed by {\bf A}lison {\bf G}ilmore in \cite{Gilmore}.
With vertices of the~cube she associated quotients of polynomial
rings by local and non-local relations.

On the geometric side, the proof of \eqref{OSG} requires a~choice
of a~coherent system of orientations for moduli spaces of holomorphic disks 
associated with  a~multipointed Heegaard diagram and to track signs in all proofs of \cite{OSCube}.
Let us mention that even though the~results of \cite{OSSSingular} and \cite{OSCube}
were stated over $\Ztwo$, these papers contain the~expected signs.
This makes the~comparison easier.
The~main ingredients of our proof are the result of \cite{AliEft}
about the~existence and properties of coherent systems of orientations,
Theorem \ref{thm:CFK-cube} establishing the skein exact triangle over $\ZZ$,
as well as Proposition \ref{prop:hfk-for-connected-planar} 
and Theorem \ref{thm:A-vs-HFK-for-planar} 
that allow to identify the~algebras assigned to vertices of the~cube.
In addition, we improve and streamline many arguments in \cite{OSCube}. For example, we consistently work with admissible diagrams, however
link diagrams used in \cite{OSCube} are not admissible, and we conceptualize the construction of twisted complexes.

On the algebraic side, we extend Gilmore's construction
to all annular webs and interpret it in terms of Soergel bimodules.
Here we work over an~arbitrary commutative ring $\scalars$ with a~fixed
inver-tible element $q$.
The~space $\gilmoreA(\web)$ that we assign to an~annular web $\web$
is a~quotient of the~quantum Hochschild homology \cite{BPW}
of the~Soergel bimodule associated with the~web
by (renormalized) non-local relations.
In the~case of a~resolution of a~braid 
this quotient is identified with Gilmore's algebra 
after renormalizing her variables:
we check that Gilmore's local relations
coincide with the~Soergel relations,
whereas non-local relations with those defined for webs. 

Combining previous results with the~general theory of quantum traces
\cite{BPW} we obtain a~new conceptual interpretation of
Gilmore's construction, opening the~floor for its further generalizations.
Recall that similarly to the~non-quantized setting,
the~quantum horizontal trace induces a~functor from quantum annular foams
to the~quantum Hochschild homology of Soergel bimodules.
In Proposition \ref{prop:gilmore-functor} we prove that
the~non-local relations are preserved by this functor.
Hence, we obtain a~new functorial evaluation of quantum annular foams
by using the \emph{quotient} of the~quantum Hochschild homology
by the~non-local relations.
This quotient can be used to generalize Gilmore's construction
and to define new homology theories. 

In particular, we modify $\gilmoreA(\web)$ by killing $q$-torsion.
Namely, given a~web $\web$ we consider the~map
\begin{equation}
	Q_{\web}\colon \gilmoreA(\web, \ZZ[q^{-1},q])
		\longrightarrow
	\gilmoreA(\web; \ZZ[q^{-1}, q]])
\end{equation}
induced by the~inclusion of coefficient rings.
In general the~map $Q_\web$ is not injective.
Dividing the~previous construction by the~kernel of $Q_\web$
and tensoring it with $\scalars$ over $\ZZ[q^{-1},q]$
produces a~new functorial assignment of a~$\scalars$-algebra $\ourQ(w)$
to a~quantum annular web $\web$.
By inserting these algebras into a~cube of resolutions for a~knot $K=\widehat{\beta}$
we obtain our main player---a~new chain complex\footnote{
	The~name of the~new complex is motivated by the~fact that it interpolates
	the~\textbf{A}lgebraic categorification of the~last two authors
	and the~\textbf{G}eometric categorification of Ozsv\'ath and Szab\'o.
}
$\ourQ(\bdiagram)$, the~homology of which we denote by $\ourQH(\bdiagram)$.
Since $\ourQ$ is defined over $\ZZ[q^{-1},q]$, it can be specialized at $q=1$,
 resolving our main problem.
We denote these specializations at $q=1$ by
$\ourQone(\bdiagram)$ and $\ourQHone(\bdiagram)$ respectively.
As we shall see, this new chain complex interpolates between the~algebraic
and geometric settings previously discussed in the~following way.

\begin{introproposition}\label{prop:ourQ-vs-gl0}
	The~homology theories $\ourQHone$ and $\homologyglzero$ coincide.
	Hence, $\ourQHone$ is a~knot invariant if\/ $\scalars$ is a~field.
\end{introproposition}

We expect the following to be true.

\begin{introconjecture}\label{conj:AGH-is-invariant}
	If\/ $\scalars$ is a~field of characteristic $0$,
	then $\ourQH$ is a~knot invariant for any $\qnt$.
\end{introconjecture}

In the~next step we analyse the~Bockstein spectral sequence
associated with the~specialization of $\ourQH$ at $q=1$.
Note that this spectral sequence preserves the~Alexander grading.
To be more precise, we fix an arbitrary field $\mathbb K$
and work over the principal ideal domain $\mathbb K[q,q^{-1}]$.
Thanks to Proposition~\ref{prop:ourQ-vs-gl0} we can identify
the~first page of our spectral sequence with $\homologyglzero$
and by \eqref{OSG} the~last page,
which is the~quotient of $\ourQH\otimes \mathbb K$ by its torsion submodule,
is isomorphic to $\HFK$
(compare this with Proposition \ref{prop:BocksteinConvergesq1}).

\begin{introtheorem}\label{thm:main}
  Assume that $\mathbb K$ is a~field and $K$ is a~knot represented
  by a~braid closure $\bdiagram$.
  Then the~$(q\mapsto 1)$ Bockstein spectral sequence applied to
  $\ourQ(\bdiagram; \mathbb{K}[q, q^{-1}])$ has
  $\homologyglzero(K;\mathbb K)$ as its first page
  and converges after finitely many steps.
  The last page is (non-canonically) isomorphic to $\HFK(K; \mathbb K)$.
\end{introtheorem}

Recall that in \cite{RW3} a~spectral sequence from
$\HHH^{\mathrm{red}}$ to $\homologyglzero$ was constructed over $\QQ$.

\begin{thm}[\cite{RW3}]
  There exists a differential $d_ 0$ of $(\avar,\qvar,\tvar)$-degree $(-2,0,1)$
  on the~Hochschild homo\-logy of reduced Soergel bimodules over any field 
  of characteristic $0$
  that induces a spectral sequence from $\HHH^{\mathrm{red}}$ to $\homologyglzero$.
\end{thm}

The~above theorem uses the~following convention for gradings:
the~Koszul differential $d_{K}$ is of $(\avar,\qvar,\tvar)$-degree
$(-2,2,1)$, whereas the~degree of the~hypercube differential
$d_{top}$ is $(0,0,1)$.

Combining this spectral sequence with the one from
Theorem \ref{thm:main}
  we get the main result of this paper.

\begin{introtheorem}[DGR Conjecture]\label{thm:HHHHFK}
	For any knot $K$ and any field of characteristic zero, the bigraded dimension of  
	$\HHH^{\mathrm{red}}$ (after forgetting the $\avar$-grading) is greater or equal to the bigraded  dimension of $\HFK$.
\end{introtheorem}

The conjectural degree of the differential $\partial_0$
 of Dunfield--Gukov--Rasmussen
was expected to be  $(\avar,\qvar,\tvar)=(2,0,1)$ and $(2k,0,1)$ 
for $k>1$
in general if one expects a full spectral sequence.
This degree is precisely $2\mathrm{deg} (d_{top})-\mathrm{deg} (d_0)$, which is
the~degree of the~differential on the~complex of
the~filtered spectral sequence formed by computing $d_0$ and then $d_{top}$
(after homology with respect to $d_K$  and ascending filtration associated to the $\avar$-grading are taken).
Notice that the combination of the two~spectral sequences guarantees that
higher differentials are always of degree $1$ with respect to the $\tvar$-grading, hence are compatible with the conjectural degrees and establish the bigraded rank inequality.

To investigate the~question whether our spectral sequence 
collapses at the first step we compute the~homology
$\homologyglzero$.
Over $\QQ$ this question can be handled using the known
computations for $\HHH^{\mathrm{red}}$ and the~spectral sequence
between $\HHH^{\mathrm{red}}$ and $\homologyglzero$.

Consider the~first case of interest, namely the~$T(3,4)$-torus knot.
The Poincar\'e polynomial of the~reduced triply graded link homology
of this knot is
\[\begin{split}
	P(\tvar, \avar,\qvar) = 
		\avar^{-6}\qvar^{-6}\tvar^{6} +
		+  \avar^{-8}\qvar^{-4}\tvar^{5} + \avar^{-6}\qvar^{-2} \tvar^{4} + (\avar^{-8}\qvar^{-2} + \avar^{-8}\qvar^{0})\tvar^{3}\\
		  + 
		(\avar^{-6}\qvar^{0} + \avar^{-6}\qvar^{2}+ \avar^{-10}\qvar^{0})\tvar^{2} 
		 + (\avar^{-8}\qvar^{2} + \avar^{-8}\qvar^{4})\tvar^{1}+ \avar^{-6}\qvar^{6}\tvar^{0}.\end{split}
\]
On one hand, a~direct investigation using the~degree of
the~differential $d_0$ shows that the~total dimension of
the~$\homologyglzero$ \cite{RW3} is at least $9$:
the only terms that can cancel out are
$\avar^{-8}\qvar^{0}\tvar^{3}$ and $\avar^{-6}\qvar^{0}\tvar^{2}$.
On the other hand, the~total dimension of $\HFK$ for the same knot is $5$,
with three pairs that should cancel out:
\[	\avar^{-10}\qvar^{0}\tvar^{2} \leftrightarrow \avar^{-8}\qvar^{0}\tvar^{3},
		\qquad
	\avar^{-8}\qvar^{2}\tvar^{1} \leftrightarrow \avar^{-6}\qvar^{2}\tvar^{2},
		\quad\text{and}\quad
	\avar^{-8}\qvar^{-2}\tvar^{3} \leftrightarrow \avar^{-6}\qvar^{-2}\tvar^{4}.
\]
A~direct consequence is that $\homologyglzero$ and $\HFK$
do not coincide over $\QQ$.
Hence, the~combination of our two spectral sequences does not always degenerate.

In the~example of $T(3,4)$, the~term $\avar^{-8}\qvar^{0}\tvar^{3}$ cancels out
with either $\avar^{-10}\qvar^{0}\tvar^{2}$ or $\avar^{-6}\qvar^{0}\tvar^{2}$.
It is a priori unclear, though, with which one and in which of the~two spectral
sequences this cancellation happens. However recent computations by Laura Marino (adapting her program \cite{marino2023computing}) show that $\HHH^{\mathrm{red}}(T(3,4))$ and $\homologyglzero(T(3,4))$  have the same rank, hence the cancellation occurs in the Bockstein type spectral sequence.

\subsection{Applications}
As it was already noticed by Manolescu~\cite{MR3392529},
the~existence of a~spectral sequence from the~reduced
triply graded homology $\HHHred$ to knot Floer homology $\HFK$
yields a~couple of detection results for $\HHHred$.
In this section we establish these detection results
for both $\HHHred$ and $\homologyglzero$ with $\QQ$-coefficients.

Recall that in \cite{MR2023281} it was proven that $\HFK$ detects
the~Seifert genus of the~knot and hence the~unknot,
which is the~only knot of genus $0$.
The~statement was originally formulated for $\scalars = \ZZ$.
However, as noticed by many authors (see for instance
\cite{MR3868231, https://doi.org/10.48550/arxiv.2011.02005}),
it remains true over any field.
Theorems \ref{thm:HHHHFK} and \ref{thm:main} allow us to extend this
result to $\HHH^{\rm red}$ and $\homologyglzero$.

\begin{introcorollary}\label{cor:unknot-detect}
  The~$\gll_0$ homology and the~reduced triply graded homology
  (both with $\QQ$-coefficients) detect the~unknot.
\end{introcorollary}

If Conjecture 1 holds, then the~same is true for $\ourQH$ at any $\qnt$.

\begin{proof}
  Suppose $K$ has the~same $\gll_0$ homology as the~unknot,
  that is $\homologyglzero(K) = \QQ$ concentrated in bidegree $(0,0)$.
  Nothing can happen in the~spectral sequence from $\homologyglzero$ to $\HFK$,
  so that $\HFK(K, \QQ) = \QQ$ in bidegree $(0,0)$ and $K$ is therefore the~unknot.
  The~argument for $\HHHred$ is the same.
\end{proof}

By the~results of Ghiggini \cite{MR2450204} and  Ni \cite{YiNi},
$\HFK$ detects fiberedness of the knot and
the~only fibered knots of genus $1$ are the~two trefoils and the~figure-eight knot.
The~knot Floer homologies of these knots appear to be different.
Hence, $\HFK$ detects these knots too.
In addition, it was proven more recently that knot Floer homology detects
the~cinquefoil \cite{https://doi.org/10.48550/arxiv.2203.01402}. 
Applying Theorems \ref{thm:HHHHFK} and \ref{thm:main}
we deduce that all these knots  are also detected by $\HHHred$ and $\homologyglzero$.
Note that $\gll_0$ and reduced triply graded homology groups for these knots
are given in Appendix~\ref{sec:computations-gl0}.

\begin{introcorollary}\label{corE}
  The~$\gll_0$ homology and the~reduced triply graded homology (both
  with $\QQ$-coefficients) detect the~two trefoils, the~figure-eight knot
  and the~cinquefoil.
\end{introcorollary}

 Besides of telling knots apart, Corollary \ref{corE}
demonstrates strength of 
algebraic knot homology theories.
\begin{proof}
  The~argument is exactly the~same for both homology theories and all knots:
  by degree reasons there are no cancellations in the~spectral sequences.
  Let us provide a~detailed proof for $\gll_0$ homology.
  Suppose that the~Poincar\'e polynomial of the~$\gll_0$ homology of a~knot
  is equal to the~one of the~cinquefoil,
  $\qvar^4 + \qvar^2\tvar^{1} + \tvar^{2} + \qvar^{-2}\tvar^{3} + \qvar^{-4}\tvar^{4}$.
  Because the~$(q\mapsto 1)$ Bockstein spectral sequence preserves the~$\qvar$-degree,
  this knot has the~same knot Floer homology as the~cinquefoil.
  Hence, by \cite{https://doi.org/10.48550/arxiv.2203.01402}, it is the cinquefoil.
  Exactly the~same argument works for both trefoils.
  In case of the~figure-eight knot, the~homology
  in the~$(\qvar, \tvar)$-degree $(0,0)$ has rank 3.
  However, since the~differential in the~$(q\mapsto 1)$
  Bockstein spectral sequence has $(\qvar, \tvar)$-degree $(0,1)$,
  there is again no cancellation. 
  To conclude the~result for $\HHH^{\rm red}$ we check that
  the~Poincar\'e polynomials listed in Appendix~\ref{sec:computations-gl0}
  have no cancelling pairs with respect to differentials of
  $(\avar,\qvar,\tvar)$-degree $(2k,0,1)$.
\end{proof}

Let us mention that similar detection results are also valid for Khovanov homology 
\cite{MR2805599, MR4393789, https://doi.org/10.48550/arxiv.2105.12102, MR4275096},
where the last two papers are based on the~Dowlin spectral sequence \cite{DowlinKh}
that uses an~integral version of \cite{OSCube}. 
Combined with our $\ZZ$-valued cube of resolutions model for $\HFK$,
these results are fully justified.

The above mentioned detection results for $\HHH^{\mathrm{red}}$
 can also be obtained by combining  Dowlin's and  Rasmussen's \cite{MR3447099} spectral sequences. Notice that in this case, one should work with the single $\delta$-graded (rather than bigraded)
 versions of Khovanov and knot Floer homologies.
Recently, in \cite{MR4393789} Baldwin and Sivek again used the~Dowlin spectral
sequence to prove that $\HHH^{\rm red}$ detects an~infinite family of pretzel
knots $P(-3,3,2n+1)$, $n\in \ZZ$, that are all not fibered.
We would expect that our methods also allow to reprove these results.

 \subsection{Outline}

Besides Introduction, this~paper is divided in four sections.
The~first section is devoted to algebraic preliminaries:
we recall classical facts and introduce notations concerning
symmetric polynomials, Soergel bimodules and quantum Hochschild homology.
Then we discuss webs and foams and finally we apply the~technology
of quantum traces \cite{BPW} to webs and foams. We define the functor $B$
from \eqref{eq:bifunctor} and  discuss its quantum horizontal trace.
Furthermore, we show that the higher quantum Hochschild homology of a singular
Soergel bimodule for generic quantum parameters vanishes (Theorem~\ref{thm:qHH(Soergel)=0}).
In Section~\ref{sec:lh} we review 
the general cube of resolutions construction and
three different combinatorial  link homologies based on it:
\begin{enumerate}
	\item the~triply graded Khovanov--Rozansky link homology,
	\item the~symmetric $\gll_1$ link homology based on the~Robert--Wagner foam evaluation formula,
	\item the $\gll_0$ homology introduced by the~two last authors in \cite{RW3}.
\end{enumerate}
In Section~\ref{sec:hfk} we review the~construction of the~twisted Heegaard Floer
homology for singular links and systematically upgrade coefficients to $\ZZ$.
Then we construct skein exact triangles and compute homology over $\ZZ$
for planar singular knots.
Section \ref{sec:main} is the~heart of the~paper:
here we finish the~construction of the~cube of resolutions over $\ZZ$,
define the~algebra $\gilmoreA(\web)$ for any annular web,
prove \eqref{OSG} and finally introduce the~complex $\ourQH$,
the~homology of which interpolates between $\gll_0$ homology and knot Floer homology.
By applying the~Bockstein spectral sequence we prove main results of this paper
---   Theorems~\ref{thm:main} and \ref{thm:HHHHFK}.
Finally, Appendix~\ref{sec:bockstein} provides a~self-contained account
on spectral sequences,
Appendix~\ref{sec:qHH-is-cyclic} contains a~technical lemma about
quantum Hochschild homology
and Appendix \ref{sec:computations-gl0} includes computation
of $\homologyglzero$ needed for the~detection results.
 \subsection{Acknowledgements}
The authors would like to thank John Baldwin, Paolo Ghiggini and Ciprian Manolescu
for helpful conversations, as well as Peter Ozsv\'{a}th and Zoltan Szab\'{o}
for the~clarification of details in their cube construction. They are
also indebted to the anonymous referees who helped us to improve the
exposition with their careful readings and valuable comments.
AB and EW are grateful to the organizers of the Budapest and Annular
meetings of the Simons Collaboration ``New Structures in
Low-Dimensional Topology'' in 2023/24 for creating a stimulating working environments.
AB is also supported by the above mentioned Simons Collaboration and
the Swiss National Science Foundation grant 200020\textunderscore{}207374.
AB and KP are supported by NCCR SwissMAP of the~Swiss National Science Foundation.
LHR was supported by the Luxembourg National Research Fund
PRIDE17/\sep 1224660/\sep GPS. EW is partially supported by the ANR projects AlMaRe
(ANR-19-CE40-0001-01), AHA (JCJC ANR-18-CE40-0001) and CHARMES (ANR-19-CE40-0017).

\section{Preliminaries}
\label{sec:ap}
\subsection{Conventions}

In this paper we work over a~fixed commutative unital ring $\scalars$
with no further restrictions and we pick an~invertible $q\in\scalars$.
An~unadorned tensor product means a~tensor product over $\scalars$.

The~bold letter $\qshift$ is used for a~shift functor in a~graded category.
In particular, $\qshift^d M$ is a~graded module $M$ shifted upwards by $d$,
so that $(\qshift^d M)_i = M_{i-d}$.
More generally, if $p(q)= \sum_{i\in Z} a_i q^i$ is a~Laurent
polynomial in $q$ with positive integral coefficients, then
\[
	p(\qshift) M := \bigoplus_{i} \qshift^i M^{\oplus a_i}
\]
We often use \emph{quantum integers}, \emph{quantum factorials},
and \emph{quantum binomials}, defines respectively as
\[
	[k] = \frac{q^{k} - q^{-k}}{q-q^{-1}},
\qquad
	[k]! = \prod_{i=1}^{k}[i]
\quad\text{and}\quad
	\qbinom{n}{k} = \dfrac{[n]!}{[k]![n-k]!}
\]
for any integers $0 \leqslant k \leqslant n$.

Complexes have differentials of degree $+1$,
with the~only exception of the~Hochschild homology (Section~\ref{sec:hochschild}).
We use $\tshift$ for the~standard homological shift,
so that $(\tshift^a C)_i = C_{i+a}$ and the~\emph{mapping cone complex}
$C(f)$ of a~chain map $f\colon C\to D$ is modeled on $\tshift C \oplus D$.

Finally, braids and webs are drawn and read from left to right,
whereas foams are drawn and read from bottom to top.
Other notation used through the~paper:
\begin{itemize}
\item $\beta$ is a braid (diagram) and	$\bdiagram$ is its braid closure;
	\item 
	$\setcrossings$ denotes the set of crossings in the diagram;
	\item $n_+, n_-, n_\times$ are the numbers of positive, negative and singular crossings, respectively.
\end{itemize}
 \subsection{Symmetric polynomials and Soergel bimodules}
In this section we summarize some useful facts about symmetric polynomials
and Soergel bimodules. We refer to \cite{MR3443860} and
\cite{EliasMakisumiThielWilliamson} for a~detailed account.

\begin{notation}
	The number of boxes of a~given Young diagram $\lambda$ is denoted by $|\lambda|$. 
Given two non-negative integers $a$ and $b$, we write $T(a,b)$ for the~set of Young diagrams with at most $a$ columns
	and at most $b$ rows.
	The~maximal diagram, a~rectangle of width $a$ and height $b$,
	is hereafter denoted by $\YDbox(a,b)$.
	Given a~Young diagram $\lambda \in T(a,b)$ we construct its
	\begin{itemize}
		\item \emph{complement} $\lambda^c \in T(a,b)$ by rotating by 180 degrees
		the~set of boxes from $\YDbox(a,b)$ that are not in $\lambda$,
		\item \emph{transpose} $\lambda^t \in T(b,a)$ by exchanging rows with columns in $\lambda$,
		\item \emph{dual} $\widehat\lambda \in T(b,a)$ as the~diagram
		$(\lambda^t)^c = (\lambda^c)^t$.
    \end{itemize}
	\begin{figure}[ht]
		\centering
		\NB{\tikz[scale= 0.7]{\input{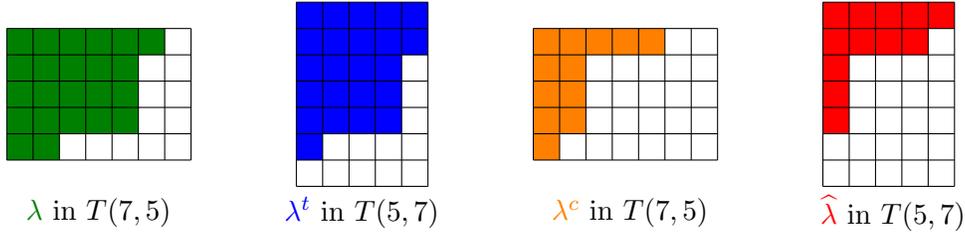}}}
		\caption{Pictorial definition of $\lambda^c$, $\lambda^t$ and $\widehat{\lambda}$.}
		\label{fig:yd}
	\end{figure}
\end{notation}

Fix a~positive number $N > 0$ and recall that $\scalars$ is a~fixed commutative
unital ring. Consider the~polynomial ring $R := \scalars[x_1,\dots,x_N]$
with an~action of the~symmetric group $\symm_N$ that permutes
the~variables. Endow $R$ with a grading by declaring that all $x_i$
are homogeneous of degree $2$.
It is a~standard fact that the~ring of invariant polynomials
\[
	\mathit{Sym}_N := R^{\symm_N}
\]
is freely generated by elementary symmetric functions
\[
	e_k(x_1,\dots,x_N) = \sum_{i_1 < \dots < i_k} x_{i_1} \cdots x_{i_k}
\]
for $k=1,\dots,N\ric$.
A~linear basis of $\mathit{Sym}_N$ is given by \emph{Schur polynomials}
$s_\lambda$ parametrized by Young diagrams $\lambda$ with at most $N$ rows.
They satisfy
\[
	s_\lambda s_\mu = \sum_{\nu} c^\nu_{\lambda\mu} s_\nu
\]
where $c^\nu_{\lambda,\mu} \in \mathbb N$, the~{Littlewood--Richardson} coefficients,
are independent of $N$.
Because $c^\nu_{\lambda\mu}=0$ unless $|\lambda| +|\mu|  = |\nu|$, the~above sum is finite.

\begin{proposition}
	Let $X\ric$, $Y\ric$ and $Z$ be pairwise disjoint finite sets of variables.
	Then the~following equations hold for any Young diagram $\lambda$:
	\begin{align}
	\label{eq:s(XuZ)}
		s_\lambda(X\sqcup Z) &= \sum_{\alpha,\beta}
			c_{\alpha\beta}^\lambda \, s_\alpha(X)\, s_\beta(Z),
	\\
	\label{eq:s(X) from s(XuZ)s(Z)}
		s_\lambda(X) & = \sum_{\alpha,\beta}
			c_{\alpha\beta}^\lambda \, (-1)^{|\beta|}\, s_\alpha(X\sqcup Z)\, s_{\beta^t}(Z),
	\text{ and}
	\\
	\label{eq:s(X)s(Y) vs s(XuZ)s(YuZ)}
		\sum_{\alpha,\beta} (-1)^{|\beta|}
			c^\lambda_{\alpha\beta}\, s_\alpha(X)\, s_{\beta^t}(Y) &=
		\sum_{\alpha,\beta} (-1)^{|\beta|}
			c^\lambda_{\alpha\beta}\, s_\alpha(X\sqcup Z)\, s_{\beta^t}(Y\sqcup Z).
	\end{align}
\end{proposition}

\begin{proof}
	The~derivation of \eqref{eq:s(XuZ)} can be found in \cite[eq.(5.9)]{MR3443860}
	and the~formula \eqref{eq:s(X) from s(XuZ)s(Z)} is the~special case of
	\eqref{eq:s(X)s(Y) vs s(XuZ)s(YuZ)} for $Y=\emptyset$.
	The~last equality is proven in \cite[Lemma A.7]{RW1}.
\end{proof}

\begin{corollary}\label{cor:sym-pol-add-variables}
	Let  $a\ric$, $b$ be two non-negative integers and $X\ric$, $Y\ric$, $Z$
	pairwise disjoint finite sets of variables.
	Then
	\[
		\sum_{\alpha \in T(a,b)} (-1)^{|\widehat{\alpha}|} 
			s_\alpha(X) s_{\widehat\alpha}(Y) =
		\sum_{\alpha \in T(a,b)} (-1)^{|\widehat{\alpha}|} 
			s_\alpha(X\sqcup Z) s_{\widehat\alpha}(Y\sqcup Z).
	\]      
\end{corollary}
\begin{proof}
	Set $\lambda = \YDbox(a,b)$ in \eqref{eq:s(X)s(Y) vs s(XuZ)s(YuZ)}.
\end{proof}

A~sequence of positive numbers $\underline k = (k_1,\dots,k_r)$
with $k_1+\dots + k_r = N$ is called a~\emph{composition of $N\ric$}.
It determines a~parabolic subgroup
\(
	\symm_{\underline k} := \symm_{k_1} \times \dots \times \symm_{k_r}
\)
of $\symm_N$ and a~ring $R^{\underline k} := R^{\symm_{\underline k}}$
of polynomials invariant under the~action of the~subgroup.
In particular, $R^{(1,\dots, 1)} = R$ and $R^{(N)} = \mathit{Sym}_N$.
Clearly,
\(
	R^{\underline k} \cong \mathit{Sym}_{k_1} \otimes \dots \otimes \mathit{Sym}_{k_r}
\).

We say that a~composition $\underline\ell$ is a~\emph{refinement}
of $\underline k$ if it is obtained by replacing each $k_i$ with
a composition of $k_i$, possibly of length 1. In such case
$\symm_{\underline\ell} \subseteq \symm_{\underline k}$ and
$R^{\underline k}$ is a~subring of $R^{\underline\ell}$.
The~following is a~standard fact from representation theory.

\begin{theorem}[{\cite[Theorem 24.40]{EliasMakisumiThielWilliamson}}]\label{thm:frob-ext}
	Let $\underline{k}$ be a composition of $N$ and $\underline\ell$ a~refinement of $\underline k$.
	Then $R^{\underline k} \subseteq R^{\underline\ell}$ is a~graded Frobenius extension.\footnote{
		An~extension $A\subseteq B$ is \emph{Frobenius} if there
		is a~nondegenerate $A$--linear trace $\epsilon\colon B\to A$.
		It is a~\emph{graded extension of degree $d$} if $A$ and $B$
		are graded and $\epsilon$ is homogeneous of degree $-2d$.
	}
	In particular, $R^{\underline\ell}$ is a~free module over $R^{\underline k}$.
\end{theorem}

\begin{example}[{cf.\ \cite[Theorem 2.12]{KLMS}}]\label{ex:elementary-extension}
	Assume that $\underline\ell= (\ell_1,\dots,\ell_{r+1})$ is an~\emph{elementary}
	refinement of $\underline k$, i.e.\ there exists an~index $i$, such that
	\[
		k_j = \begin{cases}
			\ell_j, & j < i,\\
			\ell_i+\ell_{i+1}, & j = i, \\
			\ell_{j+1}, & j > i.
		\end{cases}
              \]
	Then the~extension $R^{\underline k} \subset R^{\underline\ell}$ has degree
	$\ell_i \ell_{i+1}$ and the~basis of $R^{\underline \ell}$ is given by elements
	\[
		b_\lambda := 1^{\otimes i} \otimes s_\lambda \otimes 1^{\otimes r-i}
	\]
	with $\lambda \in T(\ell_{i+1}, \ell_i)$. The~trace map $\epsilon\colon R^{\underline \ell} \to R^{\underline k}$
	takes $b_\lambda$ to $1$ if $\lambda = \YDbox(\ell_{i+1},\ell_i)$
	and to $0$ otherwise.
\end{example}

\begin{example}\label{ex:parabolic-basis}
	The~ring $R^{\underline k}$ is a~free module over
	$R^{(N)} \cong \mathit{Sym}_N$.
	It has a  basis  given by pure tensors of Schur polynomials
	\[
		1 \otimes s_{\lambda_2} \otimes \cdots \otimes s_{\lambda_r},
	\]
	where $\lambda_i$ is a~Young diagram with at most
	 $k_1 + \ldots + k_{i-1}$ columns and $k_i$ rows.
\end{example}

Let $\ccat{Bim}$ be the~bicategory of rings, bimodules, and bimodule maps,
with the~horizontal composition given by the~tensor product of bimodules.
Consider the~induction and restriction bimodules
\[
	\mathrm{Ind}^{\underline\ell}_{\underline k} \cong
		{}_{R^{\underline\ell}}(R^{\underline\ell})_{R^{\underline k}}
\qquad
	\Res^{\underline\ell}_{\underline k} \cong
		{}_{R^{\underline k}}(\qshift^d R^{\underline\ell})_{R^{\underline\ell}}
\]
for all Frobenius extensions $R^{\underline k} \subset R^{\underline\ell}$,
where $d$ is the~degree of the~extension. Their finite compositions, i.e.\ tensor products
over the~polynomial rings, are called \emph{singular Bott--Samelson bimodules}.

\begin{definition}
	The~\emph{bicategory of singular Soergel bimodules} $\ccat{sSBim}$ is the~full graded additive
	and idempotent complete subbicategory of $\ccat{Bim}$ with rings $R^{\underline k}$
	as objects and 1-morphisms generated by singular Bott--Samelson bimodules.
	In other words, every 1-morphism in $\ccat{sSBim}(R^{\underline k}, R^{\underline\ell})$
	is a~direct summand of a~bimodule of the~form $\bigoplus_{i=1}^r \qshift^{d_i} B_i$, where each
	$B_i \in \ccat{Bim}(R^{\underline k}, R^{\underline\ell})$ is a~singular Bott--Samelson bimodule.
\end{definition}

\begin{remark}
	It follows directly from the~definition that a~singular Soergel bimodule
	is projective when seen as a~left or as a~right module. Moreover, it
	is free when it is a~direct sum of singular Bott--Samelson bimodules.
\end{remark}

\begin{remark}
	The~morphism category $\ccat{sSBim}(R,R)$ is the~category
	of classical (i.e.\ non-singular) Soergel bimodules.
\end{remark}

 \subsection{Hochschild homology}
\label{sec:hochschild}
Let $A$ be a~$\scalars$-algebra and $M$ an~$(A,A)$--bimodule.
The~\emph{Hochschild homology of $M$} is the~homology of
the~chain complex $\CHoHom_{\!\bullet}(A,M)$
with chain groups $\CHoHom_{\!n}(A, M) := M\otimes A^{\otimes n}$
and the differential given by the~alternating sum
\begin{equation}
	\begin{split}
	\partial(m\otimes a_1\otimes\dots\otimes a_n) = 
		ma_1\otimes a_2\otimes&\dots\otimes a_n \\
		+ \sum_{i=1}^{n-1} (-1)^i m\otimes a_1\otimes &\dots
					\otimes a_ia_{i+1} \otimes \dots \otimes a_n \\
		+ (-1)^n a_nm\otimes a_1\otimes&\dots\otimes a_{n-1}.
	\end{split}
\end{equation}
The~quotient $\HoHom_{\!0}(A,M) \cong M / [A,M]$ is known as
the~\emph{space~of coinvariants of $M\ric$}, where
$[A,M] := \{ am-ma \;|\; a\in A, m\in M \}$
is the~\emph{commutator} of $A$ and $M\ric$.

Given an algebra automorphism  $\varphi\in\Aut(A)$, we can replace the last
term of the differential with
\begin{equation}\label{eq:twisted-diff}
	(-1)^n \varphi(a_n)m\otimes a_1\otimes\dots\otimes a_{n-1}.
\end{equation}
The~resulting complex $\CHoHom_{\!\bullet}^\varphi(A, M)$
is the~\emph{$\varphi$-twisted Hochschild complex}.
When both $A$ and $M$ are graded, then the~complex admits a~natural automorphism,
which leads to \emph{quantum Hochschild homology} introduced in \cite{BPW}.
Fix an~invertible element $q\in\scalars$ and define $\varphi(a) = q^{-|a|}$,
where $|a|$ is the~degree of a~homogeneous element $a\in A$.
Then the~last term of the~twisted Hochschild differential
\eqref{eq:twisted-diff} takes the~form
\begin{equation}
	(-1)^n q^{-|a_n|}a_nm\otimes a_1\otimes\dots\otimes a_{n-1}.
\end{equation}
The~\emph{quantum Hochschild homology of $M\ric$}, denoted by
$\qHoHom_{\!\bullet}(A, M)$, is the~homology of this complex.  This
construction was also reviewed in \cite{Lipshitz}.  Following
the~usual conventions we write $\qCHoHom(A)$ and $\qHoHom(A)$ when
$M=A$. Additionally, when $A$ is clear from the context, we write
$\qHoHom(M)$.

\begin{remark}
	Hochschild chains can be visualized by circles divided into segments,
	one labeled with $m\in M$ and the~others with $a_0,\dots,a_n$. Each of the~terms
	of the~differential merges two segments multiplying their labels.
	\[
			\HochschildChain{30,150,270}{m/90,a_0/-30,a_1/210}
		\longmapsto
			\HochschildChain{150,270}{ma_0/30,a_1/210}
			-
			\HochschildChain{30,150}{m/90,a_0a_1/270}
			+
			\HochschildChain{30,270}{a_1m/150,a_0/-30}
	\]
	In the~twisted case add a~mark on the~circle between segments
        labeled $m$ and $a_n$. To merge these two segments, one has to
        move $a_n$ over the~mark, acting upon it with $\varphi$ as
        depicted below.
	\[
			\HochschildChain[150]{30,150,270}{m/90,a_0/-30,a_1/210}
		\rightsquigarrow\quad
			\HochschildChain[270]{30,150,270}{m/90,a_0/-30,\varphi(a_1)/210}
		\rightsquigarrow\quad
			\HochschildChain[270]{30,270}{\varphi(a_1)m/150,a_0/-30}
	\]
\end{remark}

The~quantum Hochschild homology can be seen as arising from twisting bimodules
by algebra automorphisms.
Namely, given $\varphi\in\Aut(A)$ and a~left $A$-module $M\ric$, denote by ${}_\varphi M$
its~\emph{$\varphi$-twist}, defined as the~module $M$ with the~action twisted
by $\varphi$, i.e.\ $a \cdot m:=\varphi(a) m$.
If $M$ is an~$(A,A)$-bimodule, then it follows directly from the~definition that
\begin{equation}\label{eq:twisted-HH-as-HH}
	\CHoHom_{\!\bullet}^\varphi(A, M) \cong \CHoHom_{\!\bullet}(A, {}_\varphi M).
\end{equation}
The~following property is proven in \cite{BPW}.

\begin{proposition}\label{prop:concat-twisted-modules}
	Choose graded\/ $\scalars$-algebras $A$, $B\ric$, $C\ric$
	and graded $(A,B)$- and $(B,C)$-bimodules $M\ric$ and $N\ric$.
	Then for any invertible scalars
	$\qnt_1,\qnt_2\in\scalars$ there is a~bimodule isomorphism
	\[
		\qTwist[\qnt_1]{M} \otimes_B \qTwist[\qnt_2]{N}
			\xrightarrow{\ \cong\ }
		\qTwist[\qnt_1\qnt_2]{(M\otimes_B N)}
	\]
	defined as $m\otimes n \mapsto \qnt_2^{|m|}m\otimes n$ for homogeneous $m\in M$ and $n\in N$.
\end{proposition}

This implies together with \eqref{eq:twisted-HH-as-HH} that the~quantum Hochschild
homology is invariant under cyclic permutation of tensor factors.

\begin{proposition}\label{prop:qHH-is-cyclic}
	Pick graded\/ $\scalars$-algebras $A$ and $B$ and graded $(A,B)$- and $(B,A)$-bimodules
	$M\ric$ and $N$ that are projective as left modules. Then there is an~isomorphism
	\[
		\qHH_{\!\bullet} (A, M \otimes_B N) \cong \qHH_{\!\bullet} (B, N \otimes_A M)
	\]
	for any invertible parameter $\qnt\in\scalars$.
\end{proposition}

We end this section with a~statement about the~quantum Hochschild homology
for the~algebra $R^{\underline k}$. The~proof, which is rather technical, is
postponed to Appendix~\ref{sec:qHH-is-cyclic}.

\begin{proposition}\label{prop:qHH-vanishes}
	Suppose that $1-q^d$ is invertible for $d\neq 0$.
	Then the~inclusion\/ $\scalars \subset R^{\underline k}$ induces
	a~homotopy equivalence of chain complexes
	\[
		\qCHoHom_{\!\bullet}( R^{\underline k})
			\simeq \qCHoHom_{\!\bullet}(\scalars)
			\simeq \scalars,
	\]
	where\/ $\scalars$ lives in homological degree 0.
	In particular, higher quantum Hochschild homology vanishes.
\end{proposition}

 \subsection{Webs and foams}

This section provides the~basics of webs and foams and results that are
fundamental for this paper. More details can be found in
\cite{RW1, RW3} and \cite{queffelec2014mathfrak, QRS}.
We consider only webs and foams embedded in smooth manifolds and for a~technical
reason we assume that they have \emph{collared boundary}. This means that for a~smooth
manifold $M$ we fix a~smooth embedding $\partial M\times[0,1] \to M$ that takes
$(x,0)$ to $x$. This technical condition implies a~canonical smooth structure
on the~gluing of two such manifolds along a~boundary component.

\begin{definition}\label{def:web}
	Let $\sfce$ be an~oriented smooth surface with a~collared boundary.
	A~\emph{web} $\web \subset \sfce$ is an~oriented trivalent graph,
	possibly with endpoints, smoothly embedded\footnote{Meaning that each edge is smoothly embedded when seen as a 1-dimensional manifold with boundary.} in $\sfce$ in a~way,
	such that it coincides with $\partial\web\times [0,1)$ on the~collar of $\partial\sfce$.
	The edges of the web are labeled with positive integers such that
	at each trivalent vertex the~\emph{flow condition} holds:
	the~sum of labels of incoming edges is equal to the~sum
	of labels of outgoing edges.
	We write $E(\web)$ and $V(\web)$ respectively for the sets of edges
	and vertices of a~web $\web$ and $\ell(e)$ for the~label of an~edge $e$.
	We call $\ell(e)$ the~\emph{thickness of $e$}.
\end{definition}

The~flow condition implies that each vertex of a~web is either a~\emph{split}
or a~\emph{merge}, illustrated respectively on the~left and the~right hand side
of Figure~\ref{fig:flow-condition}.
\begin{figure}[ht]
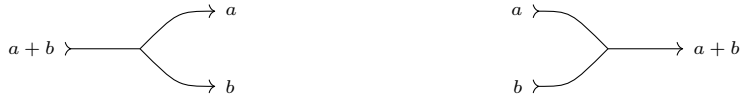

	\hfill\NB{\tikz[]{\input{\imagesfolder/hfgl0_split}}} \hfill \NB{\tikz[]{\input{\imagesfolder/hfgl0_merge}}}\hfill\ 
	\caption{A~split and a~merge vertex in a~web.}
	\label{fig:flow-condition}
\end{figure}

In this paper we are mostly interested in webs in a~strip $\IxR$ (\emph{planar webs})
or an~annulus $\SxR$ (\emph{annular webs}). We say that such a~web $\web$ is
\emph{directed} if the~projection on $[0,1]$ or $\S^1$ respectively has
no critical points when restricted to $\web$ and that projection of
orientations agree with that of $[0,1]$ or $\S^1$ respectively.
Such a~web can be visualized as a~result of a~tangential gluing of parallel
intervals oriented from left to right (or circles oriented anticlockwise in the~annular case), see Figure~\ref{fig:lamination}.
The~reverse operation is called a~\emph{lamination} \cite{MR4233203}.
\begin{figure}[ht]
  \centering
  \NB{\tikz[]{\input{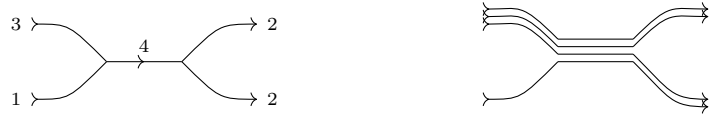}}}
  \caption{A~directed planar web of index 4 (on the~left) and its lamination (on the~right).}
  \label{fig:lamination}
\end{figure}
In particular, a~directed web $\web$ can be decomposed into a~sequence
of merges and splits. Hence, the~sum of thicknesses at a~generic section
$\web_t := \web \cap ( \{t\}\times\R )$ is constant.
We call it the~\emph{index} of $\web$. In case of webs in a~strip,
the~section $\web_0$ and $\web_1$ are called respectively
the~\emph{input} and the~\emph{output} of $\web$.

\begin{remark}
	Directed annular webs are called \emph{vinyl graphs} in \cite{RW2}.
\end{remark}

\begin{definition}\label{def:foam}
	Let $M$ be an~oriented smooth 3-manifold with a~collared boundary.
	A~\emph{foam} $\foam \subset M$ is a~collection of
        \emph{facets}, that are compact oriented surfaces
	labeled with positive integers and glued together along their boundary points in a~way,
	such that every point $p$ of $\foam$ has a~closed neighborhood homeomorphic to one of the~following:
	\begin{itemize}
		\item a~disk, when $p$ belongs to a~unique facet,
		\item $Y \times [0,1]$, where $Y$ is a~merge or a~split web, when $p$ belongs to three facets, or
		\item the~cone over the~1-skeleton of a~tetrahedron with $p$ as the~vertex of the~cone (so that it
		belongs to six facets).
	\end{itemize}
	See Figure~\ref{fig:foam-local-model} for a pictorial
        representation of these three cases. The~set of points of the~second type is a~collection
	of curves called \emph{bindings} and the~points of the~third type are called \emph{singular vertices}.
	The~\emph{boundary} $\partial\foam\ric$ of $\foam$ is the~closure of the~set of boundary points
	of facets that do not belong to a~binding. It is understood that $\foam$ coincides with
	$\partial\foam\times[0,1]$ on the~collar of $\partial M\ric$.
	We write $F(\foam)$ for the~collection of facets of $\foam$ and $\ell(\facet)$
	for the~thickness of a~facet $\facet\ric$. A~foam $\foam$ is \emph{decorated} if each
	facet $f\in F(\foam)$ is assigned a~symmetric polynomial $P_f \in \Sym_{\ell(\facet)}$.
\end{definition}

\begin{figure}[ht]
  \NB{\tikz[font=\small]{\input{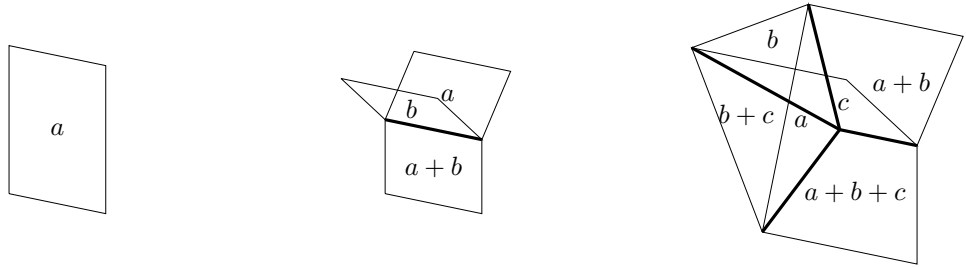}}}
	\caption{The~three local models for a~foam.}
	\label{fig:foam-local-model}
\end{figure}

\begin{remark}
	A~foam satisfies a~2-dimensional version of the~flow condition: three facets meet at each binding
	in a~way, such that the~thickness of one of them is equal to the~sum of thicknesses of the~other two.
	The~binding induces orientation from the~two thinner facets; it is opposite to the~one induced
	from the~thickest facet.
\end{remark}

The~boundary of a~foam $\foam\subset M$ is a~web in $\partial M\ric$. In case $M = \sfce\times[0,1]$
is a~thickened surface, we require that $\partial\foam\cap (\partial\sfce\times[0,1])$ is a~collection
of vertical lines, that are  lines of the form $\{x\}\times [0,1]$ for $x\in \foam$. 
A~generic section $\foam_t := \foam \cap (\sfce\times\{t\})$ is a~web, each
with the~same boundary. The~bottom and top webs $\foam_0$ and $\foam_1$ are called respectively
the~\emph{input} and \emph{output} of $\foam\ric$.

Let $\set{Foam}(M)$ be the~$\scalars$-module generated by decorated foams in $M$ modulo
\emph{local relations}, defined as follows. Consider the~collection of Robert--Wagner
evaluations
\[
	\langle-,-\rangle_N \colon \set{Foam}(\mathbb D^3)\otimes\set{Foam}(\mathbb D^3) \to \Sym_N
\]
from \cite{RW1}.
We impose the~relation $a_1\foam_1 + \dots + a_r\foam_r = 0$
whenever there is a~3-ball $B\subset M\ric$,
such that all sets $\foam_i \setminus B$ coincide and the~linear combination
$\sum_i a_i(\foam_i \cap B)$ is in the~radical of $\langle-,-\rangle_N$ for all $N>0$.
The~set $\set{Foam}(M)$ is graded by $\ZZ\oplus\ZZ$, see
\cite{FunctorialitySLN} for details.\footnote{
	This $\ZZ\oplus\ZZ$-grading is related to the~$\ZZ$-grading of $\LieGL_n$ foams by collapsing
	$(a,b)$ into $a+Nb$.
}

\subsubsection{The~bicategory of directed foams}

Let us now consider foams between planar directed webs (so that $\Sigma = \IxR$).
In this situations we impose the~additional condition that a~foam $\foam$ is
``directed'' itself, i.e.\ that the~projection onto the~side square $[0,1]\times[0,1]$
has no critical points when restricted to $\foam\ric$. This immediately implies that
a~generic section of $\foam$ is a~directed web as defined above.
A~foam of this type can be decomposed into seven basic homogeneous pieces:
traces of isotopies and six singular blocks shown in Figure~\ref{fig:basic-foams}.
For all of them the~second component of the~$(\ZZ\oplus\ZZ)$--grading vanishes,
so that the~space of directed foams is $\ZZ$--graded.

\begin{figure}[ht]
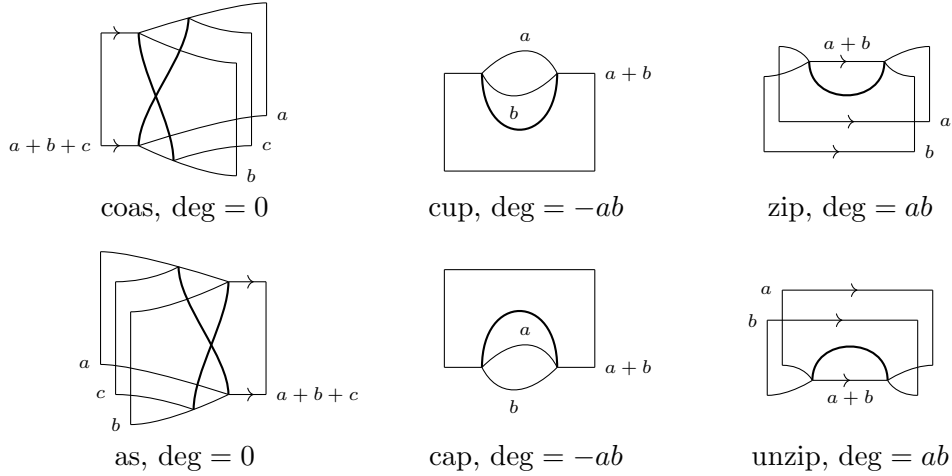

	\begin{tabular}{cp{1cm}cp{1cm}c}
		\hskip-2.5em \NB{\tikz[font=\tiny]{\begin{scope}
  \begin{scope}
    \coordinate (LL) at (0,0);
    \coordinate (L) at (0.5,0);
    \coordinate (R1) at (2.2,0.4);
    \coordinate (R2) at (2,0);
    \coordinate (R3) at (1.8, -0.4);
    \draw[->-] (LL) -- (L);
    \draw (L) .. controls +(0,0) and +(-0.5, 0) .. (R1)
    coordinate[pos =0.5] (M);
    \draw (L) .. controls +(0,0) and +(-0.5, 0) .. (R3) ;
    \draw (M) .. controls +(0,0) and +(-0.5, 0) .. (R2) ;
  \end{scope}  
 \begin{scope}[yshift = -1.5cm]
    \coordinate (LLB) at (0,0);
    \coordinate (LB) at (0.5,0);
    \coordinate (R1B) at (2.2,0.4);
    \coordinate (R2B) at (2,0);
    \coordinate (R3B) at (1.8, -0.4);
    \draw[->-] (LLB) -- (LB)  node[left, pos =0] {$a+b+c$};
    \draw (LB) .. controls +(0,0) and +(-0.5, 0) .. (R1B)    node[right] {$a$};
    \draw (LB) .. controls +(0,0) and +(-0.5, 0) .. (R3B)   node[right] {$b$}
    coordinate[pos =0.5] (MB);
    \draw (MB) .. controls +(0,0) and +(-0.5, 0) .. (R2B)    node[right] {$c$};
  \end{scope}  
  \draw (LL) -- (LLB);
  \draw (R1) -- (R1B);
  \draw (R2) -- (R2B);
  \draw (R3) -- (R3B);
  \draw[name path=path1, thick] (M) .. controls +(0,-0.5) and +(0,
  0.5) .. (LB);
  \draw[name path=path2, thick] (L) .. controls +(0,-0.5) and +(0,
  0.5) .. (MB);
  \path [name intersections={of=path1 and path2,by=O}];  
\end{scope}}} &&
		\NB{\tikz[font=\tiny]{\input{\imagesfolder/hfgl0_foam-digon-cup}}}\hskip-2em\, &&
		\NB{\tikz[font=\tiny]{\begin{scope}
  \begin{scope}
    \coordinate (L1) at (0.2,0.4);
    \coordinate (L2) at (0,0);
    \coordinate (R1) at (2.2,0.4);
    \coordinate (R2) at (2,0);
    \coordinate (ML) at (0.6, 0.2);
    \coordinate (MR) at (1.6, 0.2);
    \draw[->-] (ML) -- (MR) node[above, midway] {$a+b$};
    \draw (MR) .. controls +(0, 0) and +(-0.3,0) .. (R1) ;
    \draw (MR) .. controls +(0, 0) and +(-0.3,0) .. (R2);
    \draw (L1) .. controls +( 0.3, 0) and +(0,0) .. (ML);
    \draw (L2) .. controls +( 0.3, 0) and +(0,0) .. (ML);
  \end{scope}  
 \begin{scope}[yshift = -1cm]
    \coordinate (L1B) at (0.2,0.4);
    \coordinate (L2B) at (0,0);
    \coordinate (R1B) at (2.2,0.4);
    \coordinate (R2B) at (2,0);
    \draw[->-] (L1B) .. controls +( 0, 0) and +(0,0) .. (R1B) node [right, pos
    = 1] {$a$};
    \draw[->-] (L2B) .. controls +( 0, 0) and +(0,0) .. (R2B) node [right, pos
    = 1] {$b$};
 \end{scope}  
  \draw (R1) -- (R1B);
  \draw (R2) -- (R2B);
  \draw (L1) -- (L1B);
  \draw (L2) -- (L2B);
  \draw[thick] (ML) .. controls +(0, -0.6) and +(0, -0.6) .. (MR);
\end{scope}

}}\hskip-1em\, \\
		coas, $\deg = 0$ &&
		cup, $\deg = -ab$ &&
		zip, $\deg = ab$ \\[2ex]
		\NB{\tikz[font=\tiny]{\begin{scope}[xscale = -1 ]
  \begin{scope}
    \coordinate (LL) at (0,0);
    \coordinate (L) at (0.5,0);
    \coordinate (R1) at (2.2,0.4);
    \coordinate (R2) at (2,0);
    \coordinate (R3) at (1.8, -0.4);
    \draw[-<-] (LL) -- (L);
    \draw (L) .. controls +(0,0) and +(-0.5, 0) .. (R1)
    coordinate[pos =0.5] (M);
    \draw (L) .. controls +(0,0) and +(-0.5, 0) .. (R3) ;
    \draw (M) .. controls +(0,0) and +(-0.5, 0) .. (R2) ;
  \end{scope}  
 \begin{scope}[yshift = -1.5cm]
    \coordinate (LLB) at (0,0);
    \coordinate (LB) at (0.5,0);
    \coordinate (R1B) at (2.2,0.4);
    \coordinate (R2B) at (2,0);
    \coordinate (R3B) at (1.8, -0.4);
    \draw[-<-] (LLB) -- (LB)  node[right, pos =0] {$a+b+c$};
    \draw (LB) .. controls +(0,0) and +(-0.5, 0) .. (R1B)    node[left] {$a$};
    \draw (LB) .. controls +(0,0) and +(-0.5, 0) .. (R3B)   node[left] {$b$}
    coordinate[pos =0.5] (MB);
    \draw (MB) .. controls +(0,0) and +(-0.5, 0) .. (R2B)    node[left] {$c$};
  \end{scope}  
  \draw (LL) -- (LLB);
  \draw (R1) -- (R1B);
  \draw (R2) -- (R2B);
  \draw (R3) -- (R3B);
  \draw[name path=path1, thick] (M) .. controls +(0,-0.5) and +(0,
  0.5) .. (LB);
  \draw[name path=path2, thick] (L) .. controls +(0,-0.5) and +(0,
  0.5) .. (MB);
  \path [name intersections={of=path1 and path2,by=O}];
  
\end{scope}}}\hskip-2.5em\, &&
		\NB{\tikz[font=\tiny]{\input{\imagesfolder/hfgl0_foam-digon-cap}}}\hskip-2em\, &&
		\hskip-1em\NB{\tikz[font=\tiny]{\begin{scope}
  \begin{scope}
    \coordinate (L1) at (0.2,0.4);
    \coordinate (L2) at (0,0);
    \coordinate (R1) at (2.2,0.4);
    \coordinate (R2) at (2,0);
    \coordinate (ML) at (0.6, 0.2);
    \coordinate (MR) at (1.6, 0.2);
    \draw[->-] (ML) -- (MR) node[below, midway] {$a+b$};
    \draw (MR) .. controls +(0, 0) and +(-0.3,0) .. (R1) ;
    \draw (MR) .. controls +(0, 0) and +(-0.3,0) .. (R2);
    \draw (L1) .. controls +( 0.3, 0) and +(0,0) .. (ML);
    \draw (L2) .. controls +( 0.3, 0) and +(0,0) .. (ML);
  \end{scope}  
 \begin{scope}[yshift = 1cm]
    \coordinate (L1B) at (0.2,0.4);
    \coordinate (L2B) at (0,0);
    \coordinate (R1B) at (2.2,0.4);
    \coordinate (R2B) at (2,0);
    \draw[->-] (L1B) .. controls +( 0, 0) and +(0,0) .. (R1B) node [left, pos
    = 0] {$a$};
    \draw[->-] (L2B) .. controls +( 0, 0) and +(0,0) .. (R2B) node [left, pos
    = 0] {$b$};
 \end{scope}  
  \draw (R1) -- (R1B);
  \draw (R2) -- (R2B);
  \draw (L1) -- (L1B);
  \draw (L2) -- (L2B);
  \draw[thick] (ML) .. controls +(0, 0.6) and +(0, 0.6) .. (MR);
\end{scope}}} \\
		as, $\deg = 0$ &&
		cap, $\deg = -ab$ &&
		unzip, $\deg = ab$
	\end{tabular}
	\caption{Local models for all singularities of directed foams, together with their degrees.}
	\label{fig:basic-foams}
\end{figure}

\begin{definition}
	Let $\ccat{Foam}$ be the~bicategory of \emph{$\infty$-foams}, in which
	\begin{itemize}
		\item objects are finite sequences of points on a~line, labeled with positive integers,
		\item 1-morphisms from $\underline a$ to $\underline b$ are formal finite direct sums
		$\bigoplus_i \qshift^{d_i}\web_i$, where each $\web_i$ is a~directed web $\web\subset\IxR$ 
		with input $\underline a$ and output $\underline b$, and  $d_i \in \ZZ$.
		\item 2-morphisms from $\bigoplus_i \qshift^{d_i}\web_i$ to $\bigoplus_j \qshift^{d'_j}\web'_j$
		are matrices $(m_{ij})$, where $m_{ij}$ is a~linear combination of decorated directed foams
		in a~thickened strip with input $\web_i$, output $\web_j$, and degree $d'_j - d_i$.
	\end{itemize}
\end{definition}

\begin{remark}
	The~approach to $\ccat{Foam}$ is slightly different in \cite{queffelec2014mathfrak}.
	There one first constructs a~bicategory $n\ccat{Foam}$ of (directed)
	$\LieGL_n$ foams using technics from higher representation theory
	and writes down it presentation in terms of generators and relations.
	Then it is shown that these categories admit a~limit when $N$ goes
	to infinity. It can be shown that the~limit category coincides
	with $\ccat{Foam}$ as defined above.
\end{remark}

\begin{proposition}[{\cite[Proposition 5.10]{RW2}}, \cite{queffelec2014mathfrak}]
  \label{prop:local-isoms}
	There are graded isomorphisms of webs in $\ccat{Foam}$
	\begin{gather*}
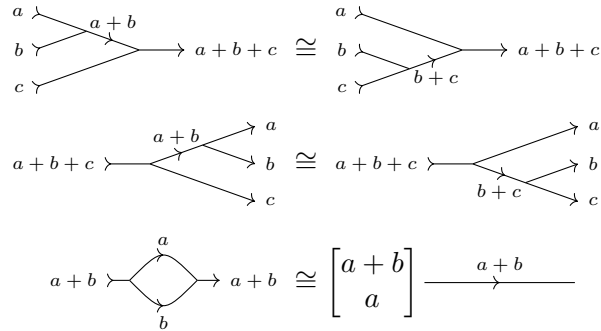
	\begin{aligned}
		\NB{\tikz[font=\tiny]{\begin{scope}
  \coordinate (a) at (-1, 0.5);
  \coordinate (b) at (-1, 0 );
  \coordinate (c) at (-1,-0.5);
  \coordinate (abc) at (0.4, 0);
  \coordinate (r) at (1, 0);
  \coordinate (ab) at (-0.3, 0.25);
  \draw[>-] (c) -- (abc) node[pos =0, left] {$c$};
  \draw[>-] (a) -- (ab) node[pos =0, left] {$a$};
  \draw[>-] (b) -- (ab) node[pos =0, left] {$b$};
  \draw[->-] (ab) --(abc) node[pos=0.5, above] {$a+b$};
  \draw[->] (abc) -- (r) node[pos =1, right] {$a+b+c$};
\end{scope}}}
			&\cong
		\NB{\tikz[font=\tiny]{\begin{scope}
  \coordinate (a) at (-1, 0.5);
  \coordinate (b) at (-1, 0 );
  \coordinate (c) at (-1,-0.5);
  \coordinate (abc) at (0.4, 0);
  \coordinate (r) at (1, 0);
  \coordinate (bc) at (-0.3, -0.25);
  \draw[>-] (c) -- (bc) node[pos =0, left] {$c$};
  \draw[>-] (a) -- (abc) node[pos =0, left] {$a$};
  \draw[>-] (b) -- (bc) node[pos =0, left] {$b$};
  \draw[->-] (bc) --(abc) node[pos=0.5, below] {$b+c$};
  \draw[->] (abc) -- (r) node[pos =1, right] {$a+b+c$};
\end{scope}}}
	\\
		\NB{\tikz[font=\tiny]{\begin{scope}
  \coordinate (a) at (1, 0.5);
  \coordinate (b) at (1, 0 );
  \coordinate (c) at (1,-0.5);
  \coordinate (abc) at (-0.4, 0);
  \coordinate (r) at (-1, 0);
  \coordinate (ab) at (0.3, 0.25);
  \draw[<-] (c) -- (abc) node[pos =0, right] {$c$};
  \draw[<-] (a) -- (ab) node[pos =0, right] {$a$};
  \draw[<-] (b) -- (ab) node[pos =0, right] {$b$};
  \draw[-<-] (ab) --(abc) node[pos=0.5, above] {$a+b$};
  \draw[-<] (abc) -- (r) node[pos =1, left] {$a+b+c$};
\end{scope}}}
			&\cong
		\NB{\tikz[font=\tiny]{\begin{scope}
  \coordinate (a) at (1, 0.5);
  \coordinate (b) at (1, 0 );
  \coordinate (c) at (1,-0.5);
  \coordinate (abc) at (-0.4, 0);
  \coordinate (r) at (-1, 0);
  \coordinate (bc) at (0.3, -0.25);
  \draw[<-] (c) -- (bc) node[pos =0, right] {$c$};
  \draw[<-] (a) -- (abc) node[pos =0, right] {$a$};
  \draw[<-] (b) -- (bc) node[pos =0, right] {$b$};
  \draw[-<-] (bc) --(abc) node[pos=0.5, below] {$b+c$};
  \draw[-<] (abc) -- (r) node[pos =1, left] {$a+b+c$};
\end{scope}}}
	\\
		\NB{\tikz[font=\tiny,scale=1.5]{\begin{scope}
\draw[>-] (0,0) -- +(0.2,0) node[pos=0, left]
      {$a+b$};
      \draw[->] (0.8,0) -- +(0.2, 0) node[pos=1, right]
      {$a+b$};
      \draw[->-] (0.2, 0) .. controls +(0.3,0.3) and +(-0.3, 0.3)
      .. (0.8,0) node[pos =0.5, above] {$a$};
      \draw[->-] (0.2, 0) .. controls +(0.3,-0.3) and +(-0.3, -0.3)
      .. (0.8,0) node[pos =0.5, below] {$b$};
\end{scope}}}
			&\cong
		\qbinom{a+b}{a} \NB{\tikz[font=\tiny,scale=2]{\begin{scope}
  \draw[->-] (0,0) -- +(1,0) node[pos= 0.5, above] {$a+b$} node[pos=
  0.5, below, white] {$a+b$};
\end{scope}}}
	\end{aligned}\\
		\NB{\tikz[font=\tiny,xscale=1.5]{\begin{scope}
  \coordinate (t1) at (-1, 0.5);
  \coordinate (t2) at (-0.3, 0.5);
  \coordinate (t3) at ( 0.3, 0.5);
  \coordinate (t4) at ( 1, 0.5);
  \coordinate (b1) at (-1,   -0.5);
  \coordinate (b2) at (-0.5, -0.5);
  \coordinate (b3) at ( 0.5, -0.5);
  \coordinate (b4) at ( 1,   -0.5);
  \draw[>-]  (t1) -- (t2) node[pos =0, left] {$a$};
  \draw[->-] (t2) -- (t3) node[pos =0.5, above] {$a+d$};
  \draw[->]  (t3) -- (t4) node[pos =1, right] {$b$};
  \draw[>-]  (b1) -- (b2) node[pos =0, left] {$b+c$};
  \draw[->-] (b2) -- (b3) node[pos =0.5, below] {$b+c-d$};
  \draw[->]  (b3) -- (b4) node[pos =1, right] {$a+c$};
  \draw[->-] (b2) -- (t2) node[pos = 0.5, left] {$d$};
  \draw[->-] (t3) -- (b3) node[pos = 0.5, right] {$a+d-c$};
\end{scope}}}
			\cong
		\bigoplus_{j = \max(0, b-a)}^b
			\qbinom{c}{d-j} \NB{\tikz[font=\tiny,xscale=1.5]{\begin{scope}
  \coordinate (t1) at (-1, 0.5);
  \coordinate (t2) at (-0.5, 0.5);
  \coordinate (t3) at ( 0.5, 0.5);
  \coordinate (t4) at ( 1, 0.5);
  \coordinate (b1) at (-1,   -0.5);
  \coordinate (b2) at (-0.3, -0.5);
  \coordinate (b3) at ( 0.3, -0.5);
  \coordinate (b4) at ( 1,   -0.5);
  \draw[>-]  (t1) -- (t2) node[pos =0, left] {$a$};
  \draw[->-] (t2) -- (t3) node[pos =0.5, above] {$b-j$};
  \draw[->]  (t3) -- (t4) node[pos =1, right] {$b$};
  \draw[>-]  (b1) -- (b2) node[pos =0, left] {$b+c$};
  \draw[->-] (b2) -- (b3) node[pos =0.5, below] {$a+c+j$};
  \draw[->]  (b3) -- (b4) node[pos =1, right] {$a+c$};
  \draw[-<-] (b2) -- (t2) node[pos = 0.5, left] {$a+j-b$};
  \draw[-<-] (t3) -- (b3) node[pos = 0.5, right] {$j$};
\end{scope}}}
	\end{gather*}
\end{proposition}

Of particular interest to us are webs and foams with labels at most 2,
the~former having all endpoints labeled one.
They arise naturally as resolutions of uncolored link diagrams.
Following \cite{RW3} we call them \emph{elementary}. In what follows
we write $\ccat{FoamEl}$ for the~linear subbicategory of $\ccat{Foam}$
generated by elementary foams and webs. More precisely, this is the smallest
subbicategory containing elementary foams and webs that is closed under direct sums of webs and linear combinations of foams.

\begin{proposition}\label{prop:local-elem-isoms}
	There are isomorphisms of elementary webs in $\ccat{FoamEl}$:
   \begin{align}
		\NB{\tikz[font=\tiny,scale=1.5]{\begin{scope}
\draw[>-] (0,0) -- +(0.2,0) node[pos=0, left]
      {$2$};
      \draw[->] (0.8,0) -- +(0.2, 0) node[pos=1, right]
      {$2$};
      \draw[->-] (0.2, 0) .. controls +(0.3,0.3) and +(-0.3, 0.3)
      .. (0.8,0) node[pos =0.5, above] {$1$};
      \draw[->-] (0.2, 0) .. controls +(0.3,-0.3) and +(-0.3, -0.3)
      .. (0.8,0) node[pos =0.5, below] {$1$};
\end{scope}}}
			&\cong
		[2] \NB{\tikz[font=\tiny,scale=1.5]{\begin{scope}
  \draw[->-] (0,0) -- +(1,0) node[pos= 0.5, above] {$2$} node[pos=
  0.5, below, white] {$2$};
\end{scope}}}
	\\
		\NB{\tikz[yscale=0.6,xscale=0.5,font=\tiny]{\begin{scope}
  \coordinate (lb) at (-3, 0);
  \coordinate (lm) at (-3, 1);
  \coordinate (lt) at (-3, 2);
  \coordinate (rb) at (3, 0);
  \coordinate (rm) at (3, 1);
  \coordinate (rt) at (3, 2);
  \coordinate (ll) at (-2, 1.5);
  \coordinate (lr) at (-1, 1.5);
  \coordinate (rl) at ( 1, 1.5);
  \coordinate (rr) at ( 2, 1.5);
  \coordinate (ml) at (-0.5,0.5);
  \coordinate (mr) at ( 0.5,0.5);

  \draw[>-] (lt) .. controls +(0.5, 0) and +(-0.2, 0.2) .. (ll) node[pos =0, left] {$1$};
  \draw[>-] (lm) .. controls +(0.5, 0) and +(-0.2,-0.2) .. (ll) node[pos =0, left] {$1$};
  \draw[>-] (lb) .. controls +(0.5, 0) and +(-0.2,-0.2) .. (ml) node[pos =0, left] {$1$};

  \draw[<-] (rt) .. controls +(-0.5, 0) and +( 0.2, 0.2) .. (rr) node[pos =0, right] {$1$};
  \draw[<-] (rm) .. controls +(-0.5, 0) and +( 0.2,-0.2) .. (rr) node[pos =0, right] {$1$};
  \draw[<-] (rb) .. controls +(-0.5, 0) and +( 0.2,-0.2) .. (mr) node[pos =0, right] {$1$};
  \draw[->-] (lr) --(ml) node[pos= 0.5, left] {$1$};
  \draw[->-] (mr) --(rl) node[pos= 0.5, right] {$1$};  
  \draw[->-] (ll) -- (lr) node[pos =0.5, above] {$2$};
  \draw[->-] (rl) -- (rr) node[pos =0.5, above] {$2$};
  \draw[->-] (ml) -- (mr) node[pos =0.5, below] {$2$};
  \draw[->-] (lr) .. controls +(0.2, 0.2) and +(-0.2, 0.2) .. (rl) node[pos =0.5, above] {$1$};
\end{scope}}}
			\oplus
		\NB{\tikz[yscale=0.6,xscale=0.3,font=\tiny]{\begin{scope}[yscale=-1]
  \coordinate (lb) at (-3, 0);
  \coordinate (lm) at (-3, 1);
  \coordinate (lt) at (-3, 2);
  \coordinate (rb) at (3, 0);
  \coordinate (rm) at (3, 1);
  \coordinate (rt) at (3, 2);
  \coordinate (ll) at (-2, 1.5);
  \coordinate (lr) at (-1, 1.5);
  \coordinate (rl) at ( 1, 1.5);
  \coordinate (rr) at ( 2, 1.5);
  \coordinate (ml) at (-0.5,0.5);
  \coordinate (mr) at ( 0.5,0.5);

  \draw[>-] (lt) .. controls +(0.5, 0) and +(-0.2, 0.2) .. (ll) node[pos =0, left] {$1$};
  \draw[>-] (lm) .. controls +(0.5, 0) and +(-0.2,-0.2) .. (ll) node[pos =0, left] {$1$};
\draw [>->] (lb) -- (rb) node [pos =0, left] {$1$} node [pos =1,
right] {$1$};
\draw [->-] (ll) -- (rr) node[pos =0.5, above] {$2$};
  \draw[<-] (rt) .. controls +(-0.5, 0) and +( 0.2, 0.2) .. (rr) node[pos =0, right] {$1$};
  \draw[<-] (rm) .. controls +(-0.5, 0) and +( 0.2,-0.2) .. (rr) node[pos =0, right] {$1$};
\end{scope}}}
			&\cong
		\NB{\tikz[yscale=0.6,xscale=0.5,font=\tiny]{\begin{scope}[yscale = -1]
  \coordinate (lb) at (-3, 0);
  \coordinate (lm) at (-3, 1);
  \coordinate (lt) at (-3, 2);
  \coordinate (rb) at (3, 0);
  \coordinate (rm) at (3, 1);
  \coordinate (rt) at (3, 2);
  \coordinate (ll) at (-2, 1.5);
  \coordinate (lr) at (-1, 1.5);
  \coordinate (rl) at ( 1, 1.5);
  \coordinate (rr) at ( 2, 1.5);
  \coordinate (ml) at (-0.5,0.5);
  \coordinate (mr) at ( 0.5,0.5);

  \draw[>-] (lt) .. controls +(0.5, 0) and +(-0.2, 0.2) .. (ll) node[pos =0, left] {$1$};
  \draw[>-] (lm) .. controls +(0.5, 0) and +(-0.2,-0.2) .. (ll) node[pos =0, left] {$1$};
  \draw[>-] (lb) .. controls +(0.5, 0) and +(-0.2,-0.2) .. (ml) node[pos =0, left] {$1$};

  \draw[<-] (rt) .. controls +(-0.5, 0) and +( 0.2, 0.2) .. (rr) node[pos =0, right] {$1$};
  \draw[<-] (rm) .. controls +(-0.5, 0) and +( 0.2,-0.2) .. (rr) node[pos =0, right] {$1$};
  \draw[<-] (rb) .. controls +(-0.5, 0) and +( 0.2,-0.2) .. (mr) node[pos =0, right] {$1$};
  \draw[->-] (lr) --(ml) node[pos= 0.5, left] {$1$};
  \draw[->-] (mr) --(rl) node[pos= 0.5, right] {$1$};  
  \draw[->-] (ll) -- (lr) node[pos =0.5, below] {$2$};
  \draw[->-] (rl) -- (rr) node[pos =0.5, below] {$2$};
  \draw[->-] (ml) -- (mr) node[pos =0.5, above] {$2$};
  \draw[->-] (lr) .. controls +(0.2, 0.2) and +(-0.2, 0.2) .. (rl) node[pos =0.5, below] {$1$};
\end{scope}}}
			\oplus
		\NB{\tikz[yscale=0.6,xscale=0.3,font=\tiny]{\input{\imagesfolder/hfgl0_hecke3}}}
   \end{align}
\end{proposition}

\subsubsection{Directed annular webs and foams}

Consider now directed annular webs, so that $\Sigma = \SxR$.
Again, we consider only \emph{directed foams} between them,
on which the~projection onto $\S^1\times[0,1]$ has no critical points.
These foams have the~same six types of singularities from Figure~\ref{fig:basic-foams}
as directed foams in a~thickened strip.

Annular webs and foams consitute a~category $\cat{AFoam}$ constructed
in the~same fashion as $\ccat{Foam}$, keeping in mind that annular webs
have no endpoints. The~objects of $\cat{AFoam}$ are formal finite direct
sums $\bigoplus_i \qshift^{d_i} \web_i$, where each $\web_i$ is a~directed
annular web, and morphisms from $\bigoplus_i \qshift^{d_i} \web_i$ to
$\bigoplus_j \qshift^{d_j'} \web_j$ are matrices $(m_{ij})$, where each $m_{ij}$ is
a~linear combination of decorated directed annular foams with input $\web_i$,
output $\web_j$, and degree $d'_j - d_i$. We impose the~same local relations
as discussed above.
It contains a~subcategory $\cat{AFoamEl}$ of \emph{elementary
annular webs and foams}, where we consider only webs and foams
with edges and facets of thickness at most 2.

\begin{figure}[ht]
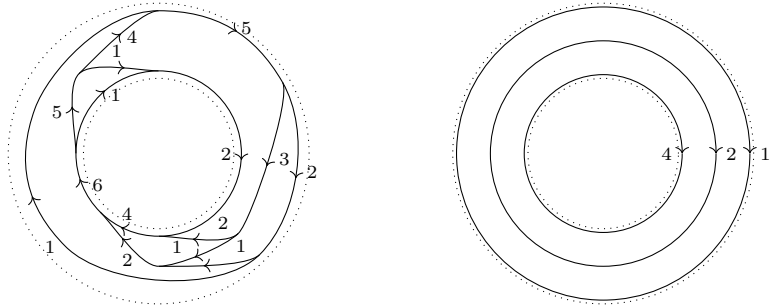

  \centering
  \NB{\tikz[]{\begin{scope}[font = \tiny, scale= 0.5]
  \coordinate (O) at (0, 0);
  \draw[dotted] (O) circle (2cm);
  \draw[dotted] (O) circle (4cm);
  \coordinate (F1) at (30:3.8);
  \coordinate (F2) at (90:3.8);
  \coordinate (F3) at (-45:3.8);
  \coordinate (M1) at (135:3);
  \coordinate (M2) at (-90:3);
  \coordinate (M3) at (-45:3);
  \coordinate (I1) at (90:2.2);
  \coordinate (I2) at (180:2.2);
  \coordinate (I3) at (225:2.2);
  \coordinate (I4) at (-90:2.2);
  \draw[-<-] (I1) arc (90:180:2.2)  node[midway, right] {$1$};
  \draw[-<-] (I2) arc (180:225:2.2)  node[midway, right] {$6$};
  \draw[-<-] (I3) arc (-135:-90:2.2)  node[midway, above] {$4$};;
  \draw[-<-] (I4) arc (-90:90:2.2) node[midway, left] {$2$};
  \draw[-<-] (M3) .. controls +(45:0.5) and +(-60:0.5) .. (F1) node[midway, right] {$3$};
  \draw[-<-] (F2) .. controls +(180:0.5) and +(45:0.5) .. (M1)  node[midway, right] {$4$};
  \draw[-<-] (I1) .. controls +(180:0.5) and +(45:0.5) .. (M1) node[midway, above] {$1$};
  \draw[-<-] (M1) .. controls +(225:0.5) and +(90:0.5) .. (I2)  node[midway, left] {$5$};
  \draw[-<-] (I3) .. controls +(-45:0.5) and +(180:0.5) .. (M2)  node[midway, below] {$2$};
  \draw[-<-] (I4) .. controls +(0:0.5) and +(225:0.5) .. (M3)  node[near end, above] {$2$};
  \draw[-<-] (M2) .. controls +(0:0.5) and +(225:0.5) .. (M3) node[near start, above] {$1$};
  \draw[-<-] (M2) .. controls +(0:0.5) and +(225:0.5) .. (F3)  node[near end, above] {$1$};
  \draw[-<-] (F1) .. controls +(120:1.5) and +(0:1.5) .. (F2)  node[midway, right] {$5$};
  \draw[-<-] (F2) .. controls +(180:1.5) and  +(70:1) .. (160:3.5) .. controls +(250:1.5) and +(135:1.5) .. (225:3.5) node[left] {$1$} .. controls +(-45:1.5) and +(225:1.5) .. (F3);
  \draw[-<-] (F3) .. controls +(45:1.5) and +(-60:1.5) .. (F1)  node[midway, right] {$2$};
\end{scope}}} \qquad\qquad
  \NB{\tikz[]{\begin{scope}[font = \tiny, scale= 0.5]
  \draw[dotted] (0,0) circle (2cm);
  \draw[dotted] (0,0) circle (4cm);
  \draw[->] (0:2.1cm) arc (0:-360:2.1cm) node[left] {$4$};
  \draw[->] (0:3cm) arc (0:-360:3cm) node[right] {$2$};
  \draw[->] (0:3.9cm) arc (0:-360:3.9cm) node[right] {$1$};
\end{scope}}}
  \caption{Examples of directed annular webs of index 7.
		The~one to the~right is $\circles_{(4,2,1)}$.}
  \label{fig:exa-annular-web}
\end{figure}
\begin{example}\label{exa:Slambda}
	Given a~finite sequence $\underline k = (k_1,\dots,k_r)$ one can consider
	a~disjoint union of $r$ concentric clockwise oriented circles with thicknesses
	$k_1, \dots, k_r$, read from the~most nested circle towards the~unnested one.
	We called it a~\emph{circular web} and denote by $\circles_{\underline{k}}$.
\end{example}

The~next proposition follows from the~Queffelec--Rose--Sartori reduction
algorithm for annular webs.

\begin{proposition}[cp.\ {\cite[Theorem 3.2]{QRS}}]
\label{prop:annular-reduction}
	Given an~annular directed web $\web$, there are graded direct sums
	of circular webs $S_L$ and $S_R$, such that $\web \oplus S_L \cong S_R$
	in $\cat{AFoam}$.
\end{proposition}

There is a~similar result for elementary annular webs, with circular
webs replaced by another class of webs.

\begin{definition}\label{def:chain-of-dumbbells}
	A~\emph{chain of dumbbells of index $\bindex$} is an~annular web $D_\bindex$
	obtained from $k$ concentric circles by glueing each pair of neighboring
	circles along an~arc, such that \nth{i} circle is glued with \nth{(i+1)}
	immediately after it is glued with \nth{(i-1)}, see Figure~\ref{fig:chain-of-dumbles0}.
\end{definition}

Note that a~chain of dumbbells of index $\bindex\geq 3$ consists of $\bindex - 1$
thick edges and $2\bindex-1$ thin edges.
\begin{figure}[ht]
    \centering
	\[
		\NB{\tikz[font= \tiny ,scale=0.5]{\input{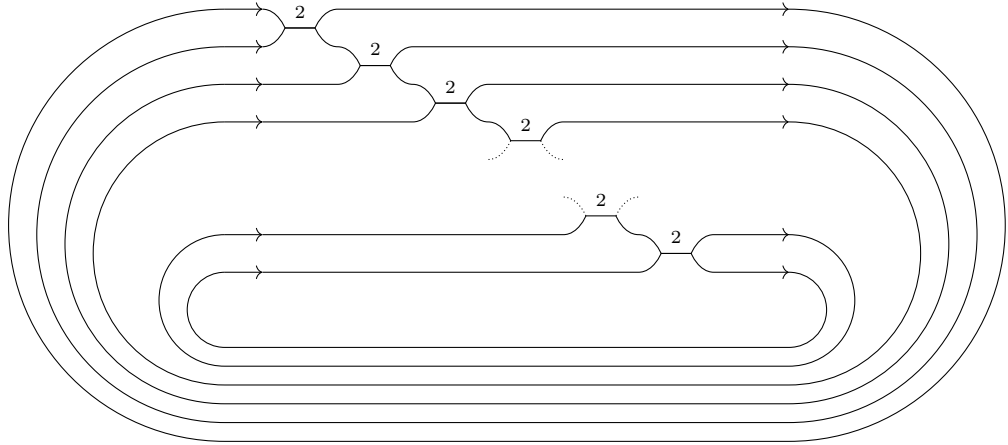}}}
	\]    
    \caption{A~chain of dumbbells.}
    \label{fig:chain-of-dumbles0}
\end{figure}
We say that an~elementary web is \emph{basic} if it is a~concetric collections
of circles and chains of dumbbells. They play the~role of circular webs in
$\cat{AFoamEl}$.

\begin{proposition}[{\cite[Corollary 2.5]{RW3}}]\label{prop:elem-annular-reduction}
	Given an~elementary annular directed web $\web$, there are graded
	direct sums of basic elementary webs $X_L$ and $X_R$, such that
	$\web \oplus X_L \cong X_R$ in $\cat{AFoamEl}$.
\end{proposition}

\subsubsection{Pointed annular webs}
\label{sec:pointed-annular}
The~last category of webs we consider is the~category $\cat{AFoamB}$
of \emph{pointed annular webs},
the~objects of which are directed annular webs,
each with a~\emph{marking} $\basepoint$ placed on an~edge of thickness 1 on the outer side of the web i.e.\ it can be connected
to the~infinity by a~curve disjoint from the~web. Moreover the~markings of all webs are located at a~fixed
point of the~ambient annulus. In particular, not all webs appear in this category.
Morphisms between two such webs are generated by annular foams with the~property
that the~markings of the~top and bottom boundary webs lie at the~boundary
of the~same facet and are connected by a vertical~interval embedded in this facet.
This interval splits the~facet into two parts that are treated as separate
facets.\footnote{
	Essentially, one can consider the~marking as a~bivalent vertex.
}
In particular, they can be decorated with different polynomials.
Forgetting the~markings of webs and lines connecting them in foams
gives a~forgetful functor $\cat{AFoamB} \to \cat{AFoam}$.

Note that there is no direct analogue of Proposition~\ref{prop:annular-reduction}.
However,  a version of  Proposition~\ref{prop:elem-annular-reduction} still holds in the pointed setting:

\begin{proposition}[{\cite[Proof of Lemma 4.29]{RW3}}]\label{prop:elem-annular-reduction-pointed}
	Given an~elementary annular directed web $\web$ with a marking on an~outer thin edge, there are graded
	direct sums of basic elementary webs $X_L$ and $X_R$ with markings on the outer thin edge, such that
	$\web \oplus X_L \cong X_R$ in $\cat{AFoamElB}$.
\end{proposition}
 \subsection{Foams and webs as Soergel bimodules}
\label{sec:foam-functor}

Directed webs and foams can be seen as a~graphical representation
of Soergel bimodules and bimodule maps. Indeed, there is a~fully faithful
functor from foams to Soergel bimodules, the~construction of which
we recall in what follows. We refer to \cite{WedrichExp, RW2} for more details.

Pick a~web $\web$ and associate with each edge $u \in E(\web)$ of thickness $r$
the~graded $\scalars$-algebra of symmetric polynomials
\( R_u := \scalars[x_{u,1},\dots, x_{u,r}]^{\symm_r}\!\),
where $\deg x_{u,i} = 2$.
For simplicity we will often write $X_u$ for the~set of variables corresponding
to the~edge $u$.
The~tensor product over $\scalars$
\[
	D(\web) := \bigotimes_{u\in E(\web)} R_u,
\]
is called the~\emph{space of decorations} of $\web$. It is the~algebra of polynomials
in edge variables that are symmetric with respect to permutions that preserve each set $X_u$.
A~pure tensor from $D(\web)$ corresponds to assigning a~symmetric polynomial $P_u \in R_u$
to each edge $u \in E(\web)$.
Therefore, we represent such elements with collections of dots on edges of $\web$,
each labeled with the~corresponding polynomial, see Figure~\ref{fig:decorated-web}.
As special cases we consider
\begin{itemize}
	\item a~dot labeled by a~Young diagram $\lambda$ representing
	the~Schur polynomial $s_\lambda$, and
	\item a~dot labeled by an~integer $i > 0$ on an~edge $u$ of thickness 1
	to represent the~monomial $x_u^i$.
\end{itemize}
Dots on the~same edge follow the~multiplicative convention:
two dots labeled $P_1$ and $P_2$ on the~same edge are equal to a~dot labeled $P_1P_2$
and an edge with no dot is decorated by $1$.

\begin{figure}[ht]
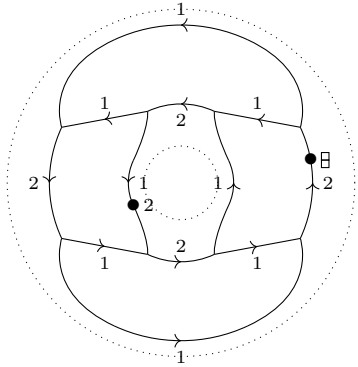

  \centering
  \NB{\tikz[scale = 0.7]{\begin{scope}[font = \tiny]
  \draw[dotted] (0,0) circle (0.7cm);
  \draw[dotted] (0,0) circle (3.3cm);
  
  \coordinate (aA1) at (65:1.5);
  \coordinate (aA2) at (65:3);
  \coordinate (bA1) at (115:1.5);
  \coordinate (bA2) at (115:3);
  \coordinate (aC1) at (245:1.5);
  \coordinate (aC2) at (245:3);
  \coordinate (bC1) at (295:1.5);
  \coordinate (bC2) at (295:3);
  \coordinate (aB1) at (155:1);
  \coordinate (aB2) at (155:2.5);
  \coordinate (bB1) at (205:1);
  \coordinate (bB2) at (205:2.5);
  \coordinate (aD1) at (335:1);
  \coordinate (aD2) at (335:2.5);
  \coordinate (bD1) at (25:1);
  \coordinate (bD2) at (25:2.5);

  \draw [->-] (aA1) arc (65:115:1.5) node[pos=0.5, below] {$2$};
  \draw [->-] (aA2) arc (65:115:3) node[pos=0.5, above] {$1$};
  \draw [->-] (aC1) arc (245:295:1.5) node[pos=0.5, above] {$2$};
  \draw [->-] (aC2) arc (245:295:3) node[pos=0.5, below] {$1$};

  \draw [->-] (aB1) arc (155:205:1) node[pos=0.5, right] {$1$} node
  [pos = 1, font=\normalsize] {$\bullet$} node
  [pos = 1, right] {$2$};
  \draw [->-] (aB2) arc (155:205:2.5) node[pos=0.5, left] {$2$};
  \draw [->-] (aD1) arc (-25:25:1) node[pos=0.5, left] {$1$};
  \draw [->-] (aD2) arc (-25:25:2.5) node[pos=0.5, right] {$2$}node
  [pos = 0.7, font=\normalsize] {$\bullet$} node
  [pos = 0.7, right, font=\small] {${\NB{\tikz[scale=0.1]{\draw (0,0) rectangle (1,1);
      \draw (1,1) rectangle (0,2);}}}$};

  \draw (bA2) .. controls +(205:1.5) and +(0,0).. (aB2);
  \draw (bC2) .. controls +(25:1.5) and +(0,0) .. (aD2);
  \draw (bB1) .. controls +(295:0.75) and +(0,0) ..(aC1);
  \draw (bD1) .. controls +(115:0.75) and +(0,0) ..(aA1);
  \draw[->-] (bA1) -- (aB2) node[pos =0.5, above] {$1$};
  \draw[->-] (bB2) -- (aC1) node[pos =0.5, below] {$1$}; 
  \draw[->-] (bC1) -- (aD2) node[pos =0.5, below] {$1$};
  \draw[->-] (bD2) -- (aA1) node[pos =0.5, above] {$1$}; 

  \draw (aA2) .. controls +(335:1.5) and +(0,0).. (bD2);
  \draw (aC2) .. controls +(155:1.5) and +(0,0) .. (bB2);
  \draw (aB1) .. controls +(65:0.75) and +(0,0) ..(bA1);
  \draw (aD1) .. controls +(245:0.75) and +(0,0) ..(bC1);

\end{scope}}}
  \caption{An~annular web with a~decoration.}
  \label{fig:decorated-web}
\end{figure}

Consider now the~ideal of \emph{local relations} $I(\web) \subset D(\web)$
generated by all differences
\begin{equation}\label{rel:Soergel}
	P(X_u) - P(X_{u'} \sqcup X_{u''}),
\end{equation}
where $u$ is an~edge of thickness $a+b$ that splits into or is a~merge of $u'$
of thickness $a$ and $u''$ of thickness $b$, and $P$ is a~symmetric polynomial
in $a+b$ variables. Diagrammatically,
\begin{equation}\label{rel:dot-migration}
\begin{aligned}
	\NB{\tikz[]{\begin{scope}[font = \tiny]
  \coordinate (l) at (0,0);
  \coordinate (o) at (1,0);
  \coordinate (rt) at (2,0.5);
  \coordinate (rb) at (2,-0.5);
  \draw[>-] (l) -- (o) node[pos = 0, left] {$a+b$} coordinate[midway] (p);
  \draw[->] (o) .. controls +(0.5, 0.5) and +(-0.5, 0) .. (rt)
  node[pos = 1, right] {$a$} coordinate[midway] (q);
  \draw[->] (o) .. controls +(0.5,-0.5) and +(-0.5, 0) .. (rb)
  node[pos = 1, right] {$b$} coordinate[midway] (r);
  \fill (p) circle (0.5mm) node[above] {$P$};
\end{scope}}}\quad &=
    		\quad\sum_{i} \NB{\tikz[]{\begin{scope}[font = \tiny]
  \coordinate (l) at (0,0);
  \coordinate (o) at (1,0);
  \coordinate (rt) at (2,0.5);
  \coordinate (rb) at (2,-0.5);
  \draw[>-] (l) -- (o) node[pos = 0, left] {$a+b$} coordinate[midway] (p);
  \draw[->] (o) .. controls +(0.5, 0.5) and +(-0.5, 0) .. (rt)
  node[pos = 1, right] {$a$} coordinate[midway] (q);
  \draw[->] (o) .. controls +(0.5,-0.5) and +(-0.5, 0) .. (rb)
  node[pos = 1, right] {$b$} coordinate[midway] (r);

  \fill (r) circle (0.5mm) node[above] {$R_{(i)}$};
  \fill (q) circle (0.5mm) node[above] {$Q_{(i)}$};
  \draw (r) node[below] {\vphantom{$Q_{(i)}$}};
\end{scope}}}
    	\qquad \text{and} \\
	\NB{\tikz[]{\begin{scope}[font = \tiny]
  \coordinate (l) at (0,0);
  \coordinate (o) at (-1,0);
  \coordinate (rt) at (-2,0.5);
  \coordinate (rb) at (-2,-0.5);
  \draw[<-] (l) -- (o) node[pos = 0, right] {$a+b$} coordinate[midway] (p);
  \draw[-<] (o) .. controls +(-0.5, 0.5) and +(0.5, 0) .. (rt)
  node[pos = 1, left] {$a$} coordinate[midway] (q);
  \draw[-<] (o) .. controls +(-0.5,-0.5) and +(0.5, 0) .. (rb)
  node[pos = 1, left] {$b$} coordinate[midway] (r);
  \fill (p) circle (0.5mm) node[above] {$P$};
\end{scope}}}\quad &=
		\quad\sum_{i} \NB{\tikz[]{\begin{scope}[font = \tiny]
  \coordinate (l) at (0,0);
  \coordinate (o) at (-1,0);
  \coordinate (rt) at (-2,0.5);
  \coordinate (rb) at (-2,-0.5);
  \draw[<-] (l) -- (o) node[pos = 0, right] {$a+b$} coordinate[midway] (p);
  \draw[-<] (o) .. controls +(-0.5, 0.5) and +(0.5, 0) .. (rt)
  node[pos = 1, left] {$a$} coordinate[midway] (q);
  \draw[-<] (o) .. controls +(-0.5,-0.5) and +(0.5, 0) .. (rb)
  node[pos = 1, left] {$b$} coordinate[midway] (r);
  \fill (r) circle (0.5mm) node[above] {$R_{(i)}\quad$};
  \fill (q) circle (0.5mm) node[above] {$Q_{(i)}$};
  \draw (r) node[below] {\vphantom{$Q_{(i)}$}};
\end{scope}}}, 
\end{aligned}
\end{equation}
where the~symmetric polynomials $Q_{(i)}$ and $R_{(i)}$ satisfy
\[
	P(X_{u'} \sqcup X_{u''}) = \sum_{i} Q_{(i)}(X_{u'}) R_{(i)}(X_{u''}).
\]
Note that the~generators of $I(\web)$ are homogeneous, so that the~ideal is graded.
Finally, given a~vertex $v\in V(\web)$ denote by $\shiftvertex{v}$ the~product of
thicknesses of the~thin edges adjacent to $v$.
The~\emph{Soergel space} associated with $\web$ is  the~graded $\scalars$-module
\[
	\set{Soergel}(\web) :=
		\qshift^{-\frac{1}{2}\sum_{v\in V(\web)} \shiftvertex{v}}
		D(\web) / I(\web).
\]
In particular, the~space is shifted downwards by the~number of thick edges
when $\web$ is an~annular elementary web. The grading shift is not necessarily integral unless the web is closed.
 Indeed, for closed webs  $\sum_{v\in V(\web)}\shiftvertex{v} = 2\sum_{v\in V_s(\web)}\shiftvertex{v} = 2\sum_{v\in V_m(\web)}\shiftvertex{v}$ where $V_s(\web)$ (resp. {}$V_m(\web)$) stands for the set of split (resp.{} merge) vertices.

Suppose now that $\web \subset \IxR$ is a~planar directed web of index $k$.
Its input and output determine compositions $\underline a$ and $\underline b$ of $k$
and $\set{Soergel}(\web)$ admits a~left and a~right action by the~algebras
$R^{\underline a}$ and $R^{\underline b}$ respectively.
Furthermore, when $\web$ consists of a~single vertex that is a~merge (resp.\ a~split),
then $\set{Soergel}(\web)$ coincides up to a~grading shift with
the~induction $\mathrm{Ind}^{\underline a}_{\underline b}$
(resp.\ restriction $\Res^{\underline a}_{\underline b}$) bimodule. 
The~results below follow immediately from the~above and the~definition
of the~Soergel space for a~web.

\begin{proposition}
	Let $\web_1$ and $\web_2$ be planar directed webs with
	$\outp(\web_1) = \underline a = \inp(\web_2)$.
	Then
	\[
		\set{Soergel}(\web_1 \circ \web_2)
			\cong
		\set{Soergel}(\web_1) \otimes_{R^{\underline a}} \set{Soergel}(\web_2).
	\]
	In particular, $\set{Soergel}(\web)$ is a~singular Soergel bimodule
	for any planar directed web $\web$.
\end{proposition}

\begin{proposition}
	Let $\widehat\web$ be the~annular closure of a~directed web $\web$. Then
	\(
		\set{Soergel}(\widehat\web)
			\cong
		\HoHom_0(\set{Soergel}(\web)).
	\)
\end{proposition}

\begin{example}
	The~Soergel bimodule associated with the~directed web $\web$
	in Figure~\ref{fig:lamination} is a~quotient of the~tensor product
	\[
		R(\web) = R^{(3,1)} \otimes R^{(4)} \otimes R^{(2,2)}
	\]
	by relations that identify any generator of $R^{(4)}$ with
	its image in either of the~two other factors. Hence,
	taking into account the~overall shift,
	\[
		\set{Soergel}(\web) = \qshift^{-\frac{7}{2}} R^{(3,1)} \otimes_{R^{(4)}} R^{(2,2)}.
	\]
\end{example}

Let us now introduce maps between Soergel spaces that correspond to
the~basic building blocks of foams depicted in Figure~\ref{fig:basic-foams}
(compare \cite{WedrichExp, RW2}). The~first four arise
as the~units and traces of associated graded Frobenius extensions
\cite{EliasMakisumiThielWilliamson}.

\begingroup
\def\arrowto{\quad\longrightarrow\quad}\def\arrowmap{\quad\longmapsto\quad}\def\colonsp{\colon\quad}\makeatletter
\def\insertpict{\@ifnextchar[{\insertpict@}{\insertpict@@[1,1,]}}\def\insertpict@[#1]{\insertpict@@[#1,#1,]}\def\insertpict@@[#1,#2,#3]{\NB{\tikz[xscale=#1,yscale=,font=\tiny]{\input{\imagesfolder/}}}}\makeatother
\begin{align*}
	\intertext{The~cup foam is assigned the~inclusion}
		\bigon\colonsp
			\set{Soergel}\left( \insertpict[1.5]{hfgl0_digon0} \right)
		&\arrowto
			\qshift^{ab} \set{Soergel}\left( \insertpict[1.5]{hfgl0_digon} \right)
	\\
		\insertpict[1.5]{hfgl0_digon0} &\arrowmap \insertpict[1.5]{hfgl0_digon},
\intertext{whereas with the~cap foam we associate the~projection}
		\nogib\colonsp
			\set{Soergel}\left( \insertpict[1.5]{hfgl0_digon} \right)
		&\arrowto
         \qshift^{ab} \set{Soergel}\left( \insertpict[1.5]{hfgl0_digon0} \right)
	\\
		\insertpict[1.5]{hfgl0_digon3} &\arrowmap \insertpict[1.5]{hfgl0_digon4}\quad,
\intertext{where
		$\displaystyle{P\star Q = \sum_{\substack{
				I\sqcup J =  \{1, \dots, a+b\}\\
				\#I = a, \#J =b
			}} \frac{P(x_I)Q(x_J)}{\nabla(x_I, x_J)}}$
		and $\displaystyle{
			\nabla(x_I, x_J) = \prod_{\substack{i \in I \\ j \in J}} (x_j - x_i)
		}$.
	}
\intertext{This map is surjective since the target bimodule is generated by $1$ and $1$ is reached since $e_a^b\star1 =1$.}
\intertext{An~\emph{unzip} is associated with the~projection}
		\piz\colonsp
			\set{Soergel}\left( \insertpict[1.5,1.2]{hfgl0_zip2} \right)
		& \arrowto
			\qshift^{-ab}\set{Soergel}\left( \insertpict[1.5,1.2]{hfgl0_zip1} \right)
	\\[3ex]
			\insertpict[1.5,1.2]{hfgl0_zip2} &\arrowmap \insertpict[1.5,1.2]{hfgl0_zip1}.
\intertext{The~\emph{unzip} is surjective since the target bimodule is generated by $1$ and $1$ is reached.}
	\intertext{A~\emph{zip} is associated with the~inclusion}
		\zip\colonsp
			\set{Soergel}\left( \insertpict[1.5,1.2]{hfgl0_zip1} \right)
		&\arrowto
			\qshift^{-ab}\set{Soergel}\left( \insertpict[1.5,1.2]{hfgl0_zip2} \right)
	\\[3ex]
		\insertpict[1.5,1.2]{hfgl0_zip1} &\arrowmap
			\sum_{\alpha \in T(a,b)} (-1)^{|\widehat\alpha|}\insertpict[1.5,1.2]{hfgl0_zip3},
\intertext{Since the composition of zip followed by the unzip is a multiplication with a given polynomial,
the zip map is injective.}
\intertext{The~\emph{multiplication} by a~homogeneous symmetric polynomial $P$ is the~map}
		m_P\colonsp
			\set{Soergel}\left(\insertpict{hfgl0_strand} \right)
		&\arrowto
			\qshift^{-\deg P}\set{Soergel}\left(\insertpict{hfgl0_strand}\right)
	\\[2ex]
		\insertpict{hfgl0_strand} &\arrowmap \insertpict{hfgl0_strand2}.
\intertext{Finally, the~\emph{associativity} and \emph{coassociativity} foams are assigned the~maps}
     \assoc\colonsp
			\set{Soergel}\left(\insertpict{hfgl0_assoc} \right)
		&\arrowto
			\set{Soergel}\left(\insertpict{hfgl0_assoc2}\right)
		\\
         \insertpict{hfgl0_assoc} &\arrowmap \insertpict{hfgl0_assoc2}
		\\[3ex]
		\coassoc\colonsp
			\set{Soergel}\left(\insertpict{hfgl0_assoc3} \right)
		&\arrowto
			\set{Soergel}\left(\insertpict{hfgl0_assoc3}\right)
		\\
			\insertpict{hfgl0_assoc3} &\arrowmap \insertpict{hfgl0_assoc4}.
\end{align*}\endgroup

Because of the~local nature of the~above definitions, they can be interpreted as maps assigned
to foams between either planar or annular directed webs. It is known that this assignment
preserves local relations.

\begin{proposition}\label{prop:Soergel-tqfts}
	When applied to planar directed webs, the~above describe a~functor of bicategories
	\[
		\set{Soergel}\colon \ccat{Foam} \to \ccat{sSBim}
	\]
	and in case of annular directed webs, a~functor
	\[
		\set{Soergel}\colon \cat{AFoam} \to \cat{grMod}.
	\]
	Finally, there is a~functor
	\[
		\set{SoergelRed}\colon \cat{AFoam}^\basepoint \to \cat{grMod}
	\]
	that assigns the~quotient $\set{SoergelRed}(\web) := \set{Soergel}(\web)/(x_\basepoint)$
	to a~pointed web $\web$, where the~variable $x_\basepoint$ is associated with
	the~marked edge.
\end{proposition}

 \subsection{A~quantum trace deformation of annular foams}
\label{sec:annulus}

Following \cite{BPW} one can show that $\cat{AFoam}$ is equivalent to
the~so-called \emph{horizontal trace} $\hTr(\ccat{Foam})$ of the~bicategory
$\ccat{Foam}$. What it roughly means is that
\begin{itemize}
	\item every annular web is isomorphic to a~web with vertices
	away from a~fixed \emph{section} $\mu := \{*\}\times\R$ of the~annulus $\SxR$,
	to which we refer as the~\emph{trace section},
	\item morphisms are generated by foams that intersect the~\emph{membrane}
	$M := \mu\times[0,1]$ in a~directed web modulo local relations away
	from the~membrane and the~\emph{horizontal trace relation} that
	allows to isotope a~piece of a~foam through $M\ric$.
\end{itemize}
The~horizontal trace can be defined on any bicategory and is functorial \cite{BPW}.
Having such a~description of $\cat{AFoam}$ we can now deform it by replacing
the~horizontal trace relation with its quantum version, which we will now
state more precisely.
Notice first that an~orientation of the~circle $\S\times\{0\}\times\{0\}$
induces a~coorientation of the~trace section $\mu$ and membrane $M\ric$.
Let $\foam\ric$ be an~annular foam $\foam\ric$ that intersects $M$ in a~web $\web$
and consider a~generic admissible ambient isotopy $\phi$ that pushes $M$
according to its coorientation, so that
\begin{itemize}
	\item $\phi(\foam)$ intersects $M$ in a~web $\web'$, and
	\item $M' := \phi(M)$ intersects $M$ only at the~collar,
	where both $M$ and $M'$ coincide.
\end{itemize}
Then $M$ and $M'$ bound a~3-ball $B$ with a~foam $\foam\cap B$ from $\web'$ to $\web$ inside.
The~\emph{quantum horizontal trace relation} states that in this setting
\[
	W = q^{-\deg(\foam \cap B)} \phi(W),
\]
see Figure~\ref{fig:membrane-relation} for an~example.

\begin{figure}[ht]
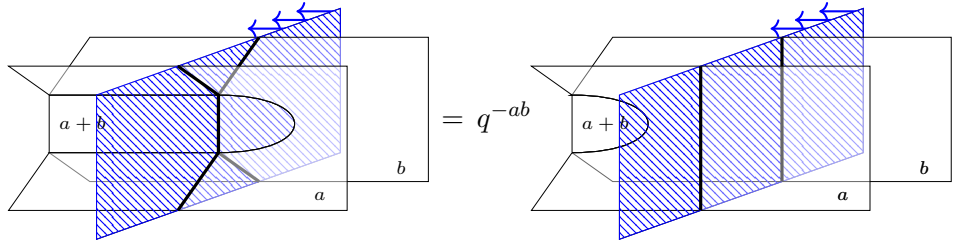

  \centering
  \[
    \NB{\tikz[scale = 1.2]{\input{\imagesfolder/hfgl0_example-membrane}}}
      \, = \,
    q^{-ab} \NB{\tikz[scale=1.2]{\input{\imagesfolder/hfgl0_example-membrane2}}}
  \]
  \caption{The~effect of moving a~foam through the~membrane.
  	The~membrane is depicted in hashed purple with its
  	coorientation indicated by purple arrows.}
  \label{fig:membrane-relation}
\end{figure}

\begin{definition}\label{def:AFoamq}
	The~category $\cat{qAFoam}$ is a~deformation of $\cat{AFoam}$,
	where we consider only annular directed webs that intersect $\mu$ generically,
	whereas on foams we impose the~quantum horizontal trace relations and only local
	relations away from the~membrane $M\ric$.
	We write $\cat{qAFoamEl}$ for its subcategory generated by elementary webs and foams.
\end{definition}

\begin{remark}
	The~quantum trace relation simply identifies a~foam $\foam$ with $\phi(\foam)$
	when $q=1$. Hence, in this case $\cat{qAFoam}$ coincides with $\cat{AFoam}$.
\end{remark}

Propositions~\ref{prop:local-isoms} and \ref{prop:local-elem-isoms} are proven locally,
so that they still hold in the~deformed setting. Likewise, the~quantum trace relation
is enough for Propositions~\ref{prop:annular-reduction}
and \ref{prop:elem-annular-reduction}.

\begin{proposition}
	There is a~functor of categories
	\begin{equation}
		\set{qSoergel}\colon \cat{qAFoam} \to \cat{grMod}
	\end{equation}
	that assigns with an~annular closure $\widehat{\web}$ of a~web $\web$
	the~graded\/ $\scalars$-module $\qHoHom_0(\set{Soergel}(\web))$.
	In particular, there is an~isomorphism
	\begin{equation}\label{eq:qSoergel-cyclicity}
		\set{qSoergel}(\widehat{\web_1\circ\web_2})
			\cong
		\set{qSoergel}(\widehat{\web_2\circ\web_1})
	\end{equation}
	for any webs $\web_1\colon \underline k\to \underline\ell$
	and $\web_2\colon\underline\ell \to \underline k$.
\end{proposition}
\begin{proof}[Sketch of proof]
	The~functoriality of $\qhTr$ provides a~functor
	\[
		\qhTr(\set{Soergel})\colon \cat{qAFoam} \to \qhTr(\ccat{sSBim}).
	\]
	Because $\ccat{sSBim}$ has duals in the sense of 
	\cite[Section 3.8.2]{BPW}, there is a~functor on $\qhTr(\ccat{sSBim})$
	that assigns with a~$(R^{\underline k},R^{\underline k})$-bimodule $M$
	its quantum space of coinvariants.
Combining the~two functors proves the~thesis.
\end{proof}

Let us now unroll the~definition of $\set{qSoergel}$ from the~above proposition.
Pick a~web $\widehat{\web}$ in the~annulus $\SxR$ that intersects generically
the~line $\mu = \{*\}\times\R$. Cutting it along $\mu$ results in a~directed
web $\web$ with $\inp(\web) = \outp(\web) = \underline k$ for some sequence
$\underline k$. To compute $\set{qSoergel}(\widehat\web)$, take the~singular Soergel
bimodule associated with $\web$ and divide it by the~quantum trace relation.
Explicitly, $\set{qSoergel}(\widehat\web)$ is the~$\scalars$-tensor product
\[
	D(\web) = \bigotimes_{e\in E(\web)} \Sym_{\ell(e)}
\]
subjected to the~Soergel relations \eqref{rel:dot-migration}
and the~quantum trace relation
\begin{align*}
	\NB{\tikz[font=\tiny]{\begin{scope}
  \draw[->] (-1,0) -- (1,0) coordinate[pos= 0.25] (p) node[pos = 0,
  left] {$a$}; 
\fill (p) circle (0.5mm) node[above] {$P$};
  \draw[\colormembrane, densely dashed] (0,-0.5) -- (0, 0.5)node[pos = 0,
  below, \colormembrane] {$q$};
\end{scope}}}
		\ =\ 
	q^{-d}\ \NB{\tikz[font=\tiny]{\begin{scope}
  \draw[->] (-1,0) -- (1,0) coordinate[pos= 0.75] (p) node[pos = 1,
  right] {$a$}; 
\fill (p) circle (0.5mm) node[above] {$P$};
  \draw[\colormembrane, densely dashed] (0,-0.5) -- (0, 0.5) node[pos = 0,
  below, \colormembrane] {$q$}; 
\end{scope}
}}
\end{align*}
where $P$ is a~homogeneous symmetric polynomial of degree $d$.

\begin{remark}\label{rmk:bivalent-vertices-in-webs}
	It is worth to think of the~intersection points of $\web$
	with the~trace section $\mu$ as bivalent vertices; we call them
	\emph{trace vertices}.
	The~quantum trace relation is then another type of local relations
	in the~Soergel space. Sometimes we shall also consider bivalent
	vertices at other places, in which case the~local relation simply
	identifies the~variables associated with the edges on both sides
	of such a~vertex.
\end{remark}

In Section~\ref{sec:pointed-annular}, we defined the category $\cat{AFoamB}$ of pointed annular webs. Recall that objects are annular webs with a marking located an~edge of thickness 1 on the outer side. We can repeat the deformation process in this setting.

\begin{definition}\label{def:AFoamq}
	The~category $\cat{qAFoamB}$ is a~deformation of $\cat{AFoamB}$,
	where we consider only pointed annular directed webs that intersect $\mu$ generically and where the marking is on $\mu$.
	On foams we impose the~quantum horizontal trace relations away from the marking and local
	relations only away from the~membrane $M\ric$.
\end{definition}

	Formally speaking, $\cat{qAFoamB}$ is a~quotient of
	a~\emph{partial horizontal trace} of $\ccat{Foam}$.
Notice that since the quantum trace relation is not imposed at the marking, 
it can be thought of as a  pair of endpoints of thickness $1$ of $\web$.

There is a~forgetful functor $\cat{qAFoamB} \to \cat{qAFoam}$,
which allows us to construct a~functor
\[
	\set{qSoergelRed}\colon \cat{qAFoamB} \to \cat{grMod}
\]
that takes a~marked web $\widehat\web$, represented as a~closure of
$\web$, to the~quotient
\[
	\set{qSoergelRed}(\widehat\web) = \qHoHom_0(\set{Soergel}(\web)) / (x_\basepoint),
\]
where $x_\basepoint$ is the~variable associated with the~output edge
of $\web$ chosen by the~marking $\basepoint$.
However, because of the~restricted trace relation in $\cat{qAFoamB}$,
the~cyclicity property \eqref{eq:qSoergel-cyclicity} does not hold
for $\set{qSoergelRed}$ unless in one of the~webs, $\web_1$ or $\web_2$,
the~top most endpoints are connected by an~interval disjoint from the~rest of the~web.

We end this section with a~result about singular Soergel bimodules,
which explains why we take only the~quantum trace to define $\set{qSoergel}$
instead of the~full quantum Hochschild homology.

\begin{theorem}\label{thm:qHH(Soergel)=0}
	Assume that $1-q^d$ is invertible for all $d\neq 0$.
	Then for any sequence $\underline k$ and
	a~bimodule $B\in \ccat{sSBim}(R^{\underline k}, R^{\underline k})$ one has
	\[
		\qHH_i(R^{\underline k}, B) = 0\qquad\text{for } i > 0.
	\]
\end{theorem}
\begin{proof}
	Because singular Sorgel bimodules are direct summands of singular Bott--Samelson bimodules,
	it is enough to prove the~formula only for the~latter. For that notice that every singular
	Bott--Samelson bimodule is of the~form $\set{Soergel}(\web)$ for some directed web $\web$.
	The~thesis follows from Propositions~\ref{prop:annular-reduction} and
	\ref{prop:qHH-vanishes}.
\end{proof}

\section{Combinatorial link homologies}
\label{sec:lh}
In this section we recall the~combinatorial framework for $\HHH\!$, $\HHH^{red}$,
as well as the~$\gll_1$ and $\gll_0$ homologies constructed by the last two authors. 

Throughout this section we fix
a link $L$  represented by an~annular closure $\bdiagram$
of a~braid diagram $\braid$, drawn horizontally from left to right
and closed below the~braid, see Figure~\ref{fig:example-resolution}.
We write $\bindex$ for the~number of strands of $\bdiagram$ or its index and $\setcrossings$ 
for the~set of crossings of $\braid$. The~latter consists of
of and $n_+$ positive $n_-$ negative crossings, and we write
$n = n_+ + n_-$ for the~total.

\subsection{The~general cube construction}
\label{sec:cube-of-resolutions}

A~map $\resolution \co \setcrossings \to \{0,1\}$ determines a~directed annular
web $\rdiagram$, called the~\emph{$I$-resolution} of $\bdiagram$, constructed
by replacing locally each crossing by either its smoothing or singularization
as indicated in Figure~\ref{fig:resolutions}.
An~example is given in Figure~\ref{fig:example-resolution}.
\begin{figure}[ht]
  \begin{tikzpicture}[font=\small]
  	\node (S0) at ( 3,0) {\NB{\tikz[]{%
%
\begin{scope}[xscale=1.25]
	\draw[->] (-0.5, 0.5) .. controls (0,0) .. (0.5, 0.5);
	\draw[->] (-0.5, -0.5) .. controls (0,0) .. (0.5, -0.5);
\end{scope}}}};
  	\node (Sx) at (-3,0) {\NB{\tikz[]{\input{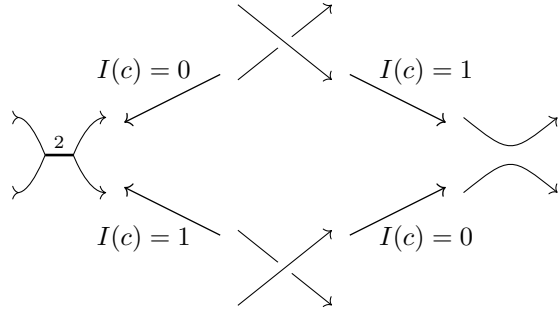}}}};
  	\node (S+) at (0, 1.5) {\NB{\tikz[]{%
%
\begin{scope}[xscale=1.25]
	\draw[->] (-0.5, -0.5) -- (0.5,0.5);
	\fill[white] (0,0) circle[radius=3pt];
	\draw[->] (-0.5, 0.5) -- (0.5,-0.5);
\end{scope}}}};
  	\node (S-) at (0,-1.5) {\NB{\tikz[]{%
%
\begin{scope}[xscale=1.25]
	\draw[->] (-0.5, 0.5) -- (0.5,-0.5);
	\fill[white] (0,0) circle[radius=3pt];
	\draw[->] (-0.5, -0.5) -- (0.5,0.5);
\end{scope}}}};
  	\begin{scope}[line width=0.5pt,shorten <=3pt, shorten >=3pt, inner sep=1pt]
  		\path
	  		(S+) edge[->] node[pos=0.3, above right] {$I(c)=1$} (S0)
	  	         edge[->] node[pos=0.3, above left]  {$I(c)=0$} (Sx)
	  	    (S-) edge[->] node[pos=0.3, below right] {$I(c)=0$} (S0)
	  	         edge[->] node[pos=0.3, below left]  {$I(c)=1$} (Sx);
	\end{scope}
  \end{tikzpicture}
  \caption{The~two resolutions of a~crossing $c$: its \emph{singularization}
	  (to the~left) and \emph{smoothing} (to the~right).
  }	 
  \label{fig:resolutions}
\end{figure}\begin{figure}[ht]
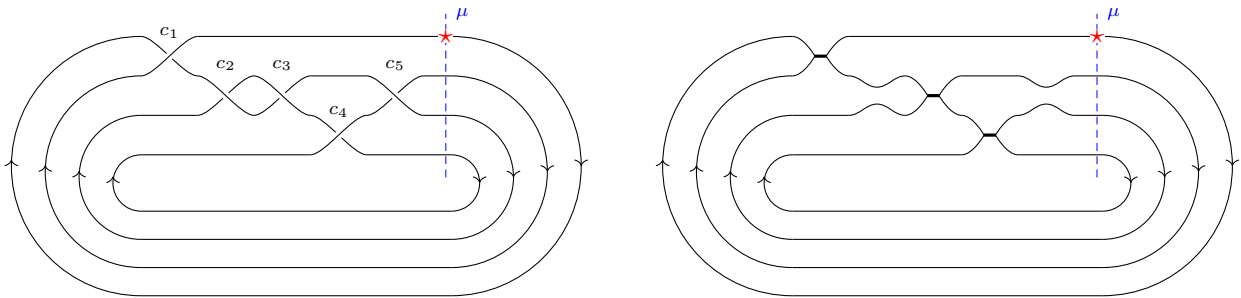

  \centering
  \NB{\tikz[scale=0.75]{\input{\imagesfolder/hfgl0_example-braid-no-vertices}}}
  \qquad
  \NB{\tikz[scale=0.75]{\input{\imagesfolder/hfgl0_example-resolution-diag-no-bivalent}}}
  \caption{A~braid diagram of the~closure of
  	$\braid = \sigma_1^{-1} \sigma_2^2 \sigma_3^{-1} \sigma_2$
  	and its resolution associated with
  	$(\resolution(c_i))_{1\leq i \leq 5} = (0,0,1,0,0)$.
	The~dashed line $\mu$ visualizes the~trace section
	and the~red star is the~position of the~marking.
  }
  \label{fig:example-resolution}
\end{figure}Two resolutions $\resolution$ and $\resolution'$ are \emph{neighboring}
if they agree on all but one crossing $c$, in which case we write
$\resolution \stackrel{c}{\longrightarrow} \resolution'$
if $\resolution(c) = 0$ and $\resolution'(c) = 1$.
With such two neighboring resolutions we associate a~foam
$\foam_{\resolution,c}\co \bdiagram_\resolution \to \bdiagram_{\resolution'}$,
which is an~unzip in case the~crossing $c$ is positive and a~zip otherwise.
These data can be arranged into a~commuting diagram in $\cat{AFoam}$ as follows.
Given a~square of neighboring resolutions
\begin{equation}\label{eq:cube-facet}
	\begin{tikzpicture}[x=15mm,y=1cm,baseline=(00.base)]
		\node (00) at (0, 0) {$\resolution_{00}$};
		\node (01) at (1, 1) {$\resolution_{01}$};
		\node (10) at (1,-1) {$\resolution_{10}$};
		\node (11) at (2, 0) {$\resolution_{11}$};
		\path[inner sep=2pt]
			(00) edge[->] node[above left] {$\scriptstyle c$} (01)
				 edge[->] node[below left] {$\scriptstyle c'$} (10)
			(11) edge[<-] node[above right] {$\scriptstyle c'$} (01)
				 edge[<-] node[below right] {$\scriptstyle c$} (10);
	\end{tikzpicture}
\end{equation}
the~foams $\foam_{\resolution_{00},c} \cup \foam_{\resolution_{01},c'}$ and
$\foam_{\resolution_{00},c'} \cup \foam_{\resolution_{10},c}$
coincide upto an~ambient isotopy.
Choose a~sign $\epsilon(\resolution,c) = \pm1$ whenever $\resolution(c)=0$,
so that
\[
	\epsilon(\resolution_{00},c) \epsilon(\resolution_{01},c') +
	\epsilon(\resolution_{00},c') \epsilon(\resolution_{10},c) = 0
\]
for each situation as in \eqref{eq:cube-facet}.
For instance, one can take $\epsilon(\resolution,c) := (-1)^{I_{\prec c}}$,
where $\resolution_{\prec c} := \sum_{c'\prec c}\resolution(c')$ for
a~fixed total ordering $\prec$ on $\setcrossings$.
Let $\Bracket{\bdiagram}$ be a~formal chain complex in $\cat{AFoam}$
supported in homological degrees $[-n_-,n_+]$, given by objects
\begin{align*}
	\Bracket{\bdiagram}_i &:= 
		\bigoplus_{|\resolution|=i+n_-} \qshift^{-i} \bdiagram_\resolution
\intertext{and the~differential}
	\partial_i &:= \sum_{\substack{
		|\resolution|=i+n_-\\
		c: \resolution(c)=0
	}}
		\epsilon(\resolution, c) W_{\resolution, c}.
\end{align*}
A~standard argument ensures that the isomorphism type of $\Bracket{\bdiagram}$
does not depend on the~choice of signs, see \cite{OddKhovanov, MR3363817}.
Clearly, $\cat{AFoam}$ can be replaced with its quantization $\cat{qAFoam}$.

\begin{theorem} 
	The~homotopy type of the formal complex $\Bracket{\bdiagram}$,
	computed	 either over $\cat{AFoam}$ or $\cat{qAFoam}$,
	is an~invariant of the annular link $L=\bdiagram$.
\end{theorem}

\begin{remark}
	When $\bdiagram$ is considered as a~link with a~marking,
	placed at the~top trace vertex, then $\Bracket{\bdiagram}$
	is a complex over $\cat{AFoam}^\basepoint$ or
	$\cat{qAFoam}^\basepoint$.
	However, one can only prove its invariance under braid moves
	and the~second Markov move (conjugacy) away from $\basepoint$.
	This is not enough for the~bracket to be
	an~invariant of links with marked components.
\end{remark}

Occasionally it will be worth to consider partial resolutions of a~link.
These are a~special type of diagrams for knotted elementary webs, in which
only thin edges cross themselves.
Following \cite{OSSSingular, OSCube} we refer to such webs as
\emph{singular links} and to the~thick edges---\emph{singular crossings}.
Although we are not using this in that paper, let us mention that every singular link admits a~diagram in a~braid position \cite{MR1191478}.

Following Gilmore \cite[Section 2]{Gilmore}, we shall also extend webs to allow \emph{bivalent vertices},
depicted  as short tags,
with the~obvious local relation that identifies
variables associated with incident edges. Notice that a trace vertex is a particular kind of a bivalent vertex where the identification of variables is composed with a multiplication by some power of $q$ 
(compare Remark~\ref{rmk:bivalent-vertices-in-webs}).
Subdividing an~edge with a~bivalent vertex is called an~\emph{insertion}
\cite{ManolescuCube}. Notice that insertion does not change the~isomorphism type
of the~associated Soergel space.
We say that a~singular link diagram $S$ is \emph{layered}
if it is in a~braid position and there are vertical lines
$\ell_0, \dots, \ell_n$ called \emph{levels},
with $\ell_n = \mu$ the~trace section,
that carry all singular crossings and bivalent vertices of $S$
and which intersect $S$ only at those points,
see Figure~\ref{fig:example-layered-diagram}.
In order to keep all resolutions of $S$ layered,
we require that real crossings are at the~lines $\ell_i$,
which results in a~bivalent vertex at the~over- and at the~underpass.
Any diagram in a~braid position can be modified to a~layered one
by a~sequence of insertions.

\begin{figure}[ht]
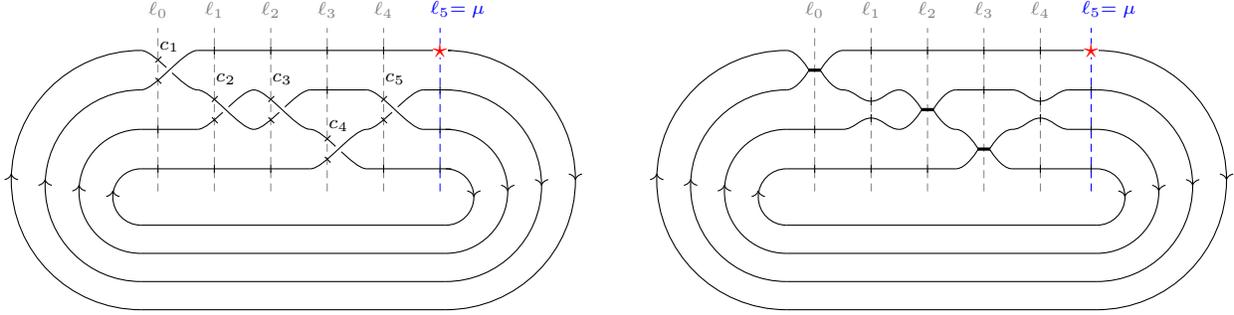

  \centering
  \NB{\tikz[scale=0.75]{\input{\imagesfolder/hfgl0_example-braid-diag}}}
  \qquad
  \NB{\tikz[scale=0.75]{\input{\imagesfolder/hfgl0_example-resolution-diag}}}
  \caption{Layered diagrams of the~braid closure and it resolution
  	from Figure~\ref{fig:example-resolution}.
  	In order to increase the~visibility of bivalent vertices,
  	the~real crossings in the~left picture are moved slightly
  	to the~right with respect to the~vertical lines $\ell_i$.
  }
  \label{fig:example-layered-diagram}
\end{figure}

 \subsection{Triply graded Khovanov--Rozansky homology}
\label{sec:HHH}
Consider the~functor $\set{Soergel}\colon \ccat{Foam} \to \ccat{sSBim}$
from Proposition~\ref{prop:Soergel-tqfts}. Given a~resolution
$\braid_\resolution$ we assign $\HoHom(\set{Soergel}(\braid_\resolution))$
to the~annular closure $\bdiagram_\resolution$.
This results in a~functor on $\cat{AFoam}$ that,
when applied to the~formal complex $\Bracket{\bdiagram}$,
produces a~complex of bigraded modules that computes
the~triple graded homology $\HHH\!$, see \cite{MR2339573}.
More precisely, writing $k$  
for the~number of strands  of $\beta$ with $\bdiagram=L$
we define $\HHH(L)$ as the~homology of the~complex of bigraded modules
\begin{equation}\label{eq:HHH-unreduced}
	C_i = \bigoplus_j \qshift^{k-2j}\HoHom_{\!j}(
		\set{Soergel}_{i+j}(\braid)
	),
\qquad\text{where}\quad
	\set{Soergel}_i(\braid) = \bigoplus_{|\resolution| = i + n_-}
		\qshift^{-i} (\set{Soergel}(\braid_\resolution))
\end{equation}
and the~extra grading is defined by putting $j$th Hochschild homology
in degree $2j - w - k$, where  $w = n_+ - n_-$ is the writhe of $\beta$.
The~differential is induced by the~zip and unzip foams,
which can be written diagrammatically as follows:
\begin{equation}\label{eq:zip-unzip-in-complex}
\begin{aligned}
  \piz \left(\NB{\tikz[]{\begin{scope}[font = \tiny]
  \begin{scope}[xscale=1.5]
    \draw[>->] (-0.5, -0.5) .. controls +(0.2, 0) and +(0,0) .. (-0.2, 0)
      node[pos=0, left] {$1$} -- (0.2, 0)
      node [midway, above] {$2$} .. controls +(0,0) and +(-0.2, 0) .. (0.5,-0.5)
      node[pos=1, right] {$1$};
    \draw[>->] (-0.5, 0.5) .. controls +(0.2, 0) and +(0,0) .. (-0.2, 0)
      node[pos=0, left] {$1$} -- (0.2, 0)
      node [midway, above] {$2$} .. controls +(0,0) and +(-0.2, 0) .. (0.5,0.5)
      node[pos=1, right] {$1$};
  \end{scope}  
\end{scope}

}} \right)
	&\quad=\quad \NB{\tikz[]{\begin{scope}[font = \tiny] 
  \begin{scope}[xscale=1.5]
    \draw[->] (-0.5, -0.5) -- (0.5,-0.5)
      node[pos=0, left] {$1$}
      node[pos=1, right] {$1$};
    \draw[->] (-0.5, 0.5) -- (0.5, 0.5)
      node[pos=0, left] {$1$}
      node[pos=1, right] {$1$};
  \end{scope}
\end{scope}

}}
\\
  \zip \left(\NB{\tikz[]{}} \right)
    &\quad=\quad \NB{\tikz[]{\begin{scope}[font = \tiny]
  \begin{scope}[xscale=1.5]
    \draw[>->] (-0.5, -0.5) .. controls +(0.2, 0) and +(0,0) .. (-0.2,
    0) node[pos=0, left] {$1$} -- (0.2, 0) node [midway, above] {$2$} .. controls +(0,0) and +(-0.2, 0) .. (0.5,-0.5) node[pos=1, right] {$1$};
    \draw[>->] (-0.5, 0.5) .. controls +(0.2, 0) and +(0,0) .. (-0.2,
    0)  node[pos=0.3, font =\normalsize] {$\bullet$} node[pos=0, left] {$1$} -- (0.2, 0) node [midway, above] {$2$} .. controls +(0,0) and +(-0.2, 0) .. (0.5,0.5) 
    node[pos=1, right] {$1$};
  \end{scope}  
\end{scope}

}} - \NB{\tikz[]{\begin{scope}[font = \tiny]
  \begin{scope}[xscale=1.5]
    \draw[>->] (-0.5, -0.5) .. controls +(0.2, 0) and +(0,0) .. (-0.2,
    0) node[pos=0, left] {$1$} -- (0.2, 0) node [midway, above] {$2$} .. controls +(0,0) and +(-0.2, 0) .. (0.5,-0.5) node[pos=0.7, font =\normalsize] {$\bullet$} node[pos=1, right] {$1$};
    \draw[>->] (-0.5, 0.5) .. controls +(0.2, 0) and +(0,0) .. (-0.2,
    0)  node[pos=0, left] {$1$} -- (0.2, 0) node [midway, above] {$2$} .. controls +(0,0) and +(-0.2, 0) .. (0.5,0.5) 
    node[pos=1, right] {$1$};
  \end{scope}  
\end{scope}

}}
\end{aligned}
\end{equation}
and it preserves both the~quantum and (modified) Hochschild grading.
The~degree shifts are chosen so that the~homology is invariant under stabilizations.

\subsubsection*{Middle homology}
It is known that the~triple graded homology splits into two copies
of \emph{middle homology} $\HHH^{\textrm{mid}}$, shifted in homological
and Hochschild degrees.
This splitting follows easily when the~Hochschild homology of Soergel
bimodules $\set{Soergel}(\braid_\resolution)$ is computed using the~Koszul
complex
\begin{equation}\label{eq:Koszul-resolution}
	K(\braid_\resolution) := \set{Soergel}(\braid_\resolution)
		\otimes_{R^e}
	\bigotimes_{i=1}^k (\qvar^2 R^e \xrightarrow{\ x_i\otimes 1 - 1\otimes x_i\ } R^e)
\end{equation}
where $R = \scalars[x_1,\dots,x_k]$ and $R^e = R\otimes R$ is the~enveloping algebra
with the~left (resp.\ right) factor acting from the~left (resp.\ right)
on $\set{Soergel}(\braid_\resolution)$.
After changing the~basis of the~algebra by trading $x_1$ for $e_1$,
the~first elementary symmetric polynomial,
the~differential in the~factor associated with $e_1$ vanishes,
because the~polynomial acts symmetrically. Hence,
$K(\braid_\resolution) = K'(\braid_\resolution) \oplus \avar^2 \qvar^2 K'(\braid_\resolution)$,
where $K'(\braid_\resolution)$ is defined as in \eqref{eq:Koszul-resolution}
except that the~big tensor product is taken for $i>1$.
The~middle homology $\HHH^{\textrm{mid}}$ arises from a~complex defined
as in \eqref{eq:HHH-unreduced}, except that the~homology of $K'$ is taken
instead of $\HoHom\!$.

\subsubsection*{Reduced homology}
\label{sec:reduced}
The~above constructions can be repeated with the~reduced bimodule $\set{SoergelRed}$.
Taking the~full Hochschild homology as in \eqref{eq:HHH-unreduced} produces
an~invariant link homology $\HHH'(L)$ that categorifies $(\avar - \avar^{-1})P_L(\avar, \qvar)$.
The~complex again splits into two copies of a~smaller complex
$C^{\textrm{red}}(L)$ that computes the~\emph{reduced homology}
$\HHH^{\textrm{red}}(L)$.
In order to see this, observe that $\set{SoergelRed}$ is isomorphic
to the~submodule $\set{Soergel}' \subset \set{Soergel}$ generated by differences of variables;
the~variable $x_i \in R$ acts on $\set{Soergel}'$ by multiplication with $x_i - x_\basepoint$.
This submodule is isomorphic to the~reduced Soergel bimodule as defined
in \cite{MR3611725} if 2 is invertible.\footnote{
	In the~original definition, the~reduced Soergel bimodule is
	generated  by the differences of variables associated
	with edges at the~same level, which is only enough to generate
	twice the~difference of any two variables.
}
Furthermore, when $k$ is invertible in $\scalars$, then
\[
	\scalars[x_1,x_2,\dots,x_k]
\cong
	\scalars[e_1, x_2-x_1, x_3-x_2, \dots, x_k - x_{k-1}].
	\]
	For instance, one has $ k x_k= e_1+\sum_{i=1}^{k-1}i(x_{i+1}-x_{i})$.
Hence, reducing to $K'$ coincides with taking Hochschild homology
with respect to the~subalgebra $R'\subset R$ generated by differences
of variables. Notice that the~identification $\set{SoergelRed} \cong \set{Soergel}'$
commutes with the~natural action of $R'$ on both bimodules.
This is how the~reduced homology was originally defined
\cite{MR3611725, MR2339573, MR2421131}.
Its Poincar\'e polynomial
\[P_L(\tvar, \avar, \qvar):=\sum_{i,j,n} \tvar^i \avar^j \qvar^n \dim {\HHH^{\textrm{red}}}^{i,j,n}(L)\]
is a new link invariant, where ${\HHH^{\textrm{red}}}^{i,j,n}(L)$ is the~homology in $i$th homological,
$j$th (modified) Hochschild and $n$th quantum degree.
This construction categorifies the~HOMFLY--PT polynomial
$P_L(\avar,\qvar)$ in the~sense that for any link $L$
there is an~equality of power series in $\qvar$
\begin{equation}\label{eq:HHH-to-HOMFLYPT}
P_L(\avar,\qvar)
		= P_L(-1, \avar, \qvar).
\end{equation}

\begin{remark}
	In order to recover Rasmussen's grading $(i',j',n')$,
	the~(modified) Hochschild grading $j$ must be negated
	and the~homological degree replaced with $i' = 2i + j$;
	the~quantum grading $n$ is not changed.
	This makes the~Hochschild differential homogeneous of $(i',j',n)$-degree
	$(0,2,2)$ and the~topological (induced by foams) differential of degree
	$(2,0,0)$. We also have
	\[
		P_L(\tvar, \avar, \qvar) = P_L^{Ras}(\tvar^{1/2}, \tvar^{1/2}\avar^{-1}, \qvar),
	\]
	where $P_L^{Ras}(\tvar, \avar, \qvar)$
	is the~Poincar\'e polynomial of $\HHH^{\textrm{red}}$ with respect to 
	Rasmussen's convention.
\end{remark}
 \subsection{\texorpdfstring{$\gll_1$}{gl(1)} homology}

The technology developed here was first introduced in \cite{RW2}
using \emph{foams} in a more general framework. It was recasted in
\cite{RW3} in a foam-free framework. Here we use this latter point of
view to recall the construction. Unless stated otherwise, in this
section we work with integral coefficients.

With a~web $\web$ we have associated in Section~\ref{sec:foam-functor}
the~\emph{space of decorations} $D(\web) = \bigotimes_{u\in E(\web)} R_u$,
where the~\emph{edge ring} $R_u$  consists of symmetric polynomials
in as many variables as the~thickness of the~edge $u$. A~pure tensor from
$D(\web)$ is visualized by dots on $\web$, see Figure~\ref{fig:decorated-web}.
In what follows we will consider quotients and subquotients of $D(\web)$.

\begin{definition}
\label{defn:coloring}
	Let $\web$ be an annular web of index $k$.
	Denote by $\mathcal{P}(\{X_1,\dots, X_k\})$ the power set
of $\{X_1, \dots, X_k\}$.
	An~\emph{omnichrome coloring} of $\web$ is a~map 
	$c\colon E(\web) \to \mathcal{P}(\{X_1,\dots, X_k\})$, such that
	\begin{itemize}
		\item for each edge $u \in E(\web)$ the~cardinality of $c(u)$ equals
		the~thickness of $u$,
		\item given a~generic section $r$ of the~annulus, the~union of the~sets
		$c(u)$ for all edges $u$ intersecting $r$ is equal to
		$\{X_1,\dots, X_k\}$, and
		\item the~\emph{flow condition} holds: if $u_1$, $u_2$ and $u_3$
		are three adjacent edges with $u_1$ the thickest of them,
		then $c(u_1)= c(u_2)\sqcup c(u_3)$.
  \end{itemize}
  The set $c(u)$ is called the \emph{color of $u$.}
\end{definition}

The~definition of omnichrome colorings has several direct implications.
\begin{enumerate}
  \item At each vertex of $\web$, the~color of the~thickest edge is
  the~disjoint union of the~colors of the~two thin edges.
  \item For a~generic section $r$ of the~anulus, the~union of sets
  $c(u)$ associated with the~edges $u$ that intersect $r$ is actually
  a~disjoint union.
  \item\label{it:varphi}
  Each coloring $c$ induces an~algebra homomorphism
  $\varphi_c\colon \set{decoration}(\web) \to \ZZ[X_1, \dots X_k]$ 
  that for  each $u \in E(\web)$ identifies the~ring $R_u$ with the~subring
  $\ZZ[c(u)]^{\symm_{\ell(u)}}$. \end{enumerate}

Let $\web$ be an annular web and $c$ be an omnichrome coloring of
$\web$. For each split vertex $v$, denote by $u_l(v)$ and $u_r(v)$
the left and right edges going out of $v$. Set
\[
	Q(\web,c) :=
	\prod_{\substack{v \textrm{ split} \\ \textrm{vertex}}}
	\prod_{\substack{X_i \in c(u_l(v)) \\ X_j \in c(u_r(v))}} (X_i - X_j).
\]
Given a~pure tensor $T \in D(\web)$ write $T_u$ for the~factor associated
with an~edge $u$. Using the ring morphism $\varphi_c$ defined in point~(\ref{it:varphi}) just above, we set:
\[
  	P(\web,T,c) = \varphi_c(T) = \prod_{u \in E(\web)} T_u(c(u))
\]
and extend it linearly to all elements of $\set{decoration}(\web)$.
Finally, define a rational function
\[
    \bracket{\web,T,c}_\infty = \frac{P(\web,T,
    c)}{Q(\web, c)} \in \mathbb{Q}(X_1,\ldots,X_k).
\]

\begin{example}
  Consider the~omnichrome coloring $c$
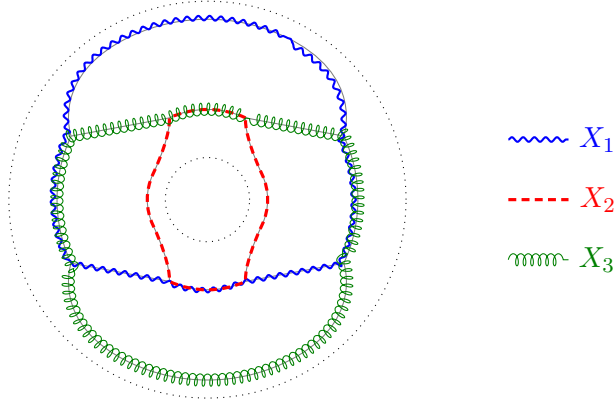
  \[
    \NB{\tikz[scale = 0.8]{\begin{scope}[font = \tiny]
  \draw[dotted] (0,0) circle (0.7cm);
  \draw[dotted] (0,0) circle (3.3cm);

  \coordinate (aA1) at (65:1.5);
  \coordinate (aA2) at (65:3);
  \coordinate (bA1) at (115:1.5);
  \coordinate (bA2) at (115:3);
  \coordinate (aC1) at (245:1.5);
  \coordinate (aC2) at (245:3);
  \coordinate (bC1) at (295:1.5);
  \coordinate (bC2) at (295:3);
  \coordinate (aB1) at (155:1);
  \coordinate (aB2) at (155:2.5);
  \coordinate (bB1) at (205:1);
  \coordinate (bB2) at (205:2.5);
  \coordinate (aD1) at (335:1);
  \coordinate (aD2) at (335:2.5);
  \coordinate (bD1) at (25:1);
  \coordinate (bD2) at (25:2.5);
\begin{scope}[gray]
  \draw  (aA1) arc (65:115:1.5);
  \draw  (aA2) arc (65:115:3);
  \draw  (aC1) arc (245:295:1.5);
  \draw  (aC2) arc (245:295:3);

  \draw  (aB1) arc (155:205:1);
  \draw  (aB2) arc (155:205:2.5);
  \draw  (aD1) arc (-25:25:1);
  \draw  (aD2) arc (-25:25:2.5);

  \draw (bA2) .. controls +(205:1.5) and +(0,0).. (aB2);
  \draw (bC2) .. controls +(25:1.5) and +(0,0) .. (aD2);
  \draw (bB1) .. controls +(295:0.75) and +(0,0) ..(aC1);
  \draw (bD1) .. controls +(115:0.75) and +(0,0) ..(aA1);
  \draw (bA1) -- (aB2);
  \draw (bB2) -- (aC1);
  \draw (bC1) -- (aD2);
  \draw (bD2) -- (aA1);

  \draw (bD2) .. controls +(0,0) and +(335:1.5) .. (aA2);
  \draw (bB2) .. controls +(0,0) and +(155:1.5) .. (aC2);
  \draw (bA1) .. controls +(0,0) and +(65:0.75) ..(aB1);
  \draw (bC1) .. controls +(0,0) and +(245:0.75) ..(aD1);

\end{scope}

      \draw[blue, decorate, decoration={snake, segment length=1.5mm,
        amplitude=0.3mm}, thick]
      (aA2) arc (65:115:3)
      .. controls +(205:1.5) and +(0,0) .. (aB2)
      arc (155:205:2.5)
      -- (aC1)
      arc (245:295:1.5)
      -- (aD2)
      arc  (-25:25:2.5)
      .. controls +(0,0) and +(335:1.5) .. (aA2);
      \draw[very thick, red, densely dashed]
      (aA1) arc (65:115:1.5)
      .. controls +(0,0) and +(65:0.75) .. (aB1)
      arc (155:205:1)
      .. controls +(295:0.75) and +(0,0) ..(aC1)
      arc (245:295:1.5)
      .. controls +(0,0) and +(245:0.75) .. (aD1)
      arc (-25:25:1)
      .. controls +(115:0.75) and + (0,0) .. (aA1);
      \draw[green!50!black, decorate, decoration={coil, segment
        length=1mm, amplitude=0.8mm}]
      (aA1) arc (65:115:1.5) -- (aB2)  arc (155:205:2.5)
.. controls +(0,0) and +(155:1.5) .. (aC2)  arc (245:295:3)
.. controls +(25:1.5) and +(0,0) .. (aD2) arc (-25:25:2.5) -- (aA1);
\end{scope}

\begin{scope}[xshift = 5cm]
  \draw[blue, decorate, decoration={snake, segment length=1.5mm,
    amplitude=0.3mm}, thick]
  (0,1) -- +(1,0) node [pos =1 , right] {$X_1$};
  \draw[very thick, red, densely dashed]
  (0,0) -- +(1,0) node [pos =1 , right] {$X_2$};
  \draw[green!50!black, decorate, decoration={coil, segment
        length=1mm, amplitude=0.8mm}]
  (0,-1) -- +(1,0) node [pos =1 , right] {$X_3$};
\end{scope}}}
  \]
  of the~decorated annular web $(\web,T)$ from Figure~\ref{fig:decorated-web}.
  We compute
  \begin{align*}
    P(\web, T, c) &= X_2^2X_1X_3,\\
    Q(\web, T, c) &= (X_3-X_1)(X_2-X_3)(X_1-X_3)(X_2 - X_1) \\
  \intertext{so that}
    \bracket{\web,T,c}_\infty &= \frac{X_2^2X_1X_3}{(X_3-X_2)(X_2-X_1)(X_3-X_1)^2}.
  \end{align*}
\end{example}

\begin{definition}
\label{defn:evaluations-infty}
  For an~annular web $\web$ and a decoration $T \in D(\web)$ 
   define the ~\emph{$\infty$-evaluation} $ \bracket{\web,T}_\infty$ of $T$ by
  \[
      \sum_{\substack{c\,:\,\textrm{omnichrome} \\ \textrm{coloring}}}
      	\bracket{\web,T,c}_\infty \in \mathbb{Q}(X_1,\ldots,X_k)
  \]
  and the~\emph{$\infty$-pairing} is the~bilinear form
  $\bracket{ -; \web; -}_\infty$  on $\set{decoration}(\web)$,
  defined on decorations $S$ and $T$ as
  $\bracket{S;\web; T}_\infty := \bracket{\web, ST}_\infty$.
  The~\emph{$\gll_\infty$ state space of $\web$} is the quotient
  \[
  	\glinfty(\web) := \quotient{D(\web)}{\mathrm{rad}\bracket{-; \web; -}_\infty}.
  \]  
  For another ring of coefficients $\scalars$ we set
  $\glinfty(\web, \scalars) := \glinfty(\web) \otimes_\ZZ \scalars$.
\end{definition}

\begin{proposition}
  Choose an~annular web $\web$ of index $k$.
  \begin{enumerate}
	\item
	For any $T \in D(\web)$, the~$\infty$-evaluation $\bracket{\web,T}_\infty$ 
	is a~symmetric polynomial in $X_1, \dots, X_k$ with coefficients in $\mathbb{Z}$.
	\item The~defining ideals for $\scalars$-modules
	$\glinfty(\web, \scalars)$ and $B(\web)$
	 as quotients of $D(\web)$ agree. 
In particular, the~Soergel relations \eqref{rel:dot-migration}
	hold in $\glinfty(\web, \scalars)$.
 \end{enumerate}
\end{proposition}
\begin{proof}
	The~first statement is the~content of \cite[Lemma 3.13]{RW3}
	and the~second one follows directly from \cite[Proposition 4.18]{RW2},
	because $B(\web)$ coincides with $\HoHom_0(B(\widetilde\web))$
	when $\web$ is a~closure of a~directed web $\widetilde\web$.
\end{proof}

\begin{definition} \label{defn:evaluations-gl1}
	Choose an~annular web $\web$ of index $k$.
    Define the~\emph{$\gll_1$ evaluation} of $T \in D(\web)$ by
    \[
      \bracket{\web,T}_1 := (\bracket{\web, T}_\infty)_{
      	\rule[-0.3ex]{0.5pt}{2.5ex}\,X_1,\dots, X_k \mapsto 0}.
    \]
    In other words, $\bracket{\web,T}_1$ is the~constant coefficient
    of $\bracket{\web,T}_\infty$.
    The~\emph{$\gll_1$ pairing} on $\web$ is the~bilinear form
    $\bracket{-; \web; -}_1$ on $\set{decoration}(\web)$ defined
    on decorations $S$ and $T$ by
    $\bracket{S;\web; T}_1 := \bracket{\web, ST}_1$.
    The~\emph{$\gll_1$ state space of $\web$} is the~quotient
    \[
    		\glone(\web) := \quotient{D(\web)}{\mathrm{rad}\bracket{-; \web; -}_1}.
    \]
    For another ring of coefficients $\scalars$ we set
    $\glone(\web, \scalars) = \glone(\web) \otimes_\ZZ \scalars$.
\end{definition}

Following its very definition $\glone(\web)$ is a~quotient of $B(\web)$.
\begin{proposition}\label{proposition:morphism-gl1}\ 
  \begin{enumerate}
  \item \label{it:functor-gl1}
    The~assignment
    $\web\mapsto \glone(\web)$ extends to a~functor
    $\glone\colon \cat{AFoam} \to \cat{grMod}$ that is
    a~quotient of the~functor from Section~\ref{sec:foam-functor}.  In
    particular, the~isomorphisms from
    Proposition~\ref{prop:local-isoms} induce isomorphisms between
    $\gll_1$ state spaces.
      \item \label{it:rank-gl1} $\glone(\web)$ is a~free graded\/
        $\aring$-module for any web $\web$.  It has rank 1 and is
        concentrated in quantum degree $0$ in case $\web$ is
        a~collection of circular webs.
      \item \label{it:annihilation-gl1}Suppose that a~generic section
        of the~annulus intersects edges $u_1,\dots,u_s$ of a~annular
        web $\web$ and let $P\in D(\web)$ represent a homogeneous
        symmetric polynomial in variables
        $X_{u_1}\sqcup\dots\sqcup X_{u_s}$ of positive degree. Then
        $P$ annihilates $\glone(\web)$.
  \end{enumerate}
\end{proposition}

\begin{proof}
  When $\aring=\ZZ$, \eqref{it:functor-gl1} was first proven in
  \cite[Section 5.1.2]{RW2} and then reformulated in a foam free
  language in \cite[Section 3.4]{RW3}; \eqref{it:rank-gl1} is the
  content of \cite[Example 3.25]{RW2}. Both statements are obtained from the results over $\ZZ$  after
  tensoring with $\aring$. Statement
  \eqref{it:annihilation-gl1} follows directly from the definition
  of $\bracket{\cdot, \cdot}_1$.
\end{proof}

Applying the~functor $\glone(-)$ to the~formal complex $\Bracket{\bdiagram}$
results in a~chain complex of $\aring$-modules
$\complexglone(\bdiagram; \aring)$ with homology denoted by
$\homologyglone(\bdiagram; \aring)$;
we call it the~\emph{$\gll_1$ homology} of $\bdiagram$.
The~differential is induced by the~zip and unzip maps
listed in \eqref{eq:zip-unzip-in-complex}.

\begin{theorem}[{\cite{RW2}}]\label{thm:hglzero-inv}
  If\/ $\aring$ is a field of characteristic $0$,
  then $\homologyglone(\bdiagram; \aring)$ is a~link invariant.
  Its graded Euler characteristic is $1$ for every link.
\end{theorem}

The construction in \cite{RW2} is done in an~equivariant
setting and over $\mathbb{Q}$. Here we consider a~simpler
non-equivariant setting, in which case the~construction
can be performed with integral coefficients. The~proof of invariance,
however, requires inverses of nonzero integers,
see \cite[Lemma~5.21]{RW2} and \cite[Lemma~5.25]{RW2}.

This invariant can  easily be extended to links colored by arbitrary
positive integers. The~setup described here corresponds to the~case
where all components are colored by $1$, known as the~uncolored case.

 \subsection{\texorpdfstring{$\gll_0$}{gl(0)} homology}

In this section we assume that the braid $\braid$ 
of index $k$ is chosen 
in such a way that its closure is a knot $K$.
Consider the chain complex $\complexglone(\bdiagram; \aring)$.
Having picked a~basepoint $\basepoint$ on $\bdiagram$,
one defines an~endomorphism $\varphi_\basepoint$ of
$\complexglone(\bdiagram; \aring)$ that multiplies the~decoration
of the~marked edge by $x^{\bindex-1}$. Diagrammatically, this reads:
\[
\NB{\tikz[]{
\begin{scope}
  \draw[->] (0,0)--(1,0) node[pos=1, right, font=\tiny] {$1$} node[pos=0.5, red]  {$\star$};
  \node at (0.5, -0.5) {$\vdots$};
    \node at (3.5, -0.5) {$\vdots$};
  \draw[decorate, decoration={brace,amplitude=5pt,mirror}] (-0.05, 0.02) -- (-0.05, -1) node[left, midway, xshift=-5pt, font=\small] {$\bindex$}; 
  \node at (2,-0.5) {$\mapsto$};
  \draw[->] (3,0)--(4,0) node[pos=1, right, font=\tiny] {$1$}
  node [pos =0.5] {$\bullet$} node [pos =0.25, font = \tiny, above] {$\bindex -1$};
  \draw[decorate, decoration={brace,amplitude=5pt}] (4.3, 0.02) -- (4.3, -1) node[right, midway, xshift=5pt, font=\small] {$\bindex$}; 

\end{scope}
}}.
\]
The fact that this is indeed a chain map
follows from the~locality of the~differential and  $\varphi_\basepoint$.
The~image of $\varphi_\basepoint$ is a subcomplex of $\complexglone(\bdiagram; \aring)$.

Let us place a basepoint on the top left endpoint
of the~braid diagram, and denote by
$\complexglzero(\bdiagram; \aring)$ and $\homologyglzero(\bdiagram; \aring)$
the~chain complex $\qshift^{1-\bindex}\image(\varphi_\basepoint)$ and its homology.
It is called the~\emph{$\gll_0$ homology} of $K=\bdiagram$.
Of course, one can act with $\varphi_\basepoint$ on $\glone(\web;\aring)$ for any
pointed annular web $\web$.
The~image defines a~space $\glzero(\web;\aring)$ called
the~\emph{$\gll_0$ state space} of $\web$.\footnote{In \cite{RW3} this space was denoted by $\glzero^\prime$.}

\begin{theorem}[{\cite{RW3}}]\label{thm:RW3}
  Assume that\/ $\aring$ is a~field.
  \begin{enumerate}
    \item The bigraded\/ $\aring$-vector space
    $\homologyglzero(\bdiagram; \aring)$ is an~invariant of the~knot $\bdiagram$.
    \item Its graded Euler characteristic is the~Alexander polynomial
    $\smash{\Delta_{\bdiagram}(\qvar)}$
    normalized to satisfy the~skein relation \eqref{eq:skein-alex}.
    \item If\/ $\aring$ has characteristic $0$, then there is
    a~bigraded spectral sequence from the~reduced triply graded
    homology to the $\gll_0$ homology.
    \item \label{it:rank-gl0}
    Let $D_k$ be the~chain of dumbbells of index $k>0$
    (see Definition~\ref{def:chain-of-dumbbells}
    and Figure \ref{fig:chain-of-dumbles0}).
    Then $\glzero(D_k;\aring)$ has dimension $1$.
  \end{enumerate}
\end{theorem}

Let us make a~few remarks about these results.
In \cite{RW3}, everything is defined and stated over
$\mathbb{Q}$. There is no  difficulty for extending definition over
$\ZZ$ or any ring $\aring$. The fact that $\aring$ is a field is
needed for proving that the construction is independent from the
basepoint: in the proof of \cite[Proposition 5.6]{RW3}, one needs to know
that the homology of a chain complex has no torsion.

It is important to notice that, contrary to $\homologyglone$,
there is no condition on the invertibility of any integers. This comes from
the fact that proofs of invariance under the~first Markov move (stabilization)
are very different in the two contexts. 
    
The same definition works for links with a basepoint.
However the resulting homology may depend on the component
of the link where the basepoint is placed.
We do not have an~example, though, for which different
choices of components yield different invariants.

\begin{remark}
	The endomorphism $\varphi_\basepoint$ used to define $\complexglzero$ admits an
	alternative description. Instead of adding $\bindex-1$ dots on the
	edge with basepoint, one can add a dot on each edge below the
	basepoint. Indeed, in $\glone(\web)$, the following relation holds
	\[
		\NB{\tikz[]{\begin{scope}[yscale = 1.5]
\begin{scope}
  \draw[->] (0,0)-- +(1,0) node[pos=1, right, font=\tiny] {$1$}
  node [pos =0.5] {$\bullet$} node [pos
  =0.25, font = \tiny, above] {$\bindex -1$};
  \draw[->] (0,-0.3)-- +(1,0) node[pos=1, right, font=\tiny] {$1$};
  \draw[->] (0,-0.6)-- +(1,0) node[pos=1, right, font=\tiny] {$1$};
  \node at (0.5, -0.9) {$\vdots$};
  \draw[->] (0,-1.4)-- +(1,0) node[pos=1, right, font=\tiny] {$1$};
  \draw [decorate,decoration={brace,amplitude=5pt}]
(-0.1, -1.45) -- (-0.1,0.05) node [black,midway,xshift=-0.4cm, font= \small] 
{$\bindex$};  
\end{scope}
\node at (2.9, -0.75) {$=\,\, (-1)^{\bindex-1}$};
\begin{scope}[xshift = 4cm]
  \draw[->] (0,0)-- +(1,0) node[pos=1, right, font=\tiny] {$1$};
  \draw[->] (0,-0.3)-- +(1,0) node[pos=1, right, font=\tiny] {$1$} node [pos =0.5] {$\bullet$};
  \draw[->] (0,-0.6)-- +(1,0) node[pos=1, right, font=\tiny] {$1$} node [pos =0.5] {$\bullet$};
  \node at (0.5, -0.9) {$\vdots$};
  \draw[->] (0,-1.4)-- +(1,0) node[pos=1, right, font=\tiny] {$1$} node [pos =0.5] {$\bullet$};
  \draw [decorate,decoration={brace,amplitude=5pt}]
(1.3,0.05) -- (1.3, -1.45) node [black,midway,xshift=0.4cm, font=\small] {$\bindex$};

\end{scope}
\end{scope}
}}
	\]
	because of the~equality
	$$x_2\cdots x_k = \sum_{i=1}^k (-1)^{i-1} x_1^{i-1} e_{k-i}(x_1,\dots,x_k)$$
	and Proposition~\ref{proposition:morphism-gl1} (3).
	The~signs in this formula has absolutely no consequence on
	the definition of $\complexglzero$ since we are only interested
	in the~image of $\varphi_\basepoint$.
\end{remark}

The~complex $\complexglzero(\bdiagram)$ is defined above as
a~subcomplex of $\qshift^{1-\bindex}\complexglone(\bdiagram)$,
but one can change the~point of view and construct it also
as a~quotient of $\qshift^{\bindex-1}\complexglone(\bdiagram)$,
which leads to a~definition via a~universal construction,
Indeed, given a~pointed annular web $\web$, the~map $\varphi_\basepoint$
is the~multiplication by a~decoration, hence, an~endomorphism of $D(\web)$.
This allows us to define a~new pairing $\bracket{-;\web;-}_0$ on $D(\web)$
as $\bracket{S;\web; T}_0 := \bracket{\varphi_\basepoint(S),\web, T}_1$
for all $S,T\in D(\web)$.
Then we have an isomorphism
\begin{equation}\label{eq:iso-gl0-quotient}
	\glzero(\web;\scalars) \cong  \quotient
		{\qshift^{\bindex -1} D(\web)}
		{\mathrm{rad}  \bracket{-; \web; -}_0}
\end{equation}
Clearly, 
\[
	{\mathrm{rad}\bracket{-; \web; -}_1}  \subseteq \mathrm{rad}  \bracket{-; \web; -}_0
\]
and the~isomorphism \eqref{eq:iso-gl0-quotient} commutes with the~differentials,
so that $\complexglzero(\bdiagram)$ is a~quotient of
$\qshift^{\bindex-1}\complexglone(\bdiagram)$.
In particular, for any pointed annular web $\web$,
the~space $\glzero(\web;\scalars)$ is a~quotient of $B(\web)$.

 \subsection{The~spectral sequence from the~reduced triply graded homology
  to \texorpdfstring{$\LieGL_0$}{gl(0)} homology}

We shall discuss here the~construction of the~spectral sequence
from $\HHH^{\textrm{red}}$ to $\homologyglzero$ sketched in \cite{RW3}.
Hereafter it is assumed that $\scalars$ is a~field of characteristic 0.

Given an~elementary web $\web$ of index $k$ write $K^{\textrm{red}}(\web)$
for the~(reduced) Koszul complex associated with $\web$
that is used to compute the~reduced triply graded homology.
It is defined as in \eqref{eq:Koszul-resolution},
except that the~reduced bimodule $\set{SoergelRed}(\web)$ is taken
and the~tensor product is for $i\geqslant 2$.
Besides the~Koszul differential $d_{H\!H}$
this complex admits an~additional differential $d_0$
that is induced by identity maps parallel to the~components of $d_{H\!H}$.
Hence, the~total complex  can be written as
\begin{equation}\label{eq:Koszul-total}
	K^{d_0}(\web) := \set{SoergelRed}(\web)
		\otimes_{R^e}
	\bigotimes_{i=2}^k (\qshift^2 R^e \xrightarrow{\ x_i\otimes 1 - 1\otimes x_i + 1\otimes 1\ } R^e)
\end{equation}
A~quick check shows that $d_0$ anticommutes with $d_{H\!H}$,
so that \eqref{eq:Koszul-total} is indeed a~chain complex of graded $\scalars$-modules.
Let us write $H^{d_0}(\web)$ for the~homology of \eqref{eq:Koszul-total}.

\begin{proposition}[cp.\ \cite{RW3}]
\label{prop:d0-vs-S0}
	The~functor $H^{d_0}$ from $\cat{AFoam}^\basepoint$ to $\scalars$-modules is isomorphic to $\glzero$.
\end{proposition}
\begin{proof}
	We first show that $H^{d_0}(\web)$ vanishes when $\web$ is disconnected.
	Indeed, suppose that $\web = \web' \sqcup \web''$
	with $\web'$ and $\web''$ of index $k'$ and $k''$ respectively.
	Then $e_1(x_{k'+1},\dots, x_{k}) \in R$ acts symmetrically on $\set{SoergelRed}(\web)$.
	After changing the~basis of $R$ by trading $x_k$ for the~above polynomial,
	the~factor in \eqref{eq:Koszul-total} for $i=k$ is the~identity map.
	Therefore, the~complex is acyclic.

	Let now $\web = D_k$ be the~chain of $k-1$ dumbbells as shown
	in Figure~\ref{fig:labelled-chain-of-dumbbells}.
\begin{figure}[ht]
		\NB{\tikz[font=\tiny,scale=0.5]{\input{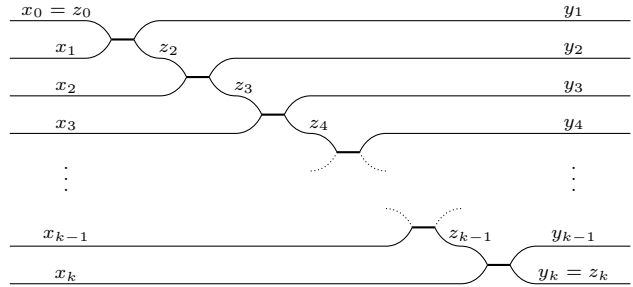}}}
		\caption{A~open chain of dumbbells}
		\label{fig:labelled-chain-of-dumbbells}
	\end{figure}
The~associated reduced Soergel bimodule $\set{SoergelRed}(\web)$ is generated by
	the~variables $x_i, y_i, z_i$ modulo local relations and $y_1 = 0$.
	Notice that the~set $\Lambda$ of these relations is \emph{regular},
	that is none of its elements is a~zero divisor modulo the~others.
	Hence, the~tensor product
	\begin{equation}\label{eq:pres-of-B}
		\bigotimes_{r \in \Lambda} ( D(\web) \xrightarrow{\ r\ } D(\web) ),
	\end{equation}
	where $D(\web)$ is the~algebra of all decorations of $\web$,
	is a~projective resolution of $\set{SoergelRed}(\web)$.
	Enlarging $\Lambda$ by adding relations
	\begin{equation}\label{eq:d0-relations}
		x_i = y_i - 1
	\end{equation}
	for $1<i\leq k$ produces a~complex that computes $H^{d_0}(\web)$.
	We claim that this enlarged set of relations is still regular.
	For that it is enough to notice that \eqref{eq:d0-relations} allows
	to rewrite local linear relations as
	\begin{equation}\label{eq:new-local-linear}
		z_i = y_i + (k-i) \qquad \text{for $1\leq i\leq k$.}
			\end{equation}
The~local quadratic relations can be then rewritten as
	\begin{align*}
		x_{i+1}z_i - y_i z_{i+1}
			&= (y_{i+1} - 1)(y_i + (k-i)) - y_i(y_{i+1} + (k-i-1)) \qquad \text{for $1\leq i <k$}\\
			&= (y_{i+1} - y_i - 1)(k-i)
	\end{align*} 
	and reduced to $y_{i+1} = y_i + 1$, and further to $y_{i+1} = i$.
	Hence, the~complex computing $H^{d_0}(D_k)$ is isomorphic via a~base change
	to the~Koszul complex associated with the~sequence of relations
	\[
		y_i = i-1, \quad
		z_i = (k-1), \quad
		x_{i+1} = i-1,
	\]
	which is clearly a~regular set. Thence, higher homology vanishes
	and $H^{d_0}(\web) = \scalars$ is generated by the~empty decoration.
	
	It follows now from Proposition~\ref{prop:elem-annular-reduction}
	that $H^{d_0}(\web)$ is concentrated in degree 0 for any elementary
	web $\web$. Hence, it is a~quotient of $\set{SoergelRed}(\widehat\web)$,
	the~Soergel space associated with the~annular closure of $\web$.
	The same holds for $\glzero$ and the~desired isomorphism is induced
	by the~identity on $\set{SoergelRed}(\widehat\web)$.
\end{proof}

Proposition~\ref{prop:d0-vs-S0} is the~main ingredient
in the~construction of the~desired spectral sequence.
Instead of computing $H^{d_0}$ at once,
one can first compute the~homology with respect to $d_{H\!H}$
and consider the~spectral sequence induced by $d_0$.
When applied to the~cube of resolutions of a~braid closure $\braid$,
this produces the~desired spectral sequence from
$\HHH^{\text{red}}(\widehat\braid)$ to $\homologyglzero(\widehat\braid)$.
Notice that the~differential on the~second page is induced by $d_0$
and of $(\avar,\qvar,\tvar)$-degree $(-2, 0, 1)$.

\begin{remark}
	A~priori $H^{d_0}(\web)$ does not admit the~quantum grading,
	because it is not preserved by the~total differential.
	However, it can be recovered by Proposition~\ref{prop:d0-vs-S0}.
	Another way to introduce the~quantum grading on $H^{d_0}(\web)$
	is to consider the~spectral sequence on $\HoHom(\set{SoergelRed}(\web))$
	induced by $d_0$ and notice that it collapses immediately \cite{RW3}.
	Hence, $H^{d_0}(\web)$ is the~homology of the~complex
	$(\HoHom(\set{SoergelRed}(\web)), d_0)$,
	in which the~Hochschild and quantum gradings are collapsed to a~single grading.
\end{remark}

\section{Heegaard Floer homology}
\label{sec:hfk}
We shall now review the~(twisted) Heegaard Floer homology
for a~marked singular link $S$ following \cite{OSSSingular, OSCube}.
We write $S_0, \dots, S_r$ for the~components of $S\ric$,
ordered top to bottom with respect to their topmost trace vertices.
In particular $\basepoint \in S_0$.

\subsection{Heegaard diagrams and holomorphic disks}

We begin with describing Heegaard diagrams for a~pointed singular link $S\ric$,
where the~marking is located on the~component $S_0$.
Let $(H_\alpha, H_\beta, \Sigma)$ be a~Heegaard splitting of $\mathbb S^3\ric$,
for which there is a~thickening $V \approx \Sigma\times[0,1]$ of $\Sigma$
satisfying what follows:
\begin{itemize}
	\item $S \cap V$ is a~collection of fibers of $V$ that includes
	all thick edges of $S$ and an~arc with the~marking $\basepoint \in S_0$,
	\item thick edges of $S$ and the~arc carrying the~marking
	are oriented from $H_\alpha$ to $H_\beta$, and
	\item $S \setminus V$ consists of untangled arcs in $H_\alpha$ and $H_\beta$,
	so that it can be isotoped onto $\partial V\ric$.
\end{itemize}
Partition $S\cap \Sigma$ into $\markX \sqcup \markO$,
where $\markX = \{X_0,\dots,X_k\}$ consists of the~intersection points,
at which the~link is oriented from $H_\alpha$ to $H_\beta$,
and $\markO = \{O_0,\dots,O_{k+s}\}$ of the~other ones
($s$ is the~number of thick edges in $S$).
The~intersection points of thick edges with $\Sigma$
form a~subset $\markXX \subset \markX$.
The~elements of $\markX$ (resp.\ $\markO$)
are called \emph{$\markX$-basepoints} (resp.\ \emph{$\markO$-basepoints});
the~points from $\markXX$ are \emph{double basepoints}.
It is understood that $X_0$ coincides with the~marking $\basepoint \in S$
and $O_0$ is located on $S$ just before $X_0$, i.e.\ both basepoints
are connected with an~arc in $H_\alpha$.

The~last condition on the~Heegaard splitting guarantees the~existence
of a~collection of $k+g$ disks in $H_\alpha$ with boundary on $\Sigma$
(here $g$ stands for the~genus of $\Sigma$)
that cut the~handlebody into balls, each with one untangled piece of $S$:
either an~arc or a~$\mathsf{Y}$-shape (a~web with a~single vertex).
The~boundary of this collection $\colalpha = \{\alpha_1, \dots, \alpha_{g+k}\}$
consists of \emph{$\alpha$-curves} that decompose $\Sigma$ into regions $A_0, \dots, A_k$,
each containing a~unique $\markX$-basepoint.
By convention we enumerate the~regions so that $X_i \in A_i$.
We choose a~collection $\colbeta = \{\beta_1, \dots, \beta_{g+k}\}$
of \emph{$\beta$-curves} likewise by decomposing $H_\beta$
and we write $B_0,\dots, B_k$ for the~closures of the~connected
components of $\Sigma \setminus \colbeta$, where $X_i \in B_i$.
It is assumed that the~two families of curves intersect transversely.

The~data $(\Sigma, \colalpha, \colbeta, \markX, \markO)$ determines
the~link $S$ completely.
We call it a~\emph{multipointed Heegaard diagram for $S$}.
It is not determined by $S\ric$, but any two Heegaard diagrams
for $S$ are related by a~finite number of certain moves
\cite[Proposition 3.3]{Multivariable}.

\begin{figure}[ht]
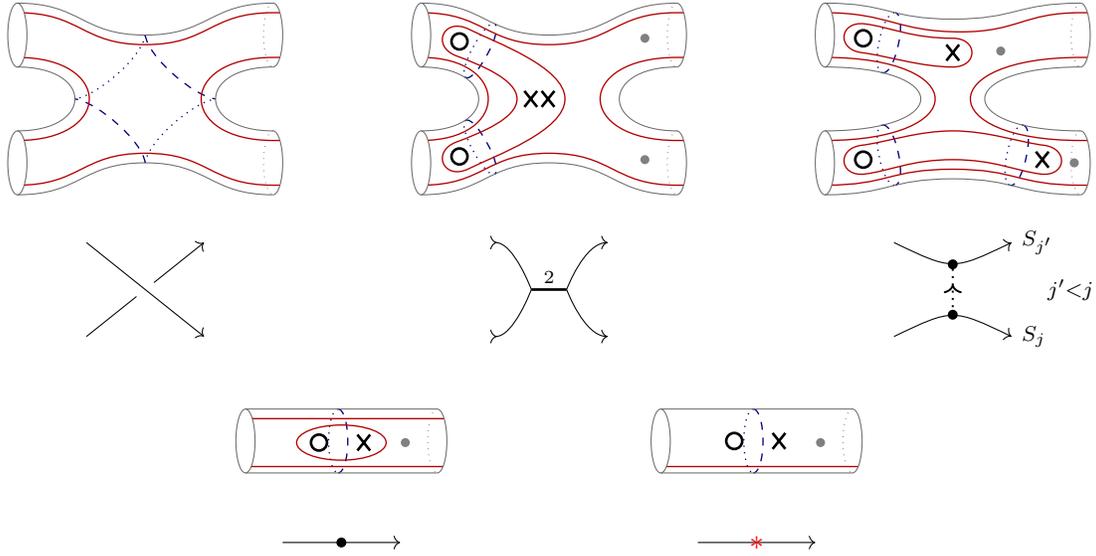

	\def\diagmarkinglevel{1}		\begin{tabular}{cp{1cm}cp{1cm}c}
		\NB{\tikz[scale=0.85]{\input{\imagesfolder/hfgl0_marked-hd-crossing}}} &&
		\NB{\tikz[scale=0.85]{\input{\imagesfolder/hfgl0_marked-hd-singular}}} &&
		\NB{\tikz[scale=0.85]{\input{\imagesfolder/hfgl0_marked-hd-two-components}}}
	\\ && \\[0.5ex]
		\NB{\tikz[scale=1.25]{%
%
\begin{scope}[xscale=1.25]
	\draw[->] (-0.5, -0.5) -- (0.5,0.5);
	\fill[white] (0,0) circle[radius=3pt];
	\draw[->] (-0.5, 0.5) -- (0.5,-0.5);
\end{scope}}} &&
		\NB{\tikz[scale=1.25]{\input{\imagesfolder/hfgl0_singular-resolution}}} &&
		\NB{\tikz[scale=1.25]{\input{\imagesfolder/hfgl0_link-connecting-arc}}}
	\end{tabular}
	\\[5ex]
	\begin{tabular}{cp{2cm}c}
		\NB{\tikz[scale=0.85]{\input{\imagesfolder/hfgl0_marked-hd-bivalent}}} &&
		\NB{\tikz[scale=0.85]{\input{\imagesfolder/hfgl0_marked-hd-marking}}}
	\\ && \\[0.5ex]
		\NB{\tikz[scale=1.25]{\input{\imagesfolder/hfgl0_link-bivalent}}} &&
		\NB{\tikz[scale=1.25]{\input{\imagesfolder/hfgl0_link-marking}}}
	\end{tabular}
	\caption{Local pictures of the~initial Heegaard diagram for a~singular link $S\ric$.
		The~top row shows the~pieces of the~diagram around a~positive crossing
		(flip the~picture for the~negative one), a~singular crossing,
		and at a~chosen line connecting different components of the~link.
		The~bottom row shows configurations around bivalent vertices:
		a~general situation to the~left and the~case of the~distinguished vertex
		(the~marking of $S$) to the~right.
		Dashed blue and solid red arcs represent $\alpha$ and $\beta$ curves respectively.
	}
	\label{fig:initial-heegaard-diagram}
\end{figure}

\begin{example}
	Given a~diagram of a~singular link $S\ric$,
	possibly with thin edges subdivided by bivalent vertices,
	we construct a~multipointed Heegaard diagram,
	called the~\emph{initial Heegaard diagram for $S\ric$}, as follows.
	First, we choose the $\alpha$- and $\beta$-curves
in a way,	
that
	 near crossings and  vertices they look as in the~left and middle pictures
	of Figure~\ref{fig:initial-heegaard-diagram}.
	The~blue $\alpha$-curves 
	are meridians around thin edges of $S$ and curves around
	real crossings. The~red $\beta$-curves are either parallel to the~contours of
	the~surface or they bound disks containing the~$X$ and $O$ basepoints;
	the~latter are called \emph{internal curves}.
	We then perform two modifications. Let $S_0$ be the~component of $S$ that
	carries the marking of $S$.  We enumerate other components $S_1,\ldots,S_r$,
	so that each $S_j$ for $j>0$ can be connected with an arrow to some
	$S_{j'}$ with $j' < j$.
	Attach a~handle to the~surface along each such arrow,
	merge two $\beta$-curves along this handle and add a~new meridional
	$\alpha$-curve around $S_j$
	as shown in the~top right corner of Figure~\ref{fig:initial-heegaard-diagram}.\footnote{
		This configuration differs by a~handle slide from the~one associated
		with a~smooth resolution in \cite{OSCube}.
	}
	This guarantees that the~surface is connected, but it has two more
	$\beta$-curves than $\alpha$-curves.
	We fix it by removing two $\beta$-curves near the~marking:
	the~internal one and one parallel to the~contour of the~surface,
	see the~bottom right corner of Figure~\ref{fig:initial-heegaard-diagram}
	(this is the~piece of the~diagram with basepoints $X_0$ and $O_0$).
	The~reader is encouraged to check that the~resulting diagram satisfies the~conditions
	for a~Heegaard diagram.
	In particular, the~number of $\beta$-curves matches the~number of $\alpha$-curves.
\end{example}

Consider the~symmetric product of the~underlying surface $\Sym^{g+k}(\Sigma)$.
It is a~symplectic manifold with Lagrangian tori
$\torus_\alpha = \alpha_1 \times \dots \times \alpha_{g+k}$
and $\torus_\beta = \beta_1 \times \dots \times \beta_{g+k}$
that are totally real with respect to any compatible almost complex structure.
Let $\CFbasis = \torus_\alpha \cap \torus_\beta$
be the~set of intersection points between the~two tori.
Given $\xgen,\ygen\in\CFbasis$ denote by $\pi_2(\xgen,\ygen)$ the~set of homotopy classes
of \emph{Whitney disks} from $\xgen$ to $\ygen$, i.e.\ continuous maps
of the~standard complex disk into $\Sym^{g+k}(\Sigma)$
that carry $-i$ (resp.\ $i$) to $\xgen$ (resp.\ $\ygen$) and points on the~boundary circle
with positive (resp.\ negative) real part to $\torus_\alpha$ (resp.\ $\torus_\beta$).
This set is never empty when $H_1(Y) = 0$ \cite[Section 2]{OS3mflds}.
Under generic conditions, the~moduli space $\moduli(\phi)$ of pseudo-holomorphic
representatives of $\phi$
is an~orientable smooth manifold of dimension equal to the~\emph{Maslov index} $\mu(\phi)$.
It admits a~free action of the~group of conformal automorphisms of a~unit disk that fix
the~points $i$ and $-i$, which is isomorphic to $\mathbb R$.
We write $\moduliR(\phi)$ for the~quotient of $\moduli(\phi)$ by the~action of this group
and refer to its elements as \emph{unparametrized} disks.
In case $\xgen=\ygen$ we consider also \emph{degenerate} Whitney disks that carry
$i$ to $\xgen$ and the~entire boundary circle either to $\torus_\alpha$ ($\alpha$-degeneracy)
or $\torus_\beta$ ($\beta$-degeneracy), but there is no additional condition on the~image of $-i$.
The~space $\nmoduli(\psi)$ of pseudo-holomorphic representatives of such a~degenerate disk $\psi$
is generically an~orientable smooth manifold of dimension $\mu(\psi)$ that admits
a~free action of a~two-dimensional subgroup $\mathbb G$ of $\mathit{PSL}(2,\mathbb R)$.
In particular, $\mu(\psi)$ is even when $\nmoduli(\psi)\neq\emptyset$.
We write $\nmoduliR(\psi)$ for the~quotient space.

We shall now recall a~combinatorial formula for the~Maslov index \cite{LipshitzCylindrical}.
Given a~Whitney disk $\phi$ and a~point $p\in\Sigma$
we write $n_p(\phi)$ for the~algebraic intersection number
of $\phi$ with the~subvariety $\{p\}\times\Sym^{g+k-1}(\Sigma)$.
Let $\{\Omega_i\}$ be the~set of (closures of) connected components
of $\Sigma \setminus (\colalpha \cup \colbeta)$, and choose $p_i\in\Omega_i$
for each $i$. The~2-chain
\[
	\mathcal D(\phi) := \sum_i n_{p_i}(\phi) [\Omega_i] \in C_2(\Sigma, \alpha\cup\beta)
\]
is called the~\emph{domain associated with $\phi$}.
It determines $\phi$ uniquely when $g+k>2$ and it is \emph{positive},
i.e.\ with no negative coefficients, when $\phi$ has a~holomorphic representative.
The~Maslov index $\mu(\phi)$ can be computed as follows.
Let $(V, \partial V) \subset (\Sigma, \alpha \cup \beta)$ be a~2-chain bounded
by arcs in $\alpha$ and $\beta$ curves and denote by $\mathrm{ac}(V)$ and $\mathrm{obt}(V)$
the~number of acute and obtuse corners of $V$. The~\emph{Euler norm} of $V$
\[
	e(V) = \chi(V) + \tfrac{1}{4}(\mathrm{obt}(V) - \mathrm{ac}(V))
\]
is additive under gluing domains together, so that it can be extended linearly
to relative 2-chains.
The~Maslov index of $\phi \in \pi_2(\xgen, \ygen)$ represented by a~domain
$\mathcal D = \mathcal D(\phi)$ is then given as
\[
	\mu(\phi) = e(\mathcal D) + a_\xgen + a_\ygen,
\]
where $a_\xgen$ (resp. $a_\ygen$) is the~sum of average multiplicities of $\mathcal D$
over the~four regions around each $x_i$ (resp.\ $y_i$). In particular, if $\mathcal D$
has only multiplicities $0$ and $1$, then $a_p = \tfrac 14$ or $\tfrac 34$ when $p$ is
an~acute or an~obtuse corner respectively. Note that both $\mathcal D$ and $\mu$ are
additive with respect to the~juxtaposition of Whitney disks
$\star\colon\pi_2(\xgen,\ygen)\times\pi_2(\ygen,\zgen) \to \pi_2(\xgen,\zgen)$.
\begin{example}

	Suppose that $\mathcal D = \mathcal D(\phi)$ is a~bigon with acute corners and
	no intersection point inside. Then $e(\mathcal D) = 1 - \tfrac{2}{4} = \tfrac{1}{2}$
	and $a_\xgen = a_\ygen = \tfrac{1}{4}$, leading to $\mu(\phi) = 1$.
	This is consistent with the~fact, that $\phi$ has a~unique holomorphic representative
	up to reparametrization by the~Riemann Mapping Theorem.
	More generally, if $\mathcal D$ is a~$2n$-gon with acute corners and intersection points
	$x_i = y_i$ for $i = 1,\ldots,r$ in its interior, then
	$e(\mathcal D) = 1 - \tfrac{n}{2}$ and $a_\xgen = a_\ygen$ = $\tfrac{n}{4} + r$,
	resulting in $\mu(\phi) = 1 + 2r$.
\end{example}

\begin{example}
	Let $\psi \in \pi_2(\xgen, \xgen)$ be a~degenerated disk with its associated domain
	$\mathcal D$ equal to $A_i$ or $B_j$.
	Because $\nmoduli(\psi)$ can be identified with the~group $\mathbb G$ under
	certain conditions on the~almost complex structure \cite[Section 5]{OS3mflds},
	we have $\mu(\psi) = \dim\mathbb G = 2$.
	The~Maslov index of $\psi$ can be also computed combinatorially as follows.
	First notice that the~$\alpha$-curves as well as the~$\beta$-curves decompose $\sfce$
	into pieces of genus 0.
	Hence, if $\mathcal D = A_i$ (resp.\ $B_j$), then it must contain extacly
	$g = g(\mathcal D)$ $\alpha$-curves (resp.\ $\beta$-curves) in its interior
	and each such curve carries a~unique intersection point $x_i$ from $\xgen$;
	this point contributes $1$ towards $a_\xgen = a_\ygen$.
	Likewise, each boundary component of $\mathcal D$ carries an~intersection
	point contributing $\tfrac{1}{2}$. Therefore $\mu(\psi) = \chi(\mathcal D) + 2g + r = 2$
	as desired, where we write $r$ for the~number of boundary components of $\mathcal D$.
\end{example}

A~domain $\pi$ is \emph{periodic} if its boundary is a~linear combination
of $\alpha$- and $\beta$-curves. It represents a~difference of two Whitney
disks connecting the~same intersection points.
Of particular interest are periodic domains that vanish at both $\markX$-
and $\markO$-basepoints. For instance, there is such a~periodic domain
\begin{equation}\label{eq:domain-pi_j}
	\pi_j = \sum_{i : X_i \in S_j} (A_i - B_i)
\end{equation}
associated with any component $S_j$ of the~singular link $S$.
We call them \emph{fundamental periodic domains}.
Clearly $\sum_j \pi_j = 0$ and one can show that
every periodic domain vanishing at both $\markX$- and $\markO$-basepoints
is a~linear combination of $\pi_j$'s.

We say that a~Heegaard
diagram is \emph{admissible} if there are no nontrivial positive periodic
domains that avoid the~set $\markX$.\footnote{
	This condition is equivalent to the~absence of nontrivial
	positive periodic domains of Maslov index 0.
}

\begin{lemma}\label{lem:initial-is-admissible}
	Initial Heegaard diagrams are admissible.
\end{lemma}

\begin{proof}
    The argument is motivated by the~one in \cite[Lemma 2.1]{ManolescuCube}.
    A~periodic domain that avoids $\markX$ is given by a 2-chain
    $\pi = \sum_{i=0}^k a_i (A_i - B_i)$. Assume that $\pi$ is positive,
    in which case $a_i \geqslant a_j$ whenever $A_i \cap B_j \neq \emptyset$.
    In particular, the coefficients $a_i$ cannot decrease when we follow
    any path in $S$. Because $S$ can be expressed as an infinite union of oriented circles (see Figure~\ref{fig:lamination}), this implies that the~coefficients are in fact
    locally constant, i.e.\ $a_i = a_{i'}$ if both $X_i$ and $X_{i'}$ lie on the~same
    component of $S$. This forces $\pi = \sum_j c_j \pi_j$ for some $c_j \in \ZZ$.
    
    Let $S_0, \ldots, S_r$ be the~connected components of $S$,
    where $\basepoint \in S_0$.
    Suppose that $S_j$ and $S_{j'}$ are next to each other as in the~top right
    picture of shown on Figure~\ref{fig:initial-heegaard-diagram}
    with $X_j$ and $X_{j'}$ the~lower and upper $X$ basepoints respectively.
    Recall that in case $j'=0$ the~red curve around $X_{j'}$ is removed.
    The~disk $B_j$, bounded by the~red curve around $X_j$, intersects $A_{j'}$
    that carries $X_{j'}$. Hence, $c_j \leqslant c_{j'}$ by the~positivity of $\pi$.
    In other words, the~values of $c_j$'s do not decrease when we go along
    the trace~section towards the~marking $\basepoint \in S_0$,
    so that $c_j \leqslant c_0$ for each $j$.
    On the~other hand, the~region $B_0$, obtained form the~surface by removing
    all disks bounded by internal $\beta$ curves, intersects all $A$ regions.
    Thus all the~coefficients $c_j$ coincide with $c_0$,
    forcing $\pi = c_0\sum_j \pi_j = 0$.
\end{proof}

\begin{remark}
	The~initial diagram from \cite[Fig. 3]{OSCube} fails to be admissible
	in case of links, because in this diagram periodic domains associated
	with different components of the~link are disjoint.
\end{remark}

The~moduli space of unparametrized curves $\moduliR(\phi)$ admits a~Gromov
compactification that adds as boundary points \emph{nodal disks} or
\emph{broken flow lines},
which are juxtapositions of Whitney disks $\phi_1, \dots, \phi_r$,
degenerated disks $\psi_1,\dots,\psi_s$ and holomorphic spheres
$S_1,\dots, S_t$ with the~property that
\[
	\mathcal D(\phi)
		= \sum_{i=1}^r \mathcal D(\phi_i)
		+ \sum_{j=1}^s \mathcal D(\psi_j)
		+ \sum_{k=1}^t \mathcal D(S_k),
\]
where $r>1$, $s>0$, or $t>0$.
In particular, $\moduliR(\phi)$ is already compact and $0$-dimensional
when $\mu(\phi)=1$ and the same holds for $\nmoduliR(\psi)$ with $\mu(\psi)=2$.
Note that holomorphic spheres do not appear unless $\mathcal D(\phi)$
covers the~entire surface, because $\mathcal D(S) = \Sigma$ for any holomorphic
sphere $S$, see \cite{OS3mflds}. 

Having chosen an~orientation of moduli spaces, we can write
$\#\moduliR(\phi)$ (resp.\ $\#\nmoduliR(\psi)$) for the~signed count
of points, where the~sign of a~point depends on whether
the~action of the~parametrization group
on the~moduli space $\moduli(\phi)$ (resp.\ $\nmoduli(\psi)$)
preserves or reverses the~orientation.
According to \cite{AliEft} one can always choose a~system of orientations
with the~following properties:
\begin{enumerate}
	\item the~orientation of $\moduliR(\phi_1)\times\moduliR(\phi_2)$ is induced
	from $\moduliR(\phi_1\star\phi_2)$ for any Whitney disks
	$\phi_1\in\pi_2(\xgen,\ygen)$ and $\phi_2\in\pi_2(\ygen,\zgen)$,
	\item $\nmoduliR(\psi)$ carries the~orientation induced from
	$\moduliR(\psi)$ in case of $\alpha$-degeneracies and the~opposite
	one in case of $\beta$-degeneracies,
	\item $\#\nmoduliR(\psi) = 1$ when $\psi = A_i$ or $B_j$.
\end{enumerate}
As an~easy application of the~above definition we show the~following fact,
which is an~important tool to analyze the Heegaard Floer differential
recalled in Section~\ref{sec:HFK}.

\begin{lemma}\label{lem:signs-for-AB-pair}
    Suppose that $A_i\cap B_i$ is a~bigon for some $i$ and let
    $\phi_1$ and $\phi_2$ be Whitney disks represented by the~domains
    $A_i \setminus B_i$ and $B_i\setminus A_i$.
    Then $\#\moduliR(\phi_1) + \#\moduliR(\phi_2) = 0$.
\end{lemma}
\begin{proof}
	Choose $\xgen$ and $\ygen$, such that $\phi_1,\phi_2 \in \pi_2(\ygen, \xgen)$.
	Let $\phi_0 \in \pi_2(\xgen, \ygen)$ and $\psi_A,\psi_B \in \pi_2(\xgen, \xgen)$ be Whitney
	disks represented by the~domains $A_i \cap B_i$, $A_i$, and $B_i$ respectively.
	Notice that $\phi_0$, $\phi_1$, and $\phi_2$ are the only Whitney disks between $\xgen$
	and $\ygen$ that have Maslov index 1 and are associated with domains contained in $A_i \cup B_i$.
	Because $\moduliR(\psi_A)$ and $\moduliR(\psi_B)$ are 1-dimensional,
	counting endpoints of each with signs gives zero. Hence,
	\begin{align*}
		0 &= \#\partial\moduliR(\psi_A)
		   = \#\left(\moduliR(\phi_0) \times \moduliR(\phi_1) \right) + \#\nmoduliR(\psi_A)
		   = \#\moduliR(\phi_0) \cdot \#\moduliR(\phi_1) + 1, \\
		0 &= \#\partial\moduliR(\psi_B)
		   = \#\left(\moduliR(\phi_0) \times \moduliR(\phi_2) \right) - \#\nmoduliR(\psi_B)
		   = \#\moduliR(\phi_0) \cdot \#\moduliR(\phi_2) - 1.
	\end{align*}
	Summing up the~two equations proves the~thesis.
\end{proof}

Computing the~numbers $\#\moduliR(\phi)$ in general is very challenging,
but in some cases the answer is known.
For instance, $\#\moduliR(\phi) = \pm1$ if $\mathcal D(\phi)$
is a~bigon. The~following is a~generalization of that.

\begin{lemma}[{\cite[Lemma 3.11]{OSCube}}]\label{lem:index-for-disk-missing-disks}
	Suppose that the~2-chain $\mathcal D(\phi)$ associated with a~Whitney
	disk $\phi \in \pi_2(\xgen, \ygen)$ is a~planar region with the~following
	properties:
	\begin{itemize}
		\item all its boundary components are $\alpha$-curves (resp.\ $\beta$-curves)
		except one that is a~$2n$-gon with acute corners and edges alternating between
		arcs in $\alpha$- and $\beta$-curves (thus, the~corners alternate between components
		of $\xgen$ and $\ygen$),
		\item there is a~collection of arcs in $\beta$-curves
		(resp.\ $\alpha$-curves) between boundary components of $\mathcal D(\phi)$
		that cut the~domain into a~disk, and
		\item no component $x_i$ or $y_i$ is in the~interior of $\mathcal D(\phi)$.
	\end{itemize}		
	Then $\mu(\phi)=1$ and $\#\moduliR(\phi) = \pm 1$.
\end{lemma}

 \subsection{The~Heegaard Floer complex}
\label{sec:HFK}

We are ready to state the definition of the~twisted Heegaard Floer homology from \cite{OSCube}.
As usual $\scalars$ is a~commutative ring and we fix an~invertible $t\in\scalars$.

\begin{definition}
	Let $(\Sigma, \colalpha, \colbeta, \markX, \markO)$ be an~admissible Heegaard diagram
	associated with a~marked singular link $S\subset Y\ric$.
	Given a~finite set of markings
	$P\subset\Sigma - (\colalpha \cup \colbeta \cup \markX \cup \markO)$
	we define the~\emph{twisted Heegaard Floer complex} $\tCFK-(S,P)$
	as the~free module over the~polynomial algebra $R=\scalars[U_0,\dots,U_{k+s}]$
	with a~basis consisting of the~intersection points
	$\xgen \in \CFbasis$ and the~differential given by
	\begin{equation}\label{eq:CFdiff}
		\partial\xgen =
			\sum_{\ygen \in \CFbasis}
			\sum_{\substack{
				\phi\in\pi_2(\xgen, \ygen)\\
				\mu(\phi) = 1\\
				\forall i\colon X_i(\phi) = 0 
			}}
			\# \moduliR(\phi)\,t^{P(\phi)}\,U_0^{O_0(\phi)} \cdots U_{k+s}^{O_{k+s}(\phi)}\ \ygen,
	\end{equation}
	where $P(\phi)$ counts multiplicities of $\mathcal D(\phi)$ at elements of $P\ric$,
	$X_i(\phi)$ is the~multiplicity of $\phi$ at $X_i$ and likewise for $O_i(\phi)$.
	The~complex $\tCFK(S,P)$ is defined as the~quotient $\tCFK-(S,P)/(U_0=0)$
	by the~variable $U_0$ associated with the~basepoint $O_0$ that is located
	on the~arc of $S$ that terminates at the~marking $\basepoint\in S$.
	The~homology of the~complexes are denoted respectively by $\tHFK-(S,P)$ and $\tHFK(S,P)$.
\end{definition}

\begin{remark}
	The~admissibility implies that any two points $\xgen, \ygen \in \CFbasis$
	are connected by only finitely many positive domains that avoid the~set $\markX$,
	making the~right hand side of \eqref{eq:CFdiff} a~finite sum.
\end{remark}

\begin{remark}
	The~above definition recovers the~usual Heegaard Floer
	complexes $\CFK-(S)$ and $\CFK(S)$ when $t=1$ or $P$ is empty.
\end{remark}

\begin{example}
	The~\emph{Ozsv\'ath--Szab\'o} twisting is given by the~set of markings
	$\OSTwist$ visualized in Figure~\ref{fig:initial-heegaard-diagram}
	with gray dots. When $S$ is in a~braid position, then we can pick
	a~subset $\trTwist \subset \OSTwist$ that consists only
	of the~dots near trace vertices. These two sets lead to essentially
	the~same twisted complex, see Corollary~\ref{cor:OS-vs-tr-twisting}.
\end{example}

The~Heegaard Floer complex is bigraded, with the~\emph{Alexander grading} $A(\xgen)\in\ZZ$
and the~\emph{Maslov grading} $M(\xgen)\in\ZZ$.
In this paper, however, we consider slightly different gradings:
the~\emph{quantum grading} $\qdeg(\xgen) := -2 A(\xgen)$ and
the~\emph{homological grading} $\hdeg(\xgen) := 2A(\xgen) - M(\xgen)$
that satisfy
\begin{align}
	\label{eq:qdeg-for-disk}
	\qdeg(\ygen) - \qdeg(\xgen) &= 2 \markX(\phi) + 2 \markXX(\phi) - 2 \markO(\phi)
\\
	\label{eq:hdeg-for-disk}
	\hdeg(\ygen) - \hdeg(\xgen) &= \mu(\phi) - 2 \markX(\phi)
\end{align}
for any $\phi \in \pi_2(\xgen, \ygen)$, where given a~finite set $Q\subset\Sigma$
we write $Q(\phi)$ for the~sum of multiplicities of $\phi$ at points from $Q$.
Note that points from $\markXX$ are counted four times in \eqref{eq:qdeg-for-disk},
because $\markXX \subset \markX$.

\begin{remark}
	At the~first sight it might seem that the~formulas
	\eqref{eq:qdeg-for-disk} and \eqref{eq:hdeg-for-disk}
	do not match those from \cite[after Definition 2.1, p.~871]{OSCube}.
	The~reason is that double $\markX$-basepoints are considered
	in \cite{OSCube} as pairs of basepoints,
	so that $\markX(\phi)$ from \cite{OSCube}
	matches our $\markX(\phi) + \markXX(\phi)$.
\end{remark}

The~homological grading is normalized using the~quotient complex
$\widehat{\mathit{CF}} = \CFK-(S)/(U_i=1)$,
in which all variables $U_i$ are specialized to 1.
It is known that this~complex computes the~cohomology of the~$k$-torus
and we require that the~top degree generator of the~homology is in degree 0.
The~quantum grading is then normalized to match the~formula
\[
	\Delta(S) = (\qvar-\qvar^{-1})^\sigma \sum_{d,s}
		(-1)^d \qvar^s \mathrm{rk} \tHFK_d(S, P; s),
\]
where $\sigma$ is the~number of singular crossings in $S$ and
$\tHFK_d(S, P; s)$ is the~degree $s$ component of the~$d$-th homology group.

We finish this section with a~short discussion
on how the~complex depends on the~twisting set.
It is known that the homology of the twisted complex over the ring of Laurent polynomials is isomorphic to that of the untwisted complex over the base field tensored with the ring of Laurent polynomials \cite[Lemma 2.2]{OSCube}. The~proof can actually be extended to singular
knots.
In case of a~singular link we can only identify complexes
when the~twisting sets are ``proportional'' on the~fundamental periodic domains
$\pi_1, \dots, \pi_r$ defined in \eqref{eq:domain-pi_j},
which are associated with the~components of the~link
that do not carry the~marking $\basepoint$.

\begin{proposition}\label{prop:untwisting}
	Let $S$ be a~singular link and suppose that $P$ and $P'$ are two
	sets of markings on a~Heegaard diagram for $S$,
	such that $P(\pi_i) = \lambda P'(\pi_i)$ for some fixed $\lambda\in\ZZ$
	and every fundamental periodic domain $\pi_i$.
	Then there is a~$\scalars$-linear isomorphism of complexes
	\[
		\Phi\colon \tCFK-(S, P, t) \xrightarrow{\ \cong\ } \tCFK-(S, P', t^\lambda),
	\]
	where the~left (resp.\ right) complex is twisted by $t$ (resp.\ $t^\lambda$).
	In particular, 
	\[
		\Phi\colon \tCFK-(S, P) \xrightarrow{\ \cong\ } \CFK-(S) \otimes_{\ZZ} \scalars
	\]
	for any singular knot $S$ and a~twisting set $P$.
\end{proposition}

As a~direct consequence, there is no essential difference between Ozsv\'ath--Szab\'o
twisting and its restriction to trace vertices in case of layered
diagrams as defined in Section~\ref{sec:cube-of-resolutions}.

\begin{corollary}\label{cor:OS-vs-tr-twisting}
	Let $S$ be a~layered diagram of a singular link with vertices at $n$ levels.
	Then there is an~isomorphism of twisted complexes
	\[
		\tCFK-(S, \OSTwist, t)
			\cong
		\tCFK-(S, \trTwist, t^n).
	\]
\end{corollary}
\begin{proof}
	We have $\OSTwist(\pi_i) = b_in$ and $\trTwist(\pi_i) = b_i$,
	where $b_i$ is the~index of the~component $S_i$.
\end{proof}

In order to proof Proposition~\ref{prop:untwisting} we need a~suitable
decomposition of the~twisted complex.

\begin{definition}
	Generators $\mathbf a = (U_0^{r^{\vphantom\prime}_0}\cdots U_{k+s}^{r_{k+s}})\xgen$
	and $\mathbf b =  (U_0^{r'_0}\cdots U_{k+s}^{r'_{k+s}})\ygen$
	are \emph{$W_0$-equivalent}, written $\mathbf a \sim \mathbf b$,
	if there is a~Whitney disk
	$\phi\in\pi_2(\xgen, \ygen)$ such that
	\begin{equation}\label{eq:disk-equivalence}
		O_i(\phi) = r'_i - r_i
	\quad\text{and}\quad
		X_j(\phi) = 0
	\end{equation}
	for each $X_j\in\markX$ and $O_i\in\markO$.
\end{definition}

The~$W_0$-equivalence is an~equivalence relation on the~set
\[
	\mathcal B = \left\{
		(U_0^{r_0^{\vphantom\prime}}\cdots U_{k+s}^{r_{k+s}}) \xgen
	\ |\ 
		\xgen \in \CFbasis, r_i\geqslant 0
	\right\}
\]
of $\scalars$-linear generators of the~Heegaard Floer complex
and it is compatible with the~differential. Hence, it induces
a~decomposition
\[
	\tCFK-(S,P) = \bigoplus_{r \in \mathcal B/_{\!\sim}} \tCFK-(S,P; r)
\]
parametrized by the~set of equivalence classes.
When $S$ is a~knot, then for any two generators of the~same Alexander degree
there is a~disk connecting them,\footnote{
	Indeed, such a~disk can be constructed by correcting any Whitney
	disk $\phi\in\pi_2(\xgen, \ygen)$ as follows.
	First, impose multiplicity zero at each $X_i$ by adding to $\mathcal D(\phi)$
	domains $A_i$ sufficiently many times.
	The~total multiplicity at $\markO$-basepoints is now equal
	to $\qdeg(\ygen) - \qdeg(\xgen) = \sum_i (r'_i - r_i)$,
	but local multiplicities at each $O_i$ may not match \eqref{eq:disk-equivalence}.
	This is fixed by travelling along the~knot and adding differences $A_i-B_i$
	to correct the~mutliplicity at each $O_i$, which affects the~multiplicity
	only at the~next basepoint.
	Once the~last basepoint is reached, the~local multiplicity is already as expected.
}
so that $\mathcal B/_{\!\sim} \approx \ZZ$.

\begin{proof}[Proof of Proposition~\ref{prop:untwisting}]
	Pick a~set $\mathcal R \subset \mathcal B$ of representants of $W_0$-equivalence classes.
	Given a~generator $\mathbf a = (U_0^{r_0^{\vphantom\prime}}\cdots U_{k+s}^{r_{k+s}}) \xgen$
	that is $W_0$-equivalent to a~chosen representant
	$(U_0^{r'_0} \cdots U_{k+s}^{r'_{k+s}}) \ygen \in \mathcal R$ pick a~Whitney disk
	$\phi\in\pi_2(\xgen,\ygen)$ satisfying \eqref{eq:disk-equivalence}.
	A~difference between any two such disks is a~linear combination
	of the~periodic domains $\pi_1, \dots, \pi_r$, so that the~value
	$P(\phi) - \lambda P'(\phi)$ is independent of the~choice of $\phi$.
	Hence,
	\[
		\Phi(\mathbf a) := t^{P(\phi) - \lambda P'(\phi)}\mathbf a
	\]
	is a~well-defined isomorphism of complexes.
	The case of a~singular knot follows, because there are no periodic
	domains $\pi_i$. In particular, the~disk $\phi$ is unique.
\end{proof}

The~assumption on connectivity of the~diagram is important for untwisting:
the~two complexes have drastically different homology when $S$ is a~split link.
For instance, the~twisted homology may vanish for fully singular links
with more than one component, see Proposition~\ref{prop:hfk-for-disconnected}
or the~proof of \cite[Proposition 3.4]{OSCube}.

 \subsection{Skein exact triangles}
\label{sec:skein-triangle}

Let us now recall the~skein exact triangle for Heegaard Floer homology.
In this section we use \emph{planar} Heegaard diagrams that look
locally as depicted in Figure~\ref{fig:planar-Heegaard-diagram}.
For simplicity, we start with the~untwisted complex and discuss
the~differences afterwards.

\begin{figure}[ht]
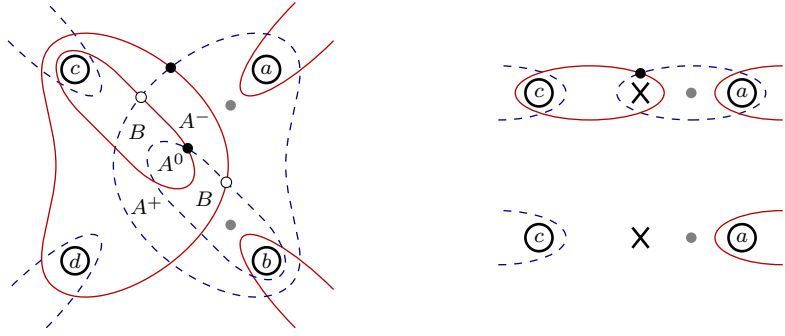

	\NB{\tikz[]{\input{\imagesfolder/hfgl0_planar-generic}}}\hskip 2cm
	\begin{tabular}{c}
		\NB{\tikz[]{\input{\imagesfolder/hfgl0_planar-trace}}} \\[8ex]
		\NB{\tikz[]{\input{\imagesfolder/hfgl0_planar-marking}}}
	\end{tabular}
	\caption{Local pictures of the~planar Heegaard diagram for a~resolution of a~singular link.
		Near a~smoothed resolution the~left diagram is used with one $\markX$-mark at
		each region decorated with $B$, whereas for a~positive (resp.\ negative) crossing
		we place $\markX$-marks at $A^0$ and $A^+$ (resp.\ $A^-$);
		the~local diagram for a~singular crossing is obtained by placing a~double
		$\markX$-mark at $A^0$ and removing the~pair of ellipses around it. For the convenience of the reader, these diagrams are repeated explicitly on Figure~\ref{fig:planar-Heegaard-diagram-2}. 
		The~top right configuration is used near a~bivalent vertex other
		than the~basepoint $\basepoint$, in which case
		the~bottom variant is used instead. The~marked intersection points
		are components of the~canonical generator $\xgen_0$: the~white points
		at smooth resolutions and the~black points elsewhere.
		Gray dots represent the~twisting markings due to Ozsv\'ath and Szab\'o.
	}
	\label{fig:planar-Heegaard-diagram}
\end{figure}

\begin{figure}[ht]
  	\NB{\tikz[]{\input{\imagesfolder/hfgl0_planar-generic-pos}}}\hskip 3cm
  	\NB{\tikz[]{\input{\imagesfolder/hfgl0_planar-generic-neg}}} \\ [0.5cm] \NB{\tikz[]{\input{\imagesfolder/hfgl0_planar-generic-sing}}} \hskip 3cm \NB{\tikz[]{\input{\imagesfolder/hfgl0_planar-generic-smooth}}}
         \caption{ From left to right and from top to bottom, 
           local pictures of the~planar Heegaard diagram for a positive crossing, a negative crossing, a singular crossing and a smoothing.
	}
	\label{fig:planar-Heegaard-diagram-2}
\end{figure}

\begin{lemma}
	The~planar Heegaard diagram associated with a~singular link is admissible.
\end{lemma}
\begin{proof}
	Suppose that $\pi = \sum_i a_i(A_i - B_i)$ is a~positive periodic domain,
	where $A_0$ and $B_0$ are the~unbounded regions. Then $A_0$ intersects each
	$B_i$ forcing $a_0 \geqslant a_i$. Likewise, $B_0$ intersects each $A_i$
	forcing $a_0 \leqslant a_i$. Hence, all $a_i$ coincide and $\pi=0$.
\end{proof}

We distinguish a~generator $\xgen_0 = \xgen_0(S) \in \CFK-(S)$ that is given
by the~intersection points marked on Figure~\ref{fig:planar-Heegaard-diagram}
with black dots except neighborhoods of smoothed resolutions, where
white points are taken instead.
The~gradings can be normalized by specifying
\begin{equation}\label{eq:deg-normalization-for-x0}
	\hdeg(\xgen_0) = 0
	\quad\text{and}\quad
	\qdeg(\xgen_0) = s-n-1,
\end{equation}
where $n$ is the~number of real crossings in $S$ and $s$ is the~number of
Seifert circles, obtained by smoothing in $S$ all crossings
(both real and singular).
The~first choice follows from the~observation below,
whereas the~normalization of the~quantum grading is justified
at the~end of this section.

\begin{lemma}\label{lem:x0-free-in-HFK}
The~generator $\xgen_0$, when considered as an~element of $\widehat{\mathit{CF}}=\CFK-(S)/(U_i=1)$,
	is a~cycle that generates the~top degree homology.
\end{lemma}
\begin{proof}
	Setting $U_i=1$ for all $i$ allows us to forget the~$\markO$-basepoints.
	The~Heegaard diagram can be then reduced, so that each $\alpha$-curve
	intersects only one $\beta$-curve, exactly in two points,
	one of which is a~component of $\xgen_0$. In fact, $\xgen_0$
	is the~top degree generator of the~associated Heegaard Floer complex,
	hence, a~cycle.
\end{proof}

We call $\xgen_0$ the~\emph{canonical generator}\footnote{
	Generally, $\xgen_0$ is not a~cycle in $\CFK-(S)$---for instance,
	the~rectangle decorated $A^-$ in a~neighborhood of
	a~positive crossing represents a~nontrivial holomorphic disk
	from $\xgen_0$.
}
of $\CFK-(S)$.
Because  
 $\widehat{\mathit{HF}}$ 
is the homology of the $k$-torus \cite[Lemma 9.1]{OS3mflds} and hence
 free,
we obtain a~following generalization of Lemma~\ref{lem:signs-for-AB-pair}
that will play an~important role in the~next section.
The~case $k=2$ has been proven in an~unpublished version of \cite{OSCube}.

\begin{corollary}\label{cor:sum-of-signs-for-x0}
	Choose an~intersection point $\ygen \in \tCFK-(S)$ with
	$\hdeg(\ygen) = \hdeg(\xgen_0)-1$. Then
	\begin{equation}\label{eq:vanishing-sum-of-disks}
		\#\moduliR(\phi_1) + \ldots + \#\moduliR(\phi_k) = 0,
	\end{equation}
	where $\phi_1,\dots,\phi_k$
	are all classes of Whitney disks from $\ygen$ to $\xgen_0$
	that have Maslov index one and are disjoint from $\markX$.
\end{corollary}
\begin{proof}
	The~homology class of $\xgen_0$ cannot be free in $\widehat{\mathit{HF}}$
	when the~sum in \eqref{eq:vanishing-sum-of-disks} does not vanish.
\end{proof}

Let $S_+$, $S_-$, $S_\times$, $S_0$ be diagrams of singular links that differ only
in a~small neighborhood of a~point $p$, where the~first three have respectively
a~positive, a~negative and a~singular crossing, whereas in $S_0$ 
the crossing is smoothed.
Write $k$ (resp.\ $s$) for the~number of $\markX$-basepoints (resp.\ double
$\markX$-basepoints) in the~associated planar Heegaard diagrams
and let $R = \scalars[U_0,\dots,U_{k+s}]$.
We write
$A_a$ and $A_b$ (resp.\ $B_c$ and $B_d$) for the~regions bounded by $\alpha$-curves
(resp.\ $\beta$-curves) that contain the~$\markO$-basepoints
labeled $a$ and $b$ (resp.\ $c$ and $d$).
Let $U^{(p)}_a$ and $U^{(p)}_b$ (resp.\ $U^{(p)}_c$ and $U^{(p)}_d$)
be the~associated variables and consider a~two term complex
\[
	\mathcal L_p = \Big( \qshift R
		\xrightarrow{U^{(p)}_a + U^{(p)}_b - U^{(p)}_c - U^{(p)}_d}
	\qshift^{-1} R \Big)
\]
generated by $\mathbf u$ and $\mathbf 1$ in homological degrees $-1$ and $0$ respectively.

\begin{theorem}[cf.\ \cite{OSCube}]\label{thm:CFK-cube}
	There are homotopy equivalences of complexes
	\begin{align*}
		\CFK-(S_+) &\simeq \tshift^{-1} \left(
			\CFK-(S_\times) \otimes_R \mathcal L_p
				\xrightarrow{\ \piz_p\ }
			\qshift^{-1} \CFK-(S_0)
		\right)
	\\
		\CFK-(S_-) &\simeq \left(
			\qshift \CFK-(S_0)
				\xrightarrow{\ \zip_p\ }
			\CFK-(S_\times) \otimes_R \mathcal L_p
		\right)
	\end{align*}
	where $\piz_p(\xgen_0(S_\times) \otimes \mathbf 1) = \xgen_0(S_0)$
	and $\zip_p(\xgen_0(S_0)) = (U^{(p)}_b - U^{(p)}_c) (\xgen_0(S_\times) \otimes \mathbf 1)$.
      \end{theorem}

\begin{proof}
	Consider first the~case of $S_-$, so that the~Heegaard diagram for $\CFK-(S_-)$
	near $p$ has $\markX$ basepoints at regions marked $A^0$ and $A^-$.
	Generators containing the~intersection point between the~regions marked $A^0$ and $A^-$
	span a~subcomplex $X_-$ that contains $\xgen_0(S_-)$ and the~differential of which
	counts holomorphic disks with multiplicity 0 at the~five marked regions.
	Let $Y_-$ be the~quotient complex, spanned by the~remaining generators,
	and consider the~diagram
\[
	\begin{tikzpicture}[x=10em,y=12ex]
		\node[anchor=mid] (00) at (0 ,0) {$\qshift^2\tshift^2 X_-$};
		\node[anchor=mid] (01) at (1, 0) {$\qshift^2\tshift X_-$};
		\node[anchor=mid] (10) at (0,-1) {$Y_-$};
		\node[anchor=mid] (11) at (1,-1) {$X_-$};
		\path (00) edge[->] node[midway,above] {$\scriptstyle\id$} (01)
		           edge[->] node[midway,above] {$\scriptstyle \Phi_{A^-B}$} (11)
		           edge[->] node[midway,left] {$\scriptstyle \Phi_{A^-}$} (10)
		      (11) edge[<-] node[midway,above] {$\scriptstyle \Phi_B$} (10)
		           edge[<-] node[midway,right] {$\scriptstyle
		           		U^{(p)}_a + U^{(p)}_b - U^{(p)}_c - U^{(p)}_d
		           $} (01);
	\end{tikzpicture}
	\]
	where the~maps $\Phi_B$, $\Phi_{A^-}$ and $\Phi_{A^-B}$ count Maslov index one
	holomorphic disks $\phi$, such that
	\begin{itemize}
		\item $B(\phi) = 1$ and $A^-(\phi) = A^0(\phi) = 0$ in case of $\Phi_B$,
		\item $B(\phi) = 0$ and $A^-(\phi) + A^0(\phi) = 1$ in case of $\Phi_{A^-}$,
		\item $B(\phi) = 1$ and $A^-(\phi) + A^0(\phi) = 1$ in case of $\Phi_{A^-B}$,
	\end{itemize}
	and $B(\phi)$ is the~total multiplicity of $\phi$ at both regions labeled $B$.
	Considering ends of moduli spaces of holomorphic disks of Maslov index two,
	we get that the~total of the~diagram is a~chain complex, where the~terms
	$U^{(p)}_a$, $U^{(p)}_b$, $U^{(p)}_c$, and $U^{(p)}_d$ come from degenerated disks represented
	by the~domains $A_a$, $A_b$, $B_c$ and $B_d$ respectively, which explains the~signs.
	The~total chain complex is clearly homotopy equivalent to the~mapping
	cone of $\Phi_B$, which is $\CFK-(S_-)$.
	The~right column is identified with $\CFK-(S_\times) \otimes_R \mathcal L_p$
	by forgetting the~fixed intersection point.
	This takes $\xgen_0(S_-) \in X_-$, which is in degree $s-n-1$,
	to $\xgen_0(S_\times) \in \CFK-(S_\times)$, which lives in degree $s-n$.
	The~left column is identified with $\qshift\tshift \CFK-(S_0)$,
	in which $\xgen_0(S_0)$ is identified with a~generator $\ygen_0^- \in Y_-$
	that is given by the~same collection of intersection points, but living
	in homological degree $-1$ and quantum degree $s-n+1$
	(apply \eqref{eq:qdeg-for-disk} and \eqref{eq:hdeg-for-disk} to the~rectangle marked $A^-$).
	There are two Whitney disks from $\ygen_0^-$ to $\xgen_0(S_-)$
	with multiplicity one at $B$, represented by domains $A_b\setminus B_c$
	and $B_c\setminus A_b$. Hence, $\Phi_B(\ygen_0^-) = \pm(U^{(p)}_b - U^{(p)}_c)\xgen_0(S_-)$,
	which is compatible with the~formula for $\zip_p$.
	
	The~case of $S_+$ is similar.
	This time the~diagram has basepoints at regions $A^0$ and $A^+\ric$.
	Gene\-rators containing the~intersection point between the~two regions
	span a~quotient complex $X_+$ of $\CFK-(S_+)$,
	which is identified with $\qshift\tshift\CFK-(S_\times)$:
	the~canonical generator $\xgen_0(S_\times)$ corresponds to $\xgen_0^+ \in X_+$
	that differ from $\xgen_0(S_+)$ by picking the~other corner of the~bigon labelled $A^0$.
	In particular, $\hdeg(\xgen_0^+) = -1$ and $\qdeg(\xgen_0^+) = s-n+1$.
	Writing $Y_+$ for the~subcomplex spanned by other generators, we consider the~diagram
	\[
	\begin{tikzpicture}[x=10em,y=12ex]
		\node[anchor=mid] (00) at (0, 1) {$X_+$};
		\node[anchor=mid] (01) at (1, 1) {$Y_+$};
		\node[anchor=mid] (10) at (0 ,0) {$\qshift^{-2}\tshift^{-1} X_+$};
		\node[anchor=mid] (11) at (1, 0) {$\qshift^{-2}\tshift^{-2} X_+$};
		\path (00) edge[->] node[midway,above] {$\scriptstyle\Phi_B$} (01)
		           edge[->] node[midway,above] {$\scriptstyle \Phi_{A^+B}$} (11)
		           edge[->] node[midway,left] {$\scriptstyle
		           	U^{(p)}_a + U^{(p)}_b - U^{(p)}_c - U^{(p)}_d
		           $} (10)
		      (11) edge[<-] node[midway,above] {$\scriptstyle \id$} (10)
		           edge[<-] node[midway,right] {$\scriptstyle \Phi_{A^+}$} (01);
	\end{tikzpicture}
	\]
	where the~maps $\Phi_B$, $\Phi_{A^+}$ and $\Phi_{A^+B}$ are analogues of the~maps from
	the~case of a~negative crossing, but using the~region marked $A^+$ instead of $A^-\ric$.
	Again, the~total object is a~chain complex that is homotopy equivalent to the~mapping
	cone of $\Phi_B$, which is $\CFK-(S_+)$.
	The~left column is identified with $\CFK-(S_\times) \otimes_R \mathcal L_p$
	and the~right one with $\qshift^{-1}\tshift^{-1} \CFK-(S_0)$.
	The~applied degree shifts follow from the~observation
	that $\xgen_0^+$, when considered as a~generator of $\CFK-(S_0)$,
	lives in the~same homological and quantum degree as $\xgen_0(S_0)$,
	because both are connected by the~domain $A^- - A^0$.
	The~property of the~unzip map follows easily.
\end{proof}

\begin{remark}
	Theorem~\ref{thm:CFK-cube} remains true for twisted complexes once powers of $t$
	are properly distributed in the~formulas for the~differential in $\mathcal L_p$
	as well as for the~zip map. Namely, each $U^{(p)}_x$ must be scaled by $t^{m_x}\!$,
	where $m_x$ counts twisting markings in the~region $A_x$ or $B_x$ (depending on $x$).
	In particular, the~formulas are unchanged when twisted by $\trTwist\!$, whereas
	each $U^{(p)}_a$ and $U^{(p)}_b$ is scaled by $t$ in case of the~Ozsv\'ath--Szab\'o
	marking $\OSTwist$.
\end{remark}

Theorem~\ref{thm:CFK-cube} implies immediately that the~degree normalization
\eqref{eq:deg-normalization-for-x0} matches the~standard one: the~polynomial
\[
	(\qvar - \qvar^{-1})^\sigma \sum_{d,s} (-1)^d \qvar^s \mathrm{rk} \HFK_d(S; s)
\]
assigned to a~diagram $S$, where $\sigma$ counts singular crossings,
satisfies the~skein relation of the~Alexander polynomial.

 \subsection{Computation for planar singular links}
Let $S$ be a~planar singular knot considered as a~diagram with no real crossings.
Recall that a~Kauffman state of such diagram is a~collection of markings
at singular crossing as shown in Figure~\ref{fig:kauffman-states}, such that
each region not adjacent to $\basepoint \in S$ contains exactly one marking.
Because switching $D^-$ with $D^+$ has the~effect of replacing $\qvar$ with $1/\qvar$,
the~Alexander polynomial $\Delta_S(\qvar)$ is symmetric.
It can be shown that the~minimal power of $\qvar$ is equal to $s-n-1$,
where $s$ is the~number of Seifert circles in $S$ and $n$ is the~number
of singular crossings.

\begin{figure}[ht]
	\NB{\tikz[]{\input{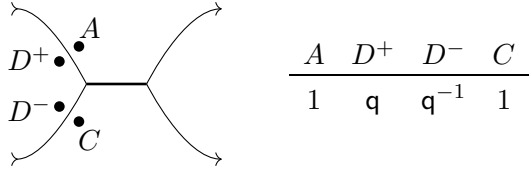}}}
	\qquad
	\begin{tabular}{cccc}
		$A$ & $D^+$ & $D^-$ & $C$
	\\ \hline
		$1$ & $\qvar$ & $\vphantom{\Big|}\qvar^{-1}$ & $1$
	\end{tabular}
	\caption{Kauffman markings at a~singular crossing and their (multiplicative)
		contributions towards the~evaluation of the~Kauffman state.
	}
	\label{fig:kauffman-states}
\end{figure}

In terms of intersection points on the~initial Heegaard diagram,
picking a~Kauffman state is equivalent to fixing points on
the~$\beta$-curves parallel to the~contours of the~underlying surface,
leaving a~choice between two intersection points on each internal $\beta$-curve,
see Figure~\ref{fig:kauffman-state-generator}.
Thus, there are $2^n$ generators associated with a~fixed Kauffman state,
where $n$ is the~total of singular crossings and bivalent vertices in $S$
other than $\basepoint$.
\begin{figure}[ht]
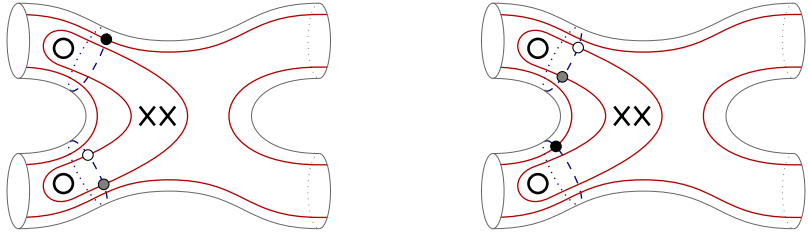

	\def\KauffmanState{A}\NB{\tikz[]{\input{\imagesfolder/hfgl0_hd-singular-kauffman}}}\hskip 2cm
	\def\KauffmanState{nD}\NB{\tikz[]{\input{\imagesfolder/hfgl0_hd-singular-kauffman}}}\caption{Local pictures for generators associated with the~$A$ state
		(to the~left) and the~$D^-$ state (to the~right).
		The~black dot on one of the~meridians is fixed by the~state,
		which leaves a~choice between the~white and gray point on
		the~other meridian. Choosing always the~white point produces
		the~generator of the~highest homological degree.
	}
	\label{fig:kauffman-state-generator}
\end{figure}
It is shown in \cite{OSSSingular} that with each Kauffman state there is
associated a~unique generator of highest (in the~conventions of this paper)
homological degree, represented in Figure~\ref{fig:kauffman-state-generator}
by black and white dots. This generator is actually in homological degree 0
and its quantum grading matches the~contribution of the~associated Kauffman
state to the~Alexander polynomial.
Our main goal is to provide an~argument that no other
generators contribute towards the~twisted homology.

Consider now the~initial diagram for $S$ together with
the~Ozsv\'ath--Szab\'o twisting $\OSTwist\!$.
As in the~previous section, given a~singular crossing $p$
we denote the~$\markO$-basepoints on outgoing (resp.\ incoming) arcs
by $O^{(p)}_a$ and $O^{(p)}_b$ (resp.\ $O^{(p)}_c$ and $O^{(p)}_d$).
Likewise, when $p$ is a~bivalent vertex, then $O^{(p)}_a$ and $O^{(p)}_c$
are located at the~outgoing and incoming arc respectively.
Define
\[
	\underline{\mathcal L}_S := \bigotimes_{p\in\setcrossings} \left(
		R \xrightarrow{\ tU^{(p)}_a + tU^{(p)}_b - U^{(p)}_c - U^{(p)}_d\ } R
	\right)
\]
as the~tensor product over $R$ of linear complexes taken over all singular
crossings $p$ (where, again, we write $R$ for the~$\scalars$-algebra
of polynomials in all variables $U_i$).

For a $\scalars$-module $M$, its $\scalars$-torsion is the $\scalars$-submodule $\{ m \in M | \exists \lambda \in \scalars \;\; \text{with}\; \lambda m =0\}$.

\begin{proposition}\label{prop:hfk-for-connected-planar}
	Let $S$ be a~planar singular knot.
	Then the quotient of
	$H_*(\tCFK-(S, \OSTwist) \otimes_R \underline{\mathcal L}_S)$
	 by its  $\scalars$-torsion is a~free\/ $\scalars[U_0]$-module concentrated
	in homological degree zero and generated by Kauffman states of $S\ric$.
	The~same holds for $H_*(\tCFK(S, \OSTwist)\otimes_R \underline{\mathcal L}_S)$
	with\/ $\scalars$ in place of\/ $\scalars[U_0]$.
\end{proposition}
\begin{proof}
	Without loss of generality we can assume that $\scalars = \ZZ[t,t^{-1}]$.
	We further extend the~ring by a~square root of $t$ and extend the~set
	of markings as shown in Figure~\ref{fig:extra-twisting}. 
\begin{figure}[ht]
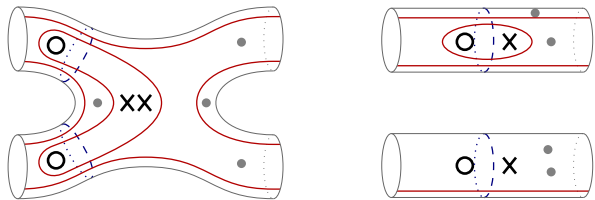

		\def\diagmarkinglevel{2}\NB{\tikz[scale=0.85]{\input{\imagesfolder/hfgl0_marked-hd-singular}}}
			\hskip 1cm
		\begin{tabular}{c}
			\NB{\tikz[scale=0.85]{\input{\imagesfolder/hfgl0_marked-hd-bivalent}}}
		\\[1cm]
			\NB{\tikz[scale=0.85]{\input{\imagesfolder/hfgl0_marked-hd-marking}}}
		\end{tabular}
		\caption{The~extra twisting markings on an~initial Heegaard diagram.
		}
		\label{fig:extra-twisting}
	\end{figure}
Write $Q$ for this new set. Because $S$ is connected,
	the~complex $\tCFK-(S,Q; t^{1/2})$ is isomorphic to $\tCFK-(S,\OSTwist; t)$
	by Proposition~\ref{prop:untwisting}.
	Hence, it suffices to prove the~thesis for the~twisting set $Q$.

	The~complex $\tCFK(S,Q)$	is filtered with respect to the~power of $t$
	and the~$t$-degree zero part of the~differential relates only generators associated
	with the~same Kauffman state.
	In fact, the~component of the~graded associate complex spanned
	by generators corresponding to a~fixed Kauffman state $\mathfrak s$
	has the~form of the~tensor product over $R$
	\[
		\mathcal N_{\mathfrak s} = \bigotimes_{p \neq \basepoint}
			\Big( R \xrightarrow{\ U^{(p)}_x\ } R \Big)
	\]
	taken over all singular crossings and bivalent vertices of $S$
	other than the~basepoint $\basepoint$,
	and where $x=c$ or $d$ depending on $p$ and $\mathfrak s$.
	We claim that the~homology of
	\begin{equation}\label{eq:H-for-KS}
		\mathcal N_{\mathfrak s} \otimes_R \mathrm{gr}\underline{\mathcal L}_S
			\cong
		\bigotimes_{p \neq \basepoint} \Big( R \xrightarrow{\ U^{(p)}_x\ } R \Big)
			\otimes_R
		\bigotimes_{p \in \setcrossings} \Big( R \xrightarrow{\ U^{(p)}_c + U^{(p)}_d\ } R \Big)		
	\end{equation}
	is freely generated by the~highest homological degree generator
	associated with $\mathfrak s$.
	This follows from the~observation that the~set of relations
	\begin{equation}\label{eq:relations-in-HFK0}
		\big\{ U^{(p)}_x \ \big|\ p \neq \basepoint \big\}
			\cup
		\big\{ U^{(p)}_c + U^{(p)}_d\ \big|\ p \in \setcrossings \big\}
	\end{equation}
	is regular, i.e.\ each element is a~non-zero divisor in the~quotient
	of $R$ by other elements.
	In particular, the~relations eliminate all variables except $U_0$.
	Following the~proof of \cite[Proposition 3.4]{OSCube} we show that
	the~generator of \eqref{eq:H-for-KS} is in homological degree $0$,
	i.e.\ its Maslov degree is twice the~Alexander degree,
	so that the~spectral sequence associated with the~filtration collapses
	immediately. This shows that each component
	$N_{\mathfrak s}\otimes \mathrm{gr}\underline{\mathcal L}_S$ contributes
	exactly one free generator towards the~ $E^\infty$ page.
	
	The~uniqueness of a~limit of a~spectral sequence (compare Appendix \ref{subsec:uniqueness})
	implies that
	\[
		H_*(\tCFK-(S,Q) \otimes_R \underline{\mathcal L}_S),
	\]
	when computed over the~completed ring $\ZZ[t^{-1/2},t^{1/2}]][U_0]$,
	is freely generated by Kauffman states and concentrated in homological degree 0.
	This proves the~statement for the~twisting set $Q$ and the~case
	of $\OSTwist$ follows from Proposition~\ref{prop:untwisting}.
	Finally, the~computation for $\tCFK$ follows, because adding $U_0$
	to \eqref{eq:relations-in-HFK0} does not affect the~regularity of the~set.
\end{proof}

Let us now discuss briefly the~case of disconnected diagrams.
In \cite{OSCube} it is shown that the~twisted homology vanishes,
because the~set of generators is empty (the~initial diagram used
in the~paper is not admissible for disconnected diagrams).
This argument is no longer true for our initial diagram,
but the~vanishing result (up to $\scalars$-torsion) still holds.

\begin{proposition}\label{prop:hfk-for-disconnected}
	Suppose that $S$ is a~planar singular link with at least two components.
	Then both $H_*(\tCFK-(S, \OSTwist) \otimes_R \underline{\mathcal L}_S)$
	and $H_*(\tCFK(S, \OSTwist)\otimes_R \underline{\mathcal L}_S)$
	are $\scalars$-torsion.
\end{proposition}
\begin{proof}
	We may assume as before that $\scalars = \ZZ[t^{1/2},t^{-1/2}]$
	and consider the~twisting set $Q$.
	For any component $S_i$ of $S$ and the~associated
	periodic domain $\pi_i$ the~equality $Q(\pi_i) = 2\OSTwist(\pi_i)$ holds,
	so that Proposition~\ref{prop:untwisting} provides again an~isomorphism
	between $\tCFK-(S,Q; t^{1/2})$ and $\tCFK-(S,\OSTwist; t)$.
	Because $S$ has at least two
	components, the~initial diagram contains a~neck
	connecting surfaces built for different components of $S$.
	Figure~\ref{fig:kauffman-state-for-links} shows four
	possibly choices of intersection points near such a~neck
	(a~black point paired with either a~white or a~gray point
	from the~same picture). Notice that the~top intersection
	points on the~two meridians cannot be picked, because
	the~intersecting $\beta$ curve must carry an~intersection
	point with some meridian around the~upper component of $S$.
		
	\begin{figure}[ht]
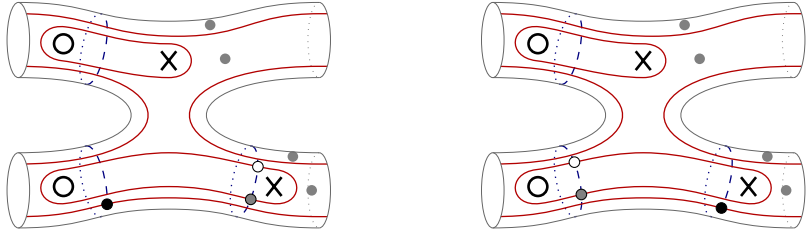

		\def\diagmarkinglevel{2}\def\KauffmanState{C1}\NB{\tikz[]{\input{\imagesfolder/hfgl0_hd-two-components-kauffman}}}\hskip 2cm
		\def\KauffmanState{C2}\NB{\tikz[]{\input{\imagesfolder/hfgl0_hd-two-components-kauffman}}}\caption{Local pictures for generators near a~neck connecting
			different components of a~singular link $S$.
		}
		\label{fig:kauffman-state-for-links}
	\end{figure}

	Consider now the~graded associate $\mathrm{gr}\tCFK-(S,Q)$.
	The~differential counts not only bigons carrying
	the~$\markO$-basepoints, but also two rectangles
	near each connecting neck. In particular, the~complex
	admits a~tensor factor of the~form
	\[
		R \xrightarrow{\ \left(\begin{smallmatrix}1 \\ U\end{smallmatrix}\right)\ }
		R\oplus R \xrightarrow{\ \left(\begin{smallmatrix}U & -1\end{smallmatrix}\right)\ }
		R,
	\]
	where $U$ is the~variable associated with the~lower $\markO$-basepoint
	in the~picture.
	This complex is acyclic, and so is the~entire graded associate complex.
	Thus, $\tCFK-(S,Q)$ is $\scalars$-torsion.
	The~statement for $\tCFK-(S;\OSTwist)$ follows
	from Proposition~\ref{prop:untwisting}.
	The~argument carries over to $\tCFK$ with no change.
\end{proof}

\section{Main results}\label{sec:main}
In this section we prove the~main results of this paper.
We start by defining \emph{Gilmore space} as a~quotient
of the~reduced Soergel space $\set{qSoergelRed}(\web)$ 
associated with a~pointed annular web $\web$ by non-local relations.
By inserting this space at vertices in the cube construction
we produce the~\emph{Gilmore complex} $\complexgilmore(\bdiagram)$.
In Theorem \ref{thm:gilmore-vs-hfk} we prove that 
the~complex is quasi-isomorphic to the twisted Heegaard Floer complex
$\tCFK(\bdiagram)$ if $\scalars = \ZZ[q^{-1},q]]$.
In Section \ref{sec:pseudo-completion} we define
our algebro-geometric quotient $\ourQ(\web)$ of the~Gilmore space.
By applying the~Bockstein spectral sequence to $\ourQ(\bdiagram)$
we prove Theorems \ref{thm:main} and \ref{thm:HHHHFK}.

\subsection{The~normalized Gilmore complex}
\label{sec:gilmore-revised}

Choose a~pointed annular web $\web$ and recall that the~marking
$\basepoint$ is on an~edge of thickness 1 that is at the~same
time an~outer edge.
We say that a~simple closed curve $\gamma$ is \emph{adapted to $\web$}
if it avoids vertices of the~web, intersects its edges transversally,
and the~region $R_\gamma$ bounded by the~curve does not contain the~marking $\basepoint$,
see Figure~\ref{fig:adapted-curve}.
The~intersection points between $\web$ and $\gamma$ fall
into two categories: \emph{incoming} and \emph{outgoing} points,
at which the~web is oriented inwards and outwards the~region $R_\gamma$
respectively.

\begin{figure}[ht]
	\NB{\tikz[]{\input{\imagesfolder/hfgl0_example-resolution-nonlocal}}}
	\caption{A~curve $\gamma$ adapted to the~web from Figure~\ref{fig:example-resolution}
		and the~bounded region $R_\gamma$. There are four incoming edges
		and three outgoing edges, one of which it thick.
		A~curve adapted to $\web$ can have turn-backs and it can cross
		an~edge more than once.
	}
	\label{fig:adapted-curve}
\end{figure}

In Section~\ref{sec:annulus} we have associated with a~pointed annular
web $\web$ the~polynomial algebra
\[
	\set{qSoergelRed}(\web) = 
	\qHoHom_0(R^{\underline k}; \set{Soergel}(\tilde\web))/(x_\basepoint),
\]
where $x_\basepoint$ is the~variable associated with the~edge terminating
at the~basepoint and $\tilde\web$ is the~directed web obtained by
cutting $\web$ along the~trace section.
Consider the~ideal $\nonlocalrelations_\web \subset \set{qSoergelRed}(\web)$
of non-local relations defined as follows.
Pick a~curve $\gamma$ adapted to $\web$ and write $\topesym_p$
for the~product of variables associated with the~edge containing
the~intersection point $p \in \web \cap \gamma$. Define
\[
	x_{\inp(\gamma)} := \prod_{p \in (\web\cap\gamma)^+} \topesym_p
\quad\text{and}\quad
	x_{\outp(\gamma)} := \prod_{p \in (\web\cap\gamma)^-} \topesym_p,
\]
where $(\web\cap\gamma)^+$ and $(\web\cap\gamma)^-$ are respectively
the~sets of incoming and outgoing intersection points, and put
\begin{equation}\label{eq:non-local-normalized}
	\nlrel_\gamma := x_{\outp(\gamma)} - q^{2i} x_{\inp(\gamma)},
\end{equation}
where $i$ is the~number of trace vertices in $R_\gamma$.
Note that $\gamma$ may intersect an~edge several times, in which case
the~variables associated with such an~edge appear in both products,
possibly with exponents bigger than 1. The~ideal $\nonlocalrelations_\web$
is generated by $\nlrel_\gamma$ for all such curves $\gamma$.

\begin{definition}\label{def:normalized-gilmore}
	The~quotient space
	\[
		\gilmoreA(\web) = \quotient
			{\set{qSoergelRed}(\web)}
			{\nonlocalrelations_\web}
	\]
	assigned to a~pointed annular web $\web$
	is called the~\emph{Gilmore space} of $\web$.
\end{definition}

Following the~common practice we write $\gilmoreA(\web; \scalars)$
to emphasize the~choice of coefficients.

\begin{example}\label{exa:GilmoreA-elementary-web}
	When $\web$ is an~elementary web, then $\gilmoreA(\web)$
	is generated by variables $x_i$ associated to its thin edges modulo
	the~following local relations
	\begin{center}
		\begin{tabular}{cp{7mm}cp{7mm}cp{7mm}c}
			\NB{\tikz[font=\small]{
				\draw[>->] (0,0) -- (1,0)
				    node[pos=0, left]  {$x_c$}
				    node[pos=1, right] {$x_a$};
				\draw (0.5,-2pt) -- ++(0,4pt);
	         }}
	         &&
			\NB{\tikz[font=\small]{
				\draw[densely dashed, \colormembrane, line width=0.3pt] (0.5, 0.5) -- (0.5, -0.5);
				\draw[>->] (0,0) -- (1,0)
				    node[pos=0, left]  {$x_c$}
				    node[pos=1, right] {$x_a$};
				\draw (0.5,-2pt) -- ++(0,4pt);
	         }}
	         &&
			\NB{\tikz[font=\small]{
				\draw[densely dashed, \colormembrane, line width=0.3pt] (0.5, 0.5) -- (0.5, -0.5);
				\draw[>->] (0,0) -- (1,0)
					node[pos=0, left]  {$x_c$}
					node[pos=1, right] {$x_a$}
					node[midway, font =\normalsize, fill=white, inner sep=1pt] {$\basepoint$};
			}}
			&&
	         \NB{\tikz[font=\small]{
	         	\draw[>-<]
	         		(-0.75, 0.5) node[left] {$x_c$}
	         			.. controls ++(0.25,0) and ++(0,0) ..
	         		(-0.25, 0) .. controls ++(0,0) and ++(0.25,0) ..
	         		(-0.75,-0.5) node[left] {$x_d$};
	         	\draw[<->]
	         		(0.75, 0.5) node[right] {$x_a$}
	         			.. controls ++(-0.25,0) and ++(0,0) ..
	         		(0.25, 0) .. controls ++(0,0) and ++(-0.25,0) ..
	         		(0.75,-0.5) node[right] {$x_b$};
		  		\draw[line width=2pt] (0.25,0) ++(0.5pt,0) -- (-0.25,0) -- ++(-0.5pt,0);
			}}
			\\[5ex]
			$x_a = x_c$ && $x_a = q^2 x_c$ && $x_a = 0$ && $x_a+x_b = x_c+x_d$
			\\[1ex]
			&&          &&             && $x_a x_b = x_c x_d$
			\\[1.5ex]
		\end{tabular}
	\end{center}
	and non-local relations $\nlrel_\gamma$ for curves $\gamma$ adapted to $\web$
	that do not intersect thick edges.
	Note the~special role of the~marking $\basepoint$: we do not enforce $x_c = 0$,
	which holds in $\set{qSoergelRed}(\web)$.
	This follows from the~non-local relation associated with a~small loop
	around the~marking, because it forces $x_a$ and $x_c$ to be~proportional.
\end{example}

\begin{example}\label{exa:GilmoreA-chain-dumbells}
  If $\web$ is a~chain of dumbbells (see Figure~\ref{fig:chain-of-dumbbells}),
  then $\gilmoreA(\web) \cong \scalars$ is generated by the constant polynomial
  if $1-q^d$ is invertible for each $d>1$. To see this,
  assign to thin edges of $\web$ variables $x_i$, $y_i$, and $z_i$
  for $i=1,\dots,k$, so that at the~$i$-th thick edge we have
  the~following situation:
  \[
	\begin{tikzpicture}[font=\small]
  		\draw[>-<]
  			(-1.5, 0.75) node[left] {$z_i$} 
  				.. controls ++(0.5,0) and ++(0,0) .. (-0.5, 0)
  				.. controls ++(0,0) and ++(0.5,0) .. (-1.5,-0.75) node[left] {$x_{i+1}$};
  		\draw[<->]
  			(1.5, 0.75) node[right] {$y_i$} 
  				.. controls ++(-0.5,0) and ++(0,0) .. (0.5, 0)
  				.. controls ++(0,0) and ++(-0.5,0) .. (1.5,-0.75) node[right] {$z_{i+1}$};
  		\draw[line width=2pt] (0.5,0) ++(0.75pt, 0) -- (-0.5,0) -- ++(-0.75pt, 0);
  		\draw[\colorgamma, densely dashed]
  			(-1.5, 0.2) .. controls ++(1,0) and ++(-0.5,0) .. (0, 0.6)
  				node[above] {$\gamma'_i$}
  				.. controls ++(0.5, 0) and ++(-1, 0) .. (1.5, 0.2)
  			(-1.5,-0.2) .. controls ++(1,0) and ++(-0.5,0) .. (0,-0.6)
  				node[below] {$\gamma_i$}
  				.. controls ++(0.5, 0) and ++(-1, 0) .. (1.5,-0.2);
	\end{tikzpicture}
  \]
  where the~curves $\gamma_i$ and $\gamma'_i$ have no more intersections with $\web$
  and the~edges with variables $x_i$ and $y_i$ meet at a~trace vertex, so that
  $x_i = q^2y_i$. It is understood that $z_1 = x_1$ and $z_k = y_k$.
  The~non-local relations associated with curves $\gamma_i$
  and $\gamma'_i$ forces $z_i = q^{2i-2k}y_i$ for each $i$.
  Substituting that in the~linear local relation
  \[
  	z_i + x_{i+1} = y_i + z_{i+1}
  \]
  forces $(q^{2i-2k}-1)(y_i - q^2y_{i+1}) = 0$,
  so that all variables are proportional to each other.
  In particular, to $y_1$, which is killed by the~basepoint relation.
  Finally, since there is no non-trivial relation involving polynomials
  of degree $0$, one has $\gilmoreA(\web) \cong \scalars$ as claimed.
\end{example}

\begin{figure}[ht]
  \[
    \NB{\tikz[font= \tiny ,scale=0.5]{\input{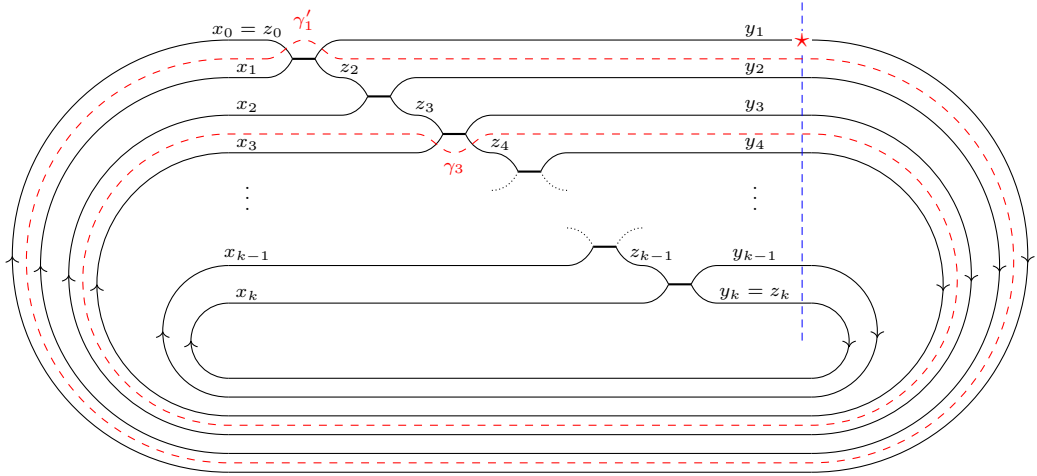}}}
  \]    
  \caption{A~pointed chain of dumbbells with a~trace section.}
  \label{fig:chain-of-dumbbells}
\end{figure}

\begin{example}\label{exa:GilmoreA-disconnected-web}
	Suppose that $\web$ is a~disjoint union of webs $\web_0, \dots, \web_r$,
	positioned so that $\web_i$ is surrounded by $\web_{i-1}$ for $i=1,\dots,r$,
	and write $w_i$ for the~index of the~component $\web_i$.
	Consider a~loop $\gamma_i$ separating $\web_{i-1}$ from $\web_i$;
	the~associated non-local relations forces $1 = q^{w_i + \dots + w_r}$.
	Hence, $\gilmoreA(\web)$ is annihilated by $1 - q^{\gcd(w_1,\dots,w_r)}$.
	In particular, the~space vanishes when $1-q^d$ is invertible for all $d > 0$.
\end{example}

\begin{proposition}\label{prop:gilmore-functor}
	The~assignment $\web \mapsto \gilmoreA(\web)$ extends to a~functor
	\[
		\gilmoreA\colon \cat{qAFoam}^\basepoint \to \cat{grMod}
	\]
	that is a~quotient of the~functor $\set{qSoergelRed}$ from Section~\ref{sec:annulus}.
\end{proposition}

In order to prove the~proposition, we need the~following property of non-local relations.

\begin{lemma}\label{lem:nlrel-vertex-move}
	Let $\gamma$ and $\gamma'$ be curves adapted to a~pointed annular web $\web$
	that coincide everywhere except a~small neighborhood of a~vertex $v$,
	in which $\gamma$ intersects only the~incoming edges,
	whereas $\gamma'$ intersects the~outgoing edges.
	Then $\nlrel_\gamma = \nlrel_{\gamma'}$ in $\set{qSoergelRed}(\web)$.
\end{lemma}
\begin{proof}
	The~only difference between $\nlrel_{\gamma}$ and $\nlrel_{\gamma'}$
	is that in one of the~two terms of $\nlrel_{\gamma}$ a~product of
	variables associated with the~edges terminating at $v$ is replaced by
	a~product of variables associated with the~edges originating at $v$.
	The~equality of both products is imposed by Soergel relations.
\end{proof}

\begin{proof}[Proof of Proposition~\ref{prop:gilmore-functor}]
    We have to check that linear maps induced by foams preserve the~ideal
    of non-local relations. In all diagrams, the~region $R_\gamma$ enclosed
    by a~simple closed curve $\gamma$ is located below $\gamma$. 

    There are six maps ($\bigon$, $\nogib$, $\zip$, $\piz$, $\assoc$ and
    $\coassoc$) to be inspected, but in the~view of Lemma~\ref{lem:nlrel-vertex-move}
    only $\zip$ required a~non-trivial check. Indeed, let us demonstrate
    how the~lemma is used in case of the~map $\nogib$, which eliminates a~bigon.

    Denote by $\web$ and $\web'$ marked annular webs with a~membrane that
    are identical except in a~small disk $D$ disjoint from the membrane and
    the~marking $\basepoint$, in which
    \[
      \web =\NB{\tikz[scale=2, font = \tiny]{}}
     \qquad\text{and}\qquad
      \web' =\NB{\tikz[scale=2, font = \tiny]{}}
     \quad.
    \]
    If a~curve $\gamma$ does not pass through the~bigon in $\web$,
    then the~relation $\nlrel_\gamma$ is clearly preserved. Otherwise,
    we apply Lemma~\ref{lem:nlrel-vertex-move} to isotope $\gamma$
    away from the~bigon:
    \[
      \NB{\tikz[scale=2, font = \tiny]{\begin{scope}
	\draw[>-] (0,0) -- +(0.2,0) node[pos=0, left] {$a+b$};
	\draw[->] (0.8,0) -- +(0.2, 0) node[pos=1, right] {$a+b$};
	\draw[->-] (0.2, 0) .. controls +(0.3,0.3) and +(-0.3, 0.3)
	   .. (0.8,0) node[pos =0.5, below] {$a$};
	\draw[->-] (0.2, 0) .. controls +(0.3,-0.3) and +(-0.3, -0.3)
	   .. (0.8,0) node[pos =0.5, below] {$b$};
	\draw[densely dashed,\colorgamma] (0, -0.2) .. controls +(0.6, 0) and +(-0.4, 0)
	   .. (1, 0.3) node [pos =1, right, \colorgamma] {$\gamma'$};
\end{scope}
}}
      \quad \rightsquigarrow \quad
      \NB{\tikz[scale=2, font = \tiny]{\begin{scope}
	\draw[>-] (-0.1,0) -- +(0.3,0) node[pos=0, left] {$a+b$};
	\draw[->] (0.8,0) -- +(0.2, 0) node[pos=1, right] {$a+b$};
	\draw[->-] (0.2, 0) .. controls +(0.3,0.3) and +(-0.3, 0.3)
	        .. (0.8,0) node[pos =0.5, below] {$a$};
	\draw[->-] (0.2, 0) .. controls +(0.3,-0.3) and +(-0.3, -0.3)
	        .. (0.8,0) node[pos =0.5, below] {$b$};
	\draw[densely dashed,\colorgamma] (-0.1, -0.2)
		.. controls +(0.2, 0.1) and +(-0.3, -0.1) .. (0.4,0.29)
	    .. controls ++(0.09, 0.03) and ++(-0.1,0) .. (1, 0.32)
	    node [pos =1, right, \colorgamma] {$\gamma'$};
    \end{scope}
}}
    \]
    Analogue arguments ensure that $\assoc$, $\coassoc$, $\bigon$ and $\piz$
    induce morphisms on quotient spaces.

    Let us now deal with $\zip$. Denote by $\web$ and $\web'$
    pointed annular webs with a~membrane that are identical except in a~small
    disk $D$ disjoint from the membrane and the marking $\basepoint$, in which
    \[
      \web =\NB{\tikz[scale=2, font = \tiny]{\begin{scope}[font = \tiny]
  \coordinate (lb) at (-.5,-0.25);
  \coordinate (lt) at (-.5,0.25);
  \coordinate (rb) at ( .5,-0.25);
  \coordinate (rt) at ( .5,0.25);
  \draw[->] (lt) -- (rt) node[pos =0.0, left] {$a$};
  \draw[->] (lb) -- (rb) node[pos =0.0, left] {$b$};
\end{scope}}} \qquad\text{and} \qquad
      \web' =\NB{\tikz[scale=2, font = \tiny]{\begin{scope}[font = \tiny]
  \coordinate (lb) at (-.6,-0.25);
  \coordinate (lt) at (-.6,0.25);
  \coordinate (lo) at (-.25, 0);
  \coordinate (ro) at ( .25, 0);
  \coordinate (rb) at ( .6,-0.25);
  \coordinate (rt) at ( .6,0.25);
  \draw[>-] (lb) .. controls +(0.2,0) and (lo) .. (lo) node[pos=0, left] {$b$};
  \draw[>-] (lt) .. controls +(0.2,0) and (lo) .. (lo) node[pos=0, left] {$a$};
  \draw[->] (ro) .. controls (ro) and +(-0.2, 0) .. (rt) node[pos=1, right] {$a$};
  \draw[->] (ro) .. controls (ro) and +(-0.2, 0) .. (rb) node[pos=1, right] {$b$};
  \draw[->-] (lo) -- (ro) node[pos=0.5, inner sep=1pt, above] {$a+b$};

\end{scope}}}.
    \]
    The only problematic curves are the ones that, inside $D\ric$, go between
    the~two edges of $\web$:
    \[
       \NB{\tikz[scale=2, font = \tiny]{\begin{scope}[font = \tiny]
  \coordinate (lb) at (-.5,-0.25);
  \coordinate (lt) at (-.5,0.25);
  \coordinate (rb) at ( .5,-0.25);
  \coordinate (rt) at ( .5,0.25);
  \draw[->] (lt) -- (rt) node[pos =0.0, left] {$a$};
  \draw[->] (lb) -- (rb) node[pos =0.0, left] {$b$};
  \draw[densely dashed, \colorgamma] (-0.50, 0) -- (0.5, 0) node [pos =1, right, \colorgamma] {$\gamma$};

\end{scope}}}.
    \]
  Let us denote by $\gamma_1$ and $\gamma_2$ curves adapted to $\web'$
  that are identical to $\gamma$ outside of $D\ric$, whereas inside they look
  like in the~following diagram:
  \[
    \NB{\tikz[scale=2, font = \tiny]{\begin{scope}[font = \tiny]
  \coordinate (lb) at (-.6,-0.25);
  \coordinate (lt) at (-.6,0.25);
  \coordinate (lo) at (-.25, 0);
  \coordinate (ro) at ( .25, 0);
  \coordinate (rb) at ( .6,-0.25);
  \coordinate (rt) at ( .6,0.25);
  \draw[>-] (lb) .. controls +(0.2,0) and (lo) .. (lo) node[pos=0, left] {$b$};
  \draw[>-] (lt) .. controls +(0.2,0) and (lo) .. (lo) node[pos=0, left] {$a$};
  \draw[->] (ro) .. controls (ro) and +(-0.2, 0) .. (rt) node[pos=1, right] {$a$};
  \draw[->] (ro) .. controls (ro) and +(-0.2, 0) .. (rb) node[pos=1, right] {$b$};
  \draw[->-] (lo) -- (ro) node[pos=0.5, inner sep=1pt, above] {$a+b$};
  \draw[densely dashed,\colorgamma] (-0.6, 0) -- (-0.52,0)
  	.. controls +(0.2, 0) and +(-0.2, 0) .. (0, 0.35) node [pos=1, above, \colorgamma] {$\gamma_1$}
  	.. controls +(0.2, 0) and +(-0.2, 0) .. (0.52, 0) -- (0.6,0);
  \draw[densely dashed, \colorgamma] (-0.6, 0) -- (-0.52,0)
  	.. controls +(0.2, 0) and +(-0.2, 0) .. (0, -0.3) node [pos=1, below, \colorgamma] {$\gamma_2$}
  	.. controls +(0.2, 0) and +(-0.2, 0) .. (0.52, 0) -- (0.6,0);
  
\end{scope}}}.
  \]
  In order to prove that the~zip map is well-defined, we shall show that
  $\nlrel_\gamma$ is mapped onto an element of the~form
  \begin{align}\label{eq:Pgamma}
	\nlrel_{\gamma_1} \sum_{\alpha \in T(a-1,b)}
		(-1)^{|\widehat{\alpha}|}
		s_{\alpha}(\myunderliney')
		s_{\widehat{\alpha}}(\myunderlinez)
	\quad+\quad
	\nlrel_{\gamma_2} \sum_{\alpha \in T(a,b-1)}
		(-1)^{b+|\widehat{\alpha}|}
		s_{\alpha}(\myunderliney')
		s_{\widehat{\alpha}}(\myunderlinez),
  \end{align}  
  where the~set of variables $Y\ric$, $Z\ric$, $Y'\ric$,
  and $Z'$ are associated with edges of the~web as indicated in
  the~figure below:
  \[
    \NB{\tikz[font = \tiny, scale =2]{\begin{scope}[font = \tiny]
  \coordinate (lb) at (-.6,-0.25);
  \coordinate (lt) at (-.6,0.25);
  \coordinate (lo) at (-.25, 0);
  \coordinate (ro) at ( .25, 0);
  \coordinate (rb) at ( .6,-0.25);
  \coordinate (rt) at ( .6,0.25);
  \draw[>-] (lb) .. controls +(0.2,0) and (lo) .. (lo)
	node[pos=0, left] {$b$} coordinate[pos=0.5] (pb);
  \draw[>-] (lt) .. controls +(0.2,0) and (lo) .. (lo)
  	node[pos=0, left] {$a$} coordinate[pos=0.5] (pa);
  \draw[->] (ro) .. controls (ro) and +(-0.2, 0) .. (rt)
  	node[pos=1, right] {$a$} coordinate[pos=0.5] (pap);
  \draw[->] (ro) .. controls (ro) and +(-0.2, 0) .. (rb)
  	node[pos=1, right] {$b$} coordinate[pos=0.5] (pbp);
  \draw[->-] (lo) -- (ro) node[pos=0.5, inner sep=1pt, above] {$a+b$};
 
	\node[draw,rectangle, gray] (PAP) at ($(pap) + (0,0.4)$) {$\myunderliney'$};
	\draw[gray] (pap) -- (PAP);
	\node[draw,rectangle, gray] (PBP) at ($(pbp) + (0,-0.4)$) {$\myunderlinez'$};
	\draw[gray] (pbp) -- (PBP);
	\node[draw,rectangle, gray] (PA) at ($(pa) + (0,0.4)$) {$\myunderliney$};
	\draw[gray] (pa) -- (PA);
	\node[draw,rectangle, gray] (PB) at ($(pb) + (0,-0.4)$) {$\myunderlinez$};
	\draw[gray] (pb) -- (PB);

\end{scope}
}}.
  \]
  This implies that $\zip(\nlrel_\gamma)$ belongs to $\nonlocalrelations_{\web'}$,
  hence, non-local relations are preserved by the~map.
  Using the~equality  
  \begin{equation}
	\sum_{\alpha \in T(a,b)}
		(-1)^{|\widehat{\alpha}|}
		s_\alpha(\myunderliney')
		s_{\widehat{\alpha}}(\myunderlinez) 
    = \sum_{\alpha \in T(a,b)}
		(-1)^{|\widehat{\alpha}|}
		s_\alpha(\myunderliney)
		s_{\widehat{\alpha}}(\myunderlinez')
  \end{equation}
  that holds in $\set{Soergel}(\web')$, we can rewrite the~image of
  $\nlrel_\gamma = x_{\outp(\gamma)} - q^{2i} x_{\inp(\gamma)}$ as
  \begin{equation}\label{eq:zip-of-NL}
	x_{\outp(\gamma)} \sum_{\alpha \in T(a,b)}
		(-1)^{|\widehat{\alpha}|}
		s_\alpha(\myunderliney')
		s_{\widehat{\alpha}}(\myunderlinez) 
	\quad-\quad
	q^{2i} x_{\inp(\gamma)} \sum_{\alpha \in T(a,b)}
		(-1)^{|\widehat{\alpha}|}
		s_\alpha(\myunderliney)
		s_{\widehat{\alpha}}(\myunderlinez').
  \end{equation}
  We shall analyze each term separately. Notice first that
  \begin{align}
  \label{eq:ine}
    x_{\mathrm{in}(\gamma_1)} &= x_{\mathrm{in}(\gamma)}
                                e_a(\myunderliney), &
    x_{\mathrm{in}(\gamma_2)} &= x_{\mathrm{in}(\gamma)}
                                e_b(\myunderlinez'),\\
  \label{eq:oute}
    x_{\mathrm{out}(\gamma_1)} &= x_{\mathrm{out}(\gamma)}
                                e_a(\myunderliney'), &
    x_{\mathrm{out}(\gamma_2)} &= x_{\mathrm{out}(\gamma)}
                                e_b(\myunderlinez).
  \end{align}
  Denote by $T_1(a,b)$ the subset of Young diagrams with exactly $a$ boxes
  in the~first column and set $T_2(a,b) = T(a,b)\setminus T_1(a,b)$.
  Note that $\widehat{\beta}$ has exactly $b$ boxes the~first column
  when $\beta \in T_2(a,b)$. Hence, for such $\alpha$ and $\beta$ one has
  \begin{align*}
	s_\alpha(\myunderliney) &=
		e_a(\myunderliney) s_{\alpha'}(\myunderliney), &
	s_{\widehat{\beta}}(\myunderlinez) &=
		e_b(\myunderlinez) s_{\widehat{\beta}'}(\myunderlinez),
	\\
	s_\alpha(\myunderliney') &=
		e_a(\myunderliney') s_{\alpha'}(\myunderliney') &
	s_{\widehat{\beta}}(\myunderlinez') &=
		e_b(\myunderlinez') s_{\widehat{\beta}'}(\myunderlinez'),
  \end{align*}
  where $\alpha'$ (resp.\ $\widehat{\beta}'$) is the~Young diagram
  $\alpha$ (resp.\ $\widehat{\beta}$) with its first column removed.
  On the one hand, using (\ref{eq:oute}) one obtains:
  \begin{multline}\label{eq:xout}
	x_{\mathrm{out}(\gamma)} \sum_{\alpha \in T(a,b)}
  		(-1)^{|\widehat{\alpha}|}
  		s_\alpha(\myunderliney')
  		s_{\widehat{\alpha}}(\myunderlinez)
  	\\
  	=\quad x_{\mathrm{out}(\gamma_1)} \sum_{\alpha\in T_1(a,b)}
         (-1)^{\left|\widehat{\alpha}\right|}
         s_{\alpha'}(\myunderliney')
         s_{\widehat{\alpha}}(\myunderlinez)
	\quad+\quad x_{\mathrm{out}(\gamma_2)} \sum_{\alpha\in T_2(a,b)}
         (-1)^{\left|\widehat{\alpha}\right|}
		s_\alpha(\myunderliney')
		s_{\widehat{\alpha}'}(\myunderlinez)
	\\
	=\quad x_{\mathrm{out}(\gamma_1)} \sum_{\alpha\in T(a-1,b)}
         (-1)^{\left|\widehat{\alpha}\right|}
         s_{\alpha}(\myunderliney')s_{\widehat{\alpha}}(\myunderlinez)
    \quad+\quad x_{\mathrm{out}(\gamma_2)} \sum_{\alpha\in T(a,b-1)}
         (-1)^{b+\left|\widehat{\alpha}\right|}
		s_\alpha(\myunderliney')
		s_{\widehat{\alpha}}(\myunderlinez).
  \end{multline}
  On the other hand, using (\ref{eq:ine}) and Corollary~\ref{cor:sym-pol-add-variables} one computes
  \begin{multline}\label{eq:xin}
	x_{\mathrm{in}(\gamma)} \sum_{\alpha \in T(a,b)}
		(-1)^{|\widehat{\alpha}|}
		s_\alpha(\myunderliney)
		s_{\widehat{\alpha}}(\myunderlinez')
	\\
	=\quad x_{\mathrm{in}(\gamma_1)} \sum_{\alpha\in T(a-1,b)}
         (-1)^{\left|\widehat{\alpha}\right|}
         s_{\alpha}(\myunderliney)
         s_{\widehat{\alpha}}(\myunderlinez')
	\quad+\quad x_{\mathrm{in}(\gamma_2)} \sum_{\alpha\in T(a,b-1)}
         (-1)^{b+\left|\widehat{\alpha}\right|}
		s_\alpha(\myunderliney)
		s_{\widehat{\alpha}}(\myunderlinez')
	\\
	=\quad x_{\mathrm{in}(\gamma_1)} \sum_{\alpha\in T(a-1,b)}
		(-1)^{|\widehat\alpha|}
		s_{\alpha}(\myunderliney\sqcup\myunderlinez)
		s_{\widehat\alpha}(\myunderlinez'\sqcup\myunderlinez)
	\quad+\quad x_{\mathrm{in}(\gamma_2)} \sum_{\alpha\in T(a,b-1)}
		(-1)^{b+|\widehat\alpha|}
		s_\alpha(\myunderliney\sqcup\myunderlinez)
		s_{\widehat\alpha}(\myunderlinez'\sqcup\myunderlinez)
  	\\
	=\quad x_{\mathrm{in}(\gamma_1)} \sum_{\alpha\in T(a-1,b)}
		(-1)^{|\widehat\alpha|}
		s_{\alpha}(\myunderliney'\sqcup\myunderlinez')
		s_{\widehat\alpha}(\myunderlinez'\sqcup\myunderlinez)
	\quad+\quad x_{\mathrm{in}(\gamma_2)} \sum_{\alpha\in T(a,b-1)}
		(-1)^{b+|\widehat\alpha|}
		s_\alpha(\myunderliney'\sqcup \myunderlinez')
		s_{\widehat\alpha}(\myunderlinez' \sqcup \myunderlinez)
  	\\
	=\quad x_{\mathrm{in}(\gamma_1)} \sum_{\alpha\in T(a-1,b)}
		(-1)^{|\widehat\alpha|}
		s_{\alpha}(\myunderliney')
		s_{\widehat\alpha}(\myunderlinez)
	\quad+\quad x_{\mathrm{in}(\gamma_2)} \sum_{\alpha\in T(a,b-1)}
		(-1)^{b+|\widehat\alpha|}
		s_\alpha(\myunderliney')
		s_{\widehat\alpha}(\myunderlinez).
\end{multline}
Putting (\ref{eq:xout}) and (\ref{eq:xin}) together, we get that
formulas \eqref{eq:Pgamma} and \eqref{eq:zip-of-NL} coincide as desired.
\end{proof}

Proposition~\ref{prop:gilmore-functor} allows us to use the~framework
from Section~\ref{sec:cube-of-resolutions} to associate with a~braid
diagram $\bdiagram$ of a~link a~chain complex $\complexgilmore(\bdiagram)$,
by applying the~functor $\gilmoreA$ to the~formal complex $\Bracket{\bdiagram}$.
We refer to it as the~\emph{(normalized) Gilmore complex of $\bdiagram$}.
It follows immediately that the~homotopy type of $\complexgilmore(\bdiagram)$
is invariant under braid moves and conjugation away from the~marking $\basepoint$.
It can be also shown that the~homology is invariant under stabilization
if $1-q^d$ is invertible for all $d > 0$.
The~question whether the~complex is truly a~knot invariant remains open.

 \subsection{The~identification with \texorpdfstring{$\protect\tHFK$}{twisted HFK}}

Hereafter we show that $\complexgilmore(\bdiagram)$ computes
the~twisted Hee\-gaard Floer homology of $\bdiagram$.
The~first step is to compare the~two constructions for diagrams with no real
crosssings; the~general case follows by applying the~exact skein triangle.
We begin by reducing the~set of generators of the~ideal
$\nonlocalrelations_\web$.
Consider a~\emph{coherent cycle} $Z$ in $S$,
i.e.\ a~directed closed path in $\web$ with no self-intersections
that does not pass through the~distinguished vertex $\basepoint$. 
Write $\nlrel_Z$ for the~nonlocal relation associated
with the~curve $\gamma_Z$ obtained from $Z$ by pushing the~cycle
slightly away from the~region $R_Z$ enclosed by $Z$. This is illustrated in Figure~\ref{fig:coherentcycles}.

\begin{figure}[ht]
  \centering
  \NB{\tikz[scale=0.75]{\input{\imagesfolder/hfgl0_example-gamma_Z}}}
  \caption{A coherent cycle $Z$ in a web and the corresponding curve $\gamma_Z$.}
  \label{fig:coherentcycles}
\end{figure}

\begin{lemma}\label{lem:nlrel-reduction}
	The~polynomials $\nlrel_Z$ parametrized by coherent cycles in $S$
	generate the~ideal $\nonlocalrelations_\web$.
\end{lemma}
\begin{proof}
	Pick a~curve $\gamma$ adapted to $\web$.
	When $\gamma$ does not surround any vertex of $S$,
	then $\inp(\gamma) = \outp(\gamma)$ and $\nlrel_\gamma = 0$.
	Suppose the~converse and that there is an~arc $\alpha \subset R_\gamma$ disjoint
	from $S$ that connects two points of $\gamma$.
	Performing a~surgery on $\gamma$ along this arc produces
	two curves $\gamma'$ and $\gamma''$ adapted to $\web$;
	$\nlrel_\gamma$ is clearly a~consequence of $\nlrel_{\gamma'}$
	and $\nlrel_{\gamma''}$. Likewise, we can perform surgery
	along such an~arc if it is inside the~interior of an~edge $u$.
	In this case
	\[
		\inp(\gamma) = \inp(\gamma')\inp(\gamma'')x_u
	\quad\text{and}\quad
		\outp(\gamma) = \outp(\gamma')\outp(\gamma'')x_u,
	\]
	so that $NL_\gamma$ is again a~consequence of $NL_{\gamma'}$ and $NL_{\gamma''}$.
	
	The~above together with Lemma~\ref{lem:nlrel-vertex-move}
	allows us to restrict generators of $\nonlocalrelations_\web$
	to relations $\nlrel_\gamma$ parametrized by curves $\gamma$ surrounding
	at least one vertex of $S$ and for which none of the~above surgeries
	can be performed.
	Such a~curve $\gamma$ must intersect every edge of $S$ at most once and
	if it surrounds a~vertex $p$, then it must also surround
	at least one edge going out of $p$ and one edge coming to $p$.
	Such curves are exactly those of the~form $\gamma_Z$,
	where $Z$ consists of the~edges of $S$ that are contained in $R_\gamma$
	and that can be connected to $\gamma$ by an~arc with its interior disjoint from $S$.
\end{proof}

Define $\gilmoreAext(\web)$ as $\gilmoreA(\web)$ but ignoring the~marking
$\basepoint$: the~usual trace relation is applied at this vertex and non-local
relations are also imposed for curves than engulfs it.
Alternatively, we can think of $\basepoint$ as placed at inifinity,
outside of $\web$, or that $\web$ is put inside a~marked loop.
The~following is a~direct consequence of Lemma~\ref{lem:nlrel-reduction}.

\begin{lemma}\label{lem:gilmore-for-disconnected}
	Let $S$ be a~planar singular link in a~braid position
	with components $S_0, \dots, S_r$,
	where $S_i$ is nested inside $S_{i-1}$ for $i=1,\dots,r$ and $S_0$
	carries the~marking $\basepoint$. Then
	\[
		\gilmoreA(S) \cong \gilmoreA(S_0)
			\otimes \gilmoreAext(S_1)
			\otimes \dots
			\otimes \gilmoreAext(S_r)
	\]
\end{lemma}

Notice that $\gilmoreAext(\web)$ is annihilated by $1-q^{2w}$,
where $w$ is the~index of $\web$.
Hence, $\gilmoreA(\web)$ is torsion when $\web$ is disconnected,
as already shown in Example~\ref{exa:GilmoreA-disconnected-web}.

A~similar result holds for the~twisted Heegaard Floer complex.
The~following is inspired by \cite[Lemma 2.2]{ManolescuCube}.

\begin{lemma}\label{lem:hfk-for-disconnected}
	Let $S$ be a~singular link in a~braid position with split components
	$S_0, \dots, S_r$, where $\basepoint \in S_0$.
	Then there is an~isomorphism of complexes
	\begin{equation}\label{eq:CFK-for-disconnected}	
		\tCFK-(S) \cong \tCFK-(S_0) \otimes \tCFK-(U \sqcup S_1)
			\otimes \dots \otimes
		\tCFK-(U \sqcup S_r),
	\end{equation}
	where in each $U\sqcup S_i$ the~marking lies on the~unknotted component.
\end{lemma}
\begin{proof}
	The~planar Heegaard diagram for $S$ can be seen as a~sequence of nested
	annular diagrams $\mathcal H_0, \dots, \mathcal H_r$ with $X_0$ outside of them
	as shown at the~left side of Figure~\ref{fig:unnesting-hd}. Here $\mathcal H_0$ together with $X_0$ represents the~component $S_0$,
	whereas each $\mathcal H_i$ for $i>0$, with an~extra pair of basepoints
	outside, represents $U \sqcup S_i$ with $\basepoint \in U$.
	Thence the~isomorphism \eqref{eq:CFK-for-disconnected} is clear at the~level
	of chain groups.
	In order to compare the~differentials we first modify the~Heegaard diagram,
	so that the~components $\mathcal H_i$ are no longer nested,
	but instead they are contained in disjoint disks on the~sphere
	and the~complement of those disks is decorated with $X_0$,
	see the~right side of Figure~\ref{fig:unnesting-hd}.
	We argue then that the~differential of the~chain complex associated
	with this modified diagram matches with that of the~right hand side
	of \eqref{eq:CFK-for-disconnected}.

	For the first step, we use the~following generalization of a~handle slide.
	Suppose that $\mathcal H = \mathcal H' \sqcup \mathcal H''$,
	where $\mathcal H''$ is contained in a~disk $D$ disjoint from
	$\mathcal H'$ and all basepoints and markings of $\mathcal H''$
	are surrounded by $\alpha$-curves of this diagram.
	Then any $\alpha$-curve of $\mathcal H'$ can be slid over $\mathcal H''$.
	This claim is easy to check by using handle slides and isotopies.
	
	Observe now that each $\mathcal H_i$ for $i > 0$ has all its basepoints
	and twisting markings inside $\alpha$-curves
	(compare Figure \ref{fig:planar-Heegaard-diagram}) and the~union
	$\mathcal H_+ = \mathcal H_1 \amalg\dots\amalg \mathcal H_r$ is contained in a~disk.
	Furthermore, $\mathcal H_0$ intersects the~trace section only in $\alpha$-curves.
	Hence, we can apply the~generalized handle slide to these $\alpha$-curves,
	pushing $\mathcal H_+$ outside of $\mathcal H_0$ as the~result.
	Applying the~same procedure inductively to $\mathcal H_+$
	we complete the first step.
	
	For the~second step, we observe first that the~resulting Heegaard diagram $\mathcal H'$
	represents the~connected sum $(\mathbb S^2, \mathcal H_0) \# \dots \#
	(\mathbb S^2, \mathcal H_r)$ of $r{+}1$ copies of a~sphere decorated
	with Heegard diagrams $\mathcal H_i$.\footnote{
		Strictly speaking we equip $\mathcal H_i$ for $i>0$ with an~extra pair
		of basepoints.
	}
	It follows now from \cite[Theorem 5.1]{Multivariable} that any holomorphic
	disk $\phi$ in the~symmetric power of $(\mathbb S^2, \mathcal H')$
	and with multiplicity $0$ at $X_0$ corresponds to a~collection of holomorphic disks
	$\phi_i$, each supported in the~symmetric power of $(\mathbb S^2, \mathcal H_i)$,
	such that $\mathcal D(\phi) = \sum_i \mathcal D(\phi_i)$
	and $\mu(\phi) = \sum_i \mu(\phi_i)$.
	Hence, $\mu(\phi) = 1$ forces all $\phi_i$ except one to be constant,
	so that $\supp\phi$ is contained in exactly one of the~diagrams $\mathcal H_i$.
	This shows that the~differentials of the~two sides of \eqref{eq:CFK-for-disconnected}
	coincide modulo 2. Finally, in order to match the~signs in the~differentials,
	we can equip moduli spaces of holomorphic curves in the~symmetric power of
	$(\mathbb S^2,\mathcal H)$ with orientations induced from the~moduli spaces
	of holomorphic curves in the~symmetric powers of $(\mathbb S^2, \mathcal H_i)$.
\end{proof}

\begin{figure}[ht]
   \begin{tikzpicture}[scale=0.9]
      \begin{scope}[xshift=-110pt]
        \foreach \r/\lbl in {18pt/1,45pt/0} {
          \draw[style=hd/alpha] (0,\r)++(0,-6pt) -- ++(-30pt,0);
          \draw[style=hd/alpha] (0,\r)++(0,-3pt) -- ++(-30pt,0);
          \draw[style=hd/alpha] (0,\r)++(0,+0pt) -- ++(-30pt,0);
          \draw[style=hd/alpha] (0,\r)++(0,+3pt) -- ++(-30pt,0);

          \draw[gray!50,line width=15pt]
          (0,\r) arc[radius=\r, start angle=90, end angle=-90] -- ++(-30pt,0)
          arc[radius=\r, start angle=270, end angle=90] -- ++(5pt, 0); \node[font=\small] at (-35pt,\r) {$\mathcal H_\lbl$};
        }
        \draw[gray!50,line width=15pt]
        (0,18pt) -- ++(-26pt,0);
        
        \draw[densely dashed, \colormembrane, line width=0.3pt] (-12.5pt, 5pt) -- (-12.5pt, 80pt);
        \draw[line width=1pt] (-20pt,75pt)
        ++(0.1,0.125) -- ++(-0.2,-0.25) ++(0.2,0) -- ++(-0.2,0.25);
        \node at (-12.5pt,0) {$\dots$};
      \end{scope}

      \draw[->] (-52pt,0) -- ++(50pt,0)
      node[midway,above,font=\tiny] {isotopies and}
      node[midway,below,font=\tiny] {handle slides};
      \draw[->] (145pt,0) -- ++(50pt,0)
      node[midway,above,font=\tiny] {isotopies and}
      node[midway,below,font=\tiny] {handle slides};
      
      \begin{scope}[xshift=86pt]
        \foreach \r/\lbl in {18pt/1,45pt/0} {
\draw[style=hd/alpha] (0,\r)++(0,+0pt) -- ++(-30pt,0);
          \draw[style=hd/alpha] (0,\r)++(0,+3pt) -- ++(-30pt,0);
                                                                                
          \draw[gray!50,line width=15pt]
          (0,\r) arc[radius=\r, start angle=90, end angle=-90] -- ++(-30pt,0)
          arc[radius=\r, start angle=270, end angle=90] -- ++(5pt, 0); \node[font=\small] at (-35pt,\r) {$\mathcal H_\lbl$};
        }
        \draw[gray!50,line width=15pt]
        (0,18pt) -- ++(-26pt,0);
        \draw[style=hd/alpha] (0,45pt)++(0,-6pt) arc (90:270:3pt) arc (90:-90: 33pt) -- ++(-30pt,0) arc (270:90: 33pt) -- ++(5pt,0) arc (-90:90:3pt);
        \draw[style=hd/alpha] (0,45pt)++(0,-3pt) arc (90:270:6pt) arc (90:-90: 30pt) -- ++(-30pt,0) arc (270:90: 30pt) -- ++(5pt,0) arc (-90:90:6pt);

        \draw[densely dashed, \colormembrane, line width=0.3pt] (-12.5pt, 5pt) -- (-12.5pt, 80pt);
        \draw[line width=1pt] (-20pt,75pt)
        ++(0.1,0.125) -- ++(-0.2,-0.25) ++(0.2,0) -- ++(-0.2,0.25);
        \node at (-12.5pt,0) {$\dots$};
      \end{scope}

      \begin{scope}[xshift=230pt]
        \foreach \x/\lbl in {0/0, 65pt/1}{
          \draw[gray!50, line width=16pt]
          (\x,20pt) -- ++(5pt,0)
          arc[radius=20pt, start angle=90, end angle=-90] -- ++(-5pt,0)
          arc[radius=20pt, start angle=270, end angle=90] -- cycle;
          \draw[densely dashed, \colormembrane, line width=0.3pt] (\x, 5pt) ++(7pt,0) -- ++(0, 30pt);
          \draw (\x,20pt) ++(-2pt,0) node[font=\small] {$\mathcal H_\lbl$};
        }
        \node at (67.5pt,0) {$\dots$};
        \draw[line width=1pt] (32.5pt,50pt)
        ++(0.1,0.125) -- ++(-0.2,-0.25) ++(0.2,0) -- ++(-0.2,0.25);
      \end{scope}
   \end{tikzpicture}
   \caption{Sketches of the~planar Heegaard diagram for a~split link
      and the~diagram desired for the~proof of Lemma~\ref{lem:hfk-for-disconnected}.
      Each thick circle represent a~collection of $\alpha$- and $\beta$-curves
      that intersect each other and represent one of the~components of the~link.
   }
   \label{fig:unnesting-hd}
\end{figure}
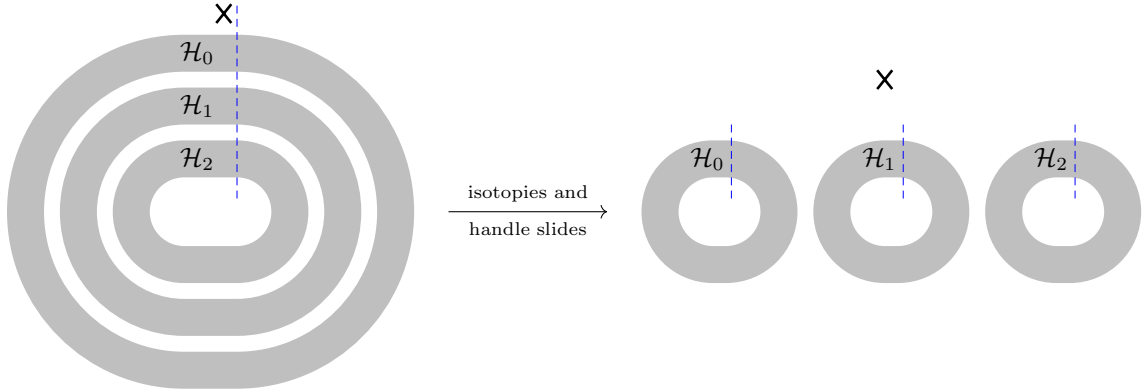

Consider now the~complex $\tCFK-(S)$ computed from the~planar
Heegaard diagram for $S\ric$, where the~link diagram is considered
to consists only of singular crossings and bivalent vertices.
This simplifies the~diagram a~lot: each basepoint, except
the~one related to the~marking $\basepoint$, lies in a~bigon.
The~generators of the~complex are associated in \cite[Section 3]{OSCube}
with \emph{coherent multicycles} in $S\ric$,
i.e.\ (possibly empty) collections of disjoint coherent cycles;
the~multicycle $Z$ associated with a~generator $\xgen$
consists of the~edges of $S$ that corresponds to those bigons
with $\markO$-basepoints that have a~component of $\xgen$
at one of the~corners. Other components of $\xgen$ are corners
of bigons containing $\markX$-basepoints.

\begin{theorem}[cp.\ {\cite[Theorem 3.1]{OSCube}}]
\label{thm:A-vs-HFK-for-planar}
	Let $S$ be a~diagram of a~planar singular link in a~braid position.
	Then	 there is an~isomorphism
	\[
		H_0(\tCFK(S) \otimes \mathcal L_S) \cong \gilmoreA(S)
	\]
	that sends the~homology class of\/ $\xgen_0$ to the~unit of $\gilmoreA(S)$
	and for every thin edge $e$ identifies the~action of\/ $U_e$ with
	the~multiplication by $x_e$.
\end{theorem}
\begin{proof}
Assume first that $S$  has only one component.
	Following \cite{OSCube} we show first that $\xgen_0$ is a~cycle
	that generates the~homology.
	For that define the~\emph{deviation} of a~generator $\xgen \in \tCFK-(S)$
	as the~number of components of $\xgen$ that are right corners of bigons.
	As in \cite[Section 3]{OSCube} we can show that
	$\hdeg(\xgen) = -|Z| - d$, 
where $|Z|$ is the~number of components
	of the~multicycle $Z$ associated with $\xgen$ and $d$ is the~deviation
	of the~generator.
	Hence, each multicycle $Z$ admits a~unique generator $\xgen_Z$
	of maximal homological degree $-|Z|$.
	In particular, this implies that $\xgen_0 = \xgen_\emptyset$
	is a~cycle that generates the~0th homology.
The~relations come from counting pseudo-holomorphic disks from
	generators of degree $-1$ to $\xgen_0$. There are two types of such generators:
	the~deviation 1 generators associated to $Z=\emptyset$ and the~generators
	$\xgen_Z$ associated with (connected) coherent cycles $Z$.
	
	The~first kind of generators impose local relations on $\xgen_0$.
	Let $\xgen$ be deviated from $\xgen_0$ at one bigon. Because $Z=\emptyset$,
	the~bigon contains a~basepoint $X_p$ that corresponds to a~vertex $p\in S$.
	Write $A_p$ and $B_p$ for the~regions containing $X_p$ and bounded respectively
	by $\alpha$- and $\beta$-curves.
	The~differential $\partial\xgen$ has only two terms that correspond
	to  $A_p \setminus B_p$ and $B_p \setminus A_p$, and which appear
	with opposite signs due to Lemma~\ref{lem:signs-for-AB-pair}. Thence,
	\[
		\partial\ygen = \pm\left( U^{(p)}_a U^{(p)}_b - U^{(p)}_c U^{(d)}_d \right)\xgen_0
	\]
	when $p$ is a~singular crossing and
	\[
		\partial\ygen = \pm\left( \tau U^{(p)}_a - U^{(p)}_c \right)\xgen_0
	\]
	in case of a~bivalent vertex, where $\tau=t$ appears only in case
	of trace vertices and $\tau=1$ otherwise.
	Hence, all local relations hold in the~homology.

	The~second kind of generators impose non-local relations on $\xgen_0$.
	For any connected coherent cycle $Z$ there are exactly two positive
	Maslov index one Whitney disks from $\xgen_Z$ to $\xgen_0$
	that avoid $\markX$:
	\begin{equation}
		\phi_1 = (R_Z \cup \bigcup_{X_i\in Z} B_i) \setminus \bigcup_{i > 0} A_i
	\qquad\text{and}\qquad
		\phi_2 = (R_Z \cup \bigcup_{X_i\in Z} A_i) \setminus \bigcup_{i > 0} B_i,
	\end{equation}	
	see \cite[Figure 10]{OSCube}.
	Note that the~connectivity of $S$ is important for this to hold.
	By Lemma~\ref{lem:index-for-disk-missing-disks}, $\moduliR(\phi_i) = \pm1$,
	and $\moduliR(\phi_1) + \moduliR(\phi_2) = 0$ by Corollary~\ref{cor:sum-of-signs-for-x0}.
	In our framework, the~disks $\phi_1$ and $\phi_2$ contribute
	(up to sign) towards $\partial\xgen_Z$
	respectively $t^{w(Z)} U_{\outp(Z)}$ and $U_{\inp(Z)}$,
	where $U_{\outp(Z)}$ (resp.\ $U_{\inp(Z)}$) is the~product of variables associated
	with edges going out of (resp.\ coming into) the~(closed) region $R_Z$ bounded by $Z$
	and $w(Z)$ is the~number of trace vertices in $R_Z$.
	Hence,
	\[
		\partial\xgen_Z = \pm\left( t^{w(Z)} U_{\outp(Z)} - U_{\inp(Z)} \right)\xgen_0,
	\]
	where the~difference in parentheses matches $\nlrel_{\gamma_Z}$ for $t=q^{-2}$.
	Together with Lemma~\ref{lem:nlrel-reduction} this implies
	that non-local relations hold in the~homology.
	
	Suppose now that $S = U \sqcup S'\ric$,
	where $S'$ is connected and surrounded by the~unknot $U$
	that carries the~basepoint $\basepoint$.
	The~same argument as above shows that the~0th homology
	coincides with $\gilmoreA(S) = \gilmoreAext(S')$.
	Together with Lemmata~\ref{lem:gilmore-for-disconnected}
	and \ref{lem:hfk-for-disconnected} this proves the~general case.
\end{proof}

Choose now a~braid diagram $\bdiagram$ of a~link.
By applying Theorem~\ref{thm:CFK-cube} we can replace $\tCFK(\bdiagram)$
with a~homotopy equivalent complex $\widetilde C(\bdiagram)$ modelled
on the~cube of resolutions from Section~\ref{sec:cube-of-resolutions}:
\begin{itemize}
	\item a~vertex corresponding to a~full resolution $S$
	is decorated with $\tCFK(S)\otimes \mathcal L_S$,
	\item edges are decorated with $\zip$ and $\piz$ maps,
	\item higher components of the~differential are possible.
\end{itemize}
The~column filtration on this cube leads to a~spectral sequence
converging to $\tHFK(\bdiagram)$.
When $\scalars$ is $q$-complete, e.g.\ $\scalars=\ZZ[q^{-1},q]]$,
then $H(\tCFK(S)\otimes\mathcal L_S)$ is concentrated in homological
degreee 0 by Proposition~\ref{prop:hfk-for-connected-planar}
and the~first page of the~spectral sequence coincides with $\complexgilmore(\bdiagram)$.
In particular, the~spectral sequence collapses on the~second page.

\begin{theorem}\label{thm:gilmore-vs-hfk}
	Suppose that $\bdiagram$ is a~braid diagram of a~knot
	and let $\scalars=\ZZ[q^{-1},q]]$ with $t=q^{-2}$.
	Then there is a~quasi-isomorphism
	\[
		\tCFK(\bdiagram) \xrightarrow{\ \sim\ } \complexgilmore(\bdiagram).
	\]
	In particular,
	$H(\complexgilmore(\bdiagram)) \cong \HFK(\bdiagram)\otimes\scalars$
	is a~knot invariant.
\end{theorem}
\begin{proof}
	The~desired quasi-isomorphism is a~composition of a~sequence
	of homotopy equivalences from Theorem~\ref{thm:CFK-cube},
	which replaces $\tCFK(\bdiagram)$ with a~cube of complexes
	computed from full resolutions of $\bdiagram$,	
	followed by an~epimorphism onto $\complexgilmore(\bdiagram)$
	given by
	\[
		\tCFK(\bdiagram_\resolution) \otimes \mathcal L_{\bdiagram_\resolution}
			\longrightarrow
		H_0(\tCFK(\bdiagram_\resolution) \otimes \mathcal L_{\bdiagram_\resolution})
			\cong \gilmoreA(\bdiagram_\resolution)
	\]
	at each resolution $\bdiagram_\resolution$. The~projection
	on $0$th homology is well-defined, because the~canonical generator
	$\xgen_0 \in \tCFK(\bdiagram_\resolution)$ has the~maximal homological degree.
	It is a~quasi-isomorphism, because higher homology groups vanish,
	see Proposition~\ref{prop:hfk-for-connected-planar}.
\end{proof}

 \subsection{An~identification with the~original Gilmore complex}

Choose a~braid diagram $\bdiagram$ of a~knot and its complete resolution $\bdiagram_\resolution$.
The~original construction of the~algebra $\gilmoreA(\bdiagram_\resolution)$
as described in \cite{OSCube, Gilmore} computes the~Heegaard Floer homology
twisted by the~set $\OSTwist$.
This algebra, denoted here by $\gilmoreA'(\bdiagram_\resolution)$,
assumes $\bdiagram_\resolution$ is layered in the~sense of Section~\ref{sec:cube-of-resolutions}
with $n+1$ levels, where $n$ is the~number of crossings in $\bdiagram$ (the~extra level is the~trace section)
and it is generated like $\gilmoreA(\bdiagram_\resolution)$ by thin edges
of $\bdiagram_\resolution$. The~local relations, however, are twisted differently:
\begin{center}
	\begin{tabular}{cp{7mm}cp{7mm}cp{7mm}c}
		\NB{\tikz[font=\small]{
			\draw[>->] (0,0) -- (1,0)
			    node[pos=0, left]  {$x_c$}
			    node[pos=1, right] {$x_a$};
			\draw (0.5,-2pt) -- ++(0,4pt);
         }}
         &&
		\NB{\tikz[font=\small]{
			\draw[densely dashed, \colormembrane, line width=0.3pt] (0.5, 0.5) -- (0.5, -0.5);
			\draw[>->] (0,0) -- (1,0)
			    node[pos=0, left]  {$x_c$}
			    node[pos=1, right] {$x_a$};
			\draw (0.5,-2pt) -- ++(0,4pt);
         }}
         &&
		\NB{\tikz[font=\small]{
			\draw[densely dashed, \colormembrane, line width=0.3pt] (0.5, 0.5) -- (0.5, -0.5);
			\draw[>->] (0,0) -- (1,0)
				node[pos=0, left]  {$x_c$}
				node[pos=1, right] {$x_a$}
				node[midway, font =\normalsize, fill=white, inner sep=1pt] {$\basepoint$};
		}}
		&&
         \NB{\tikz[font=\small]{
         	\draw[>-<]
         		(-0.75, 0.5) node[left] {$x_c$}
         			.. controls ++(0.25,0) and ++(0,0) ..
         		(-0.25, 0) .. controls ++(0,0) and ++(0.25,0) ..
         		(-0.75,-0.5) node[left] {$x_d$};
         	\draw[<->]
         		(0.75, 0.5) node[right] {$x_a$}
         			.. controls ++(-0.25,0) and ++(0,0) ..
         		(0.25, 0) .. controls ++(0,0) and ++(-0.25,0) ..
         		(0.75,-0.5) node[right] {$x_b$};
	  		\draw[line width=2pt] (0.25,0) ++(0.5pt,0) -- (-0.25,0) -- ++(-0.5pt,0);
		}}
		\\[5ex]
		$tx_a = x_c$ && $x_a = x_c$ && $x_a = 0$ && $t(x_a+x_b) = x_c+x_d$
		\\[1ex]
		&&          &&             && $t^2 x_a x_b = x_c x_d$
		\\[1.5ex]
	\end{tabular}
\end{center}
and non-local relations, parametrized by coherent cycles $Z$ in $S$,
take the~form
\[
	\nlrel'_Z = t^{|Z|}x_{\outp(Z)} - x_{\inp(Z)},
\]
where $|Z|$ is a~weighted counts of (non-trace) vertices in $R_Z$:
each (non-trace) bivalent vertex contributes $1$, whereas a~singular
crossing contributes $2$.
The~form of local relations suggests already
an~isomorphism between the~two algebras.

\begin{proposition}\label{prop:renormalized-gilmore}
	Let\/ $\scalars = \ZZ[t^{1/2},t^{-1/2}]$ and set $q=t^{-n/2}$,
	where $n$ is the~number of crossings in $\bdiagram$.
	Then there is an~isomorphism of algebras
	\[
		\gilmoreA'(\bdiagram_\resolution) \ni x_\alpha
			\xrightarrow{\ \cong\ }
		t^{-n(\alpha)} x_{\overline\alpha} \in \gilmoreA(\bdiagram_\resolution),
	\]
	where $\overline\alpha$ is the~edge of $\web$ that contains the~image of
	the~semi-arc $\alpha$ in the~resolution and $n(\alpha)$ is the~number
	of crossings in $\braid$ to the~left of $\alpha$.
\end{proposition}
\begin{proof}[Proof of Proposition~\ref{prop:renormalized-gilmore}]
	Renormalize the~basis of $\gilmoreA'(\bdiagram_\resolution)$ by setting
	$\tilde x_\alpha := t^{n(\alpha)}x_\alpha$. Clearly, the~local relations
	at non-trace vertices do not involve $t$ anymore, whereas at a~trace vertex
	the~linear relation $x_a = x_c$ is replaced with $\tilde x_a = t^n \tilde x_c$,
	that coincides with the~quantum trace relation $\tilde x_a = q^{-2} \tilde x_c$.
	In particular, variables at both sides of a~bivalent vertex other than
	the~trace vertex are identified.

	It remains to show that the~non-local relation
	$\nlrel'_Z$ associated with a~coherent cycle $Z$,
	when rewritten in the~new basis,
	takes the~form \eqref{eq:non-local-normalized}
	for $\gamma = \gamma_Z$ a~small push of $Z$.
	In other words, the~power of $t$ must equal $in$,
	where $i$ is the~number of trace vertices in $R_Z$.
	For that resolve $\bdiagram_\resolution$ into a~collection of concentric
	loops $\ell_1, \dots, \ell_k$ by replacing every singular crossing with
	two horizontal lines, each with a~bivalent vertex on it.
	The~exponent of $t$ in $\nlrel'_Z$ counts then bivalent vertices inside $\gamma_Z$.
	
	Consider first a~loop $\ell_r$, the~trace vertex of which is inside $\gamma_Z$.
	If it is entirely contained by $\gamma_Z$, then it contributes exactly $n$ towards
	the~power of $t$. Otherwise, each arc with $s$ bivalent vertices
	outside of $\gamma_Z$
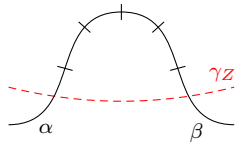
	\[
		\begin{tikzpicture}
			\def\tag(#1)#2{\draw (#1) ++(#2:-3pt) -- ++(#2:6pt)}\draw[\colorgamma, densely dashed] (0,0.5) .. controls ++(1,-0.25) and ++(-1,-0.25) .. (3,0.5);
			\draw (0,0)
			    .. controls ++(1,0) and ++(-1,0) .. (1.5,1.5)
			    .. controls ++(1,0) and ++(-1,0) .. (3,0);
			\tag(0.75,0.75){-15};
			\tag(1.01,1.298){-42};
			\tag(1.5,1.5){90};
			\tag(1.99,1.298){42};
			\tag(2.25,0.75){15};
			\node[above,color=\colorgamma] at (2.85, 0.45) {$\scriptstyle\gamma_Z$};
			\node[below] at (0.5, 0.15) {$\scriptstyle\alpha$};
			\node[below] at (2.5, 0.15) {$\scriptstyle\beta$};
		\end{tikzpicture}
	\]
	lowers the~contributions of the~loop towards $\weight(\gamma_Z)$ by $s$.
	However, the~semi-arcs $\alpha$ and $\beta$ containing the~left and
	right endpoints of the~arc satisfy $n(\beta) = n(\alpha) + s$,
	so that renormalizing the~variables increases the~contribution back.
	Hence, in the~renormalized basis, each such loop contributes exactly
	$n$ towards $\weight(\gamma_Z)$.
	
	Conversely, if the~trace vertex of $\ell_r$ is not contained by $\gamma_Z$,
	then $\ell_r$ does not contribute towards the~power of $t$.
	Indeed, for every arc of $\ell_r$ with $s$ vertices inside $\gamma_Z$
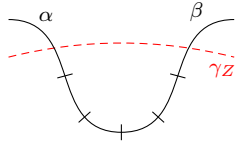
	\[
		\begin{tikzpicture}[yscale=-1]
			\def\tag(#1)#2{\draw (#1) ++(#2:-3pt) -- ++(#2:6pt)}\draw[\colorgamma, densely dashed] (0,0.5) .. controls ++(1,-0.25) and ++(-1,-0.25) .. (3,0.5);
			\draw (0,0)
			    .. controls ++(1,0) and ++(-1,0) .. (1.5,1.5)
			    .. controls ++(1,0) and ++(-1,0) .. (3,0);
			\tag(0.75,0.75){-15};
			\tag(1.01,1.298){-42};
			\tag(1.5,1.5){90};
			\tag(1.99,1.298){42};
			\tag(2.25,0.75){15};
			\node[below,color=\colorgamma] at (2.85, 0.45) {$\scriptstyle\gamma_Z$};
			\node[above] at (0.5, 0.15) {$\scriptstyle\alpha$};
			\node[above] at (2.5, 0.15) {$\scriptstyle\beta$};
		\end{tikzpicture}
	\]
	and the~left and right endpoints on semi-arcs $\alpha$ and $\beta$ respectively,
	we have $n(\beta) - n(\alpha) = s$. Hence, renormalizing variables lowers
	the~power of $t$ by $s$, cancelling the~contribution of the~vertices from the~arc.

	Hence, the power of $t$ in $\nlrel'_Z$, when rewritten in the~new basis,
	is equal to $in$ as desired.
\end{proof}
 \subsection{A pseudo completion}
\label{sec:pseudo-completion}

In this section, we introduce a~functor $\ourQ$
that interpolates $\gll_0$ homology and knot Floer homology.
It comes from the~observation that Theorem~\ref{thm:gilmore-vs-hfk}
relates Gilmore's construction to knot Floer homology when
coefficients are $\ZZ[q^{-1},q]]$.
On the~other hand, the~definition of $\gll_0$ homology can be ``morally''
thought of as the Gilmore one specialized at $q=1$.
The functor $\ourQ$ aims to take the best of these two incompatible worlds.

Coefficients over which chain complexes are considered will play an
important role in this section. We emphasize this importance by
writing them systematically. Moreover, despite the~construction of $\ourQ$
makes sense for any pointed annular web, we focus on the~case of elementary
webs.

Given an~annular web $\web$, consider the map:
\[
	Q_\web \co \gilmoreA(\web; \ZZ[q,q^{-1}]) \lra
		\gilmoreA(\web; \ZZ[q^{-1},q]]) 
\]
given by extending the scalars. It may not be injective. Define
\[
   \ourQ(\web) := \quotient{\gilmoreA(\web; \ZZ[q,q^{-1}])}{\ker Q_\web}
\]
and more generally $\ourQ(\web; \scalars) :=
\ourQ(\web) \otimes_{\ZZ[q,q^{-1}]} \scalars$
for any $\ZZ[q,q^{-1}]$-module $\scalars$. 
Notice that in $\ourQ(\web)$ we kill every decoration $x\in D(\web)$
that is annihilated in $\gilmoreA(\web; \ZZ[q,q^{-1}])$
by some nontrivial polynomial $p(q)\in\ZZ[q,q^{-1}]$.
Because the~homomorphism $Q_\web$ is natural with respect to actions of foams,
$\ourQ(-; \scalars)$ extends to a~functor on $\cat{qAFoamB}$.

\begin{lemma}\label{lem:Q-fin-gen}
	If\/ $\scalars$ is a~PID and $\web$ is an elementary pointed annular web,
	then $\ourQ(\web; \scalars)$ is a~free\/ $\scalars$-module of finite rank.
\end{lemma}
\begin{proof}
	Notice first that $\ourQ(\web; \scalars)$ vanishes when $\web$ is disconnected
	and is free of rank one when $\web$ is a~chain of dumbbells, see
	Example~\ref{exa:GilmoreA-chain-dumbells}. The~thesis
	 follows now from the~functoriality of $\ourQ(-; \scalars)$ and
	Proposition~\ref{prop:elem-annular-reduction}, because a~submodule
	of a~finitely generated free module over a~PID is finitely generated
	and torsion-free, hence free.
\end{proof}

Using the~cube of resolutions approach one extends $\ourQ$
to braid diagrams and we write $\ourQH(\bdiagram; \scalars)$
for the~homology of the corresponding complex.
This complex plays a~central in our subsequent constructions.
We simplify the~notation to $\ourQone$ and $\ourQHone$ respectively when $q=1$.

While it can be shown that $\ourQH(\bdiagram; \scalars)$ is a~braid invariant
that is preserved under stabilization, checking the~first Markov move
(conjugacy) is challenging.

\begin{customconj}{1}
	If\/ $\scalars$ is a~field of characteristic $0$,
	then $\ourQH$ is a~knot invariant for any $\qnt$.
\end{customconj}

As a direct consequence of the~construction, $\ourQ(\web)$
can be identified with a~$\ZZ[q,q^{-1}]$-subspace of
$\gilmoreA(\web; \ZZ[q^{-1},q]])$ of maximal rank.
This observation leads immediately to the~following result.

\begin{proposition}\label{prop:ourQ-to-Gilmore}
	For any~braid closure $\bdiagram$,
$\ourQ(\bdiagram; \ZZ[q^{-1},q]])$ 	and
	$\complexgilmore(\bdiagram; \ZZ[q^{-1},q]])$ are isomorphic
	complexes of graded $\ZZ[q^{-1},q]]$-modules.
	In particular, $\ourQ(\bdiagram; \ZZ[q^{-1},q]])$
	is quasi-isomorphic to $\CFK(\bdiagram)\otimes \ZZ[q^{-1},q]]$
	when $\bdiagram$ is a~knot.
\end{proposition}
\begin{proof}
	The~map $\ourQ(\web,\ZZ[q,q^{-1}] )\to \gilmoreA(\web,\ZZ[q^{-1},q]])$
	induced by the inclusion of the coefficients is injective,
	due to the definition of $\ourQ$,
	and it becomes an~isomorphism after tensoring with $\ZZ[q^{-1},q]]$
	over $\ZZ[q,q^{-1}]$.
	The~last statement follows from Theorem~\ref{thm:gilmore-vs-hfk}.
\end{proof}

Specializing the~complex $\complexgilmore(\bdiagram)$
at $q=1$ does not recover the~$\LieGL_0$ complex,
e.g.\ $\gilmoreA(\web)$ may not vanish for a~disconnected web $\web$.
The~situation is different for $\ourQ$.

\begin{proposition}\label{prop:ourQ-and-gl0}
	For any elementary pointed annular web $\web$ there is
	an~isomorphism $\ourQone(\web; \scalars) \cong \glzero(\web; \scalars)$
	that intertwines the~action of foams.
	In particular,
	$\ourQone(\bdiagram; \scalars)$ and $\complexglzero(\bdiagram; \scalars)$
	are naturally isomorphic
	as complexes of graded\/ $\scalars$-modules.
\end{proposition}
\begin{proof}
	Both $\ourQone(\web; \scalars)$ and $\glzero(\web; \scalars)$ are quotients of
	the~Soergel space $\set{Soergel}(\web)$ of the~web $\web$. We claim that
	the~identity on $\set{Soergel}(\web)$ induces the~desired isomorphism.
	Due to functoriality of both constructions and
	Proposition~\ref{prop:elem-annular-reduction}
	it is enough to check the~claim for basic elementary webs.
	
	This is clear when $\web$ has more than one component,
	because in this case both spaces are zero.
	Otherwise $\web$ is either a~single circle or a~chain of dumbbells and in each case
	both spaces are freely generated by the~empty decoration,
	see Theorem~\ref{thm:RW3}, Examples~\ref{exa:GilmoreA-elementary-web}
	and \ref{exa:GilmoreA-chain-dumbells}.
\end{proof}

Proposition \ref{prop:ourQ-vs-gl0} from the~introduction
is an~immediate corollary of the~above result.

\begin{customprop}{A}
	The~homology theories $\ourQHone$ and $\homologyglzero$ coincide.
	Hence, $\ourQHone$ is a~knot invariant if\/ $\scalars$ is a~field.
\end{customprop}
\begin{proof}
	The~first statement follows from applying Proposition~\ref{prop:ourQ-and-gl0}
	to each vertex in the~cube of resolutions.
	Since the~$\LieGL_0$ homology is a~knot invariant when $\scalars$ is a~field,
	then so is $\ourQHone\!$.
\end{proof}

\begin{remark}
	Another consequence of Propositon~\ref{prop:ourQ-and-gl0}
	is that $\ourQone(\web)$ is a~free $\scalars$-module for
	any elementary pointed web $\web$ and any ring $\scalars$,
	because $\glzero(\web; \scalars)$ is free.
	This strengthens Lemma~\ref{lem:Q-fin-gen} when $q=1$.
\end{remark}

 \subsection{The~spectral sequence}
\label{sec:specseq}

In this short section we establish the~main result of the~paper.
The~idea is to apply to $\ourQH(\bdiagram)$ the~$(q\mapsto 1)$ Bockstein spectral sequence,
discussed in details in Appendix~\ref{sec:bockstein}.
For this purpose we fix an~arbitrary field $\mathbb K$
and work over a~PID $\mathbb K[q,q^{-1}]$, where we can specialize $q=1$.

\begin{customthm}{B}
  Assume that $\mathbb K$ is a~field and $K$ is a~knot represented
  by a~braid closure $\bdiagram$.
  Then the~$(q\mapsto 1)$ Bockstein spectral sequence applied to
  $\ourQ(\bdiagram; \mathbb{K}[q, q^{-1}])$ has
  $\homologyglzero(K;\mathbb K)$ as its first page
  and converges after finitely many steps.
  The last page is (non-canonically) isomorphic to $\HFK(K; \mathbb K)$.
\end{customthm}

\begin{proof}
  The~thesis follows directly from Proposition~\ref{prop:BocksteinConvergesq1},
  which we can apply thanks to Lemma~\ref{lem:Q-fin-gen}. Indeed it states that
  the $(q\mapsto 1)$ Bockstein spectral sequence has
  $\ourQHone(\bdiagram; \mathbb K)$
on the first page and converges to the quotient of
  $\ourQH(\bdiagram; \mathbb K[q^{-1},q])$ by its torsion submodule, tensored with $\mathbb K$.
  The~former is isomorphic to $\homologyglzero(K, \mathbb K)$
  by Proposition~\ref{prop:ourQ-vs-gl0} and we identify the~latter
  with $\HFK(K, \mathbb K)$ as follows. 

  Because $\mathbb K[q^{-1}, q]]$ contains the~field of fractions of
  $\mathbb K[q,q^{-1}]$, the universal coefficient theorem
  and Proposition \ref{prop:ourQ-to-Gilmore} imply that 
  the~homology groups of $\ourQ(\bdiagram, \mathbb K[q,q^{-1}])$ and
  $\complexgilmore(\bdiagram, \mathbb K[q^{-1}, q]])$ have the~same rank.
  On the~other hand, we know from Theorem~\ref{thm:gilmore-vs-hfk}
  that $H(\complexgilmore(\bdiagram, \mathbb K[q^{-1}, q]]))$
  is isomorphic to $\HFK(\bdiagram) \otimes \mathbb K[q^{-1}, q]]$.
  Hence, the quotient of $\ourQH(\bdiagram, \mathbb K[q,q^{-1}])$ by its torsion submodule
  has the same rank as $\HFK({\bdiagram}) \otimes \mathbb K[q, q^{-1}]$
  and we conclude by tensoring both sides with $\mathbb K$.
\end{proof}

If the~characteristic of $\mathbb K$ is $0$,
then we can consider this spectral sequence together with the~one
from Theorem \ref{thm:RW3},
establishing the DGR Conjecture. 

\begin{customthm}{C}[DGR Conjecture]
	For any knot $K$ and any field of characteristic zero, the bigraded dimension of  
	$\HHH^{\mathrm{red}}$ (after forgetting the $\avar$-grading) is greater or equal to the bigraded  dimension of $\HFK$.
\end{customthm}
 
\appendix

\section{On Bockstein spectral sequences}
\label{sec:bockstein}
\subsection{Limits of spectral sequences}\label{subsec:uniqueness}

This short section is a~survey of \cite[Chapter 3]{MR1793722}.
In what follows, we consider decreasing
filtrations of modules and chain complexes, not necessary bounded.
More explicitly, a~\emph{filtration} of a~$\scalars$-module $M$ is
a~sequence $(F^n)_{n \in  \ZZ}$ of submodules of $M$ satisfying
\[
  F^{n} \supseteq F^{n+1}
  \qquad\text{for all }n \in \ZZ.
\]
Its {\it associated graded} module $\gr^{\bullet}(M)$ is defined
as the~sequence of quotients 
\[
	\gr^n(M)=F^n / F^{n+1} \qquad \text{for }n \in \ZZ.
\]
For a~chain complex $(C,d)$ it is understood that the~submodules $F^n$
are subcomplexes.
In such case each chain group $C_i$ is filtered by $F^nC_i := F^n \cap C_i$
and $d(F^nC_i) \subseteq F^nC_{i+1}$ for all $i,n \in \ZZ$.
There is also an~associated filtration on the~homology
with $F^n\!H(C,d)$, defined as the~image of the~natural map
$H(F^n, d)\to H(C,d)$.

With every filtered chain complex $(C,d,F)$ there is an~associated
\emph{spectral sequence} $\{E_r\}_{r\in\mathbb N}$ with the~first page
\[
	E^{n,p-n}_1= H^{p}(\gr^n(C)).
\]
We say that the~spectral sequence \emph{converges} to the~homology of
the~filtered chain complex $(C, d, F)$ if the $\infty$-page is directly
related to the~filtration of $H(C, d)$ by
\[
	E^{n,p-n}_\infty=\gr^n H^{p}(C, d).
\]
By \cite[Theorem 3.2]{MR1793722} this is the~case
if the~filtered chain complex $(C,d,F)$ is \emph{exhaustive},
i.e.\ $\bigcup_{n} F^nC = C$,
and \emph{weakly convergent}.
The~later holds for instance when $\bigcap_{n} F^n C =0$;
for the~general definition see \cite[Definition 3.1]{MR1793722}.

\begin{example}\label{ex:A1}
	Consider graded filtered  $\scalars$-complexes $(C, d, F)$
	and  $(C\oplus \scalars, d\oplus 0, \tilde F)$
	with $\tilde F^n = F^n \oplus \scalars$.
	The~associated graded of these complexes coincide,
	so that they induce the~same exact sequence.
	We deduce that the~spectral sequence converges simultaneously
	to both $H(C,d)$ and $H(C,d)\oplus \scalars$.
\end{example}

To assure the~uniqueness of the~limit of a~spectral sequence
associated with $(C,d,F)$,
the~filtration needs to be \emph{Hausdorff},
that is weakly convergent and
\[
	\bigcap_{n \in \ZZ} F^n\!H(C,d) =0.
\]
Note that the~second filtration in Example \ref{ex:A1} is not Hausdorff.
An~important source of exhaustive and Hausdorff filtrations
is provided by~\emph{completions} of filtered chain complexes.

Recall that the~\emph{completion} of a~filtered module $(M, F)$ is the~inverse limit
\[
	\widehat M := \lim_{\leftarrow s} M / F^s
\]
together with the~filtration		
\[
	\widehat F^n := \lim_{\leftarrow s} F^n / F^{n+s}.
\]
The~filtered module $(\widehat M, \widehat F)$ is exhaustive and Hausdorff
by \cite[Prop. 3.12]{MR1793722}.
Finally, given a~map $f\colon (M,F_M) \to (N,F_N)$ of exhaustive filtered modules,
the~induced map $\hat f\colon \widehat M\to \widehat N$ is an~isomorphism if and only if
$\gr f \colon \gr M \to \gr N$ is an~isomorphism
\cite[Prop. 3.14]{MR1793722}.
This implies the~following important statement (compare \cite[Cor. 3.15]{MR1793722}).

\begin{theorem}
	Assume that $\left\{f_r: E_r \to E'_r\right\}_{r\geq 0}$
	is a~morphism of spectral sequences $E_r$ and $E'_r$
	that converges to $(M,F_M)$ and $(N, F_N)$ respectively.
	If $f_n$ is an~isomorphism for some $n$,
	then	 $f_\infty$ induces an~isomorphism of filtered modules
	$(\widehat M, \widehat F_M)$ and $(\widehat N, \widehat F_N)$.
\end{theorem}	

\begin{example} 
	Consider the~module $\ZZ[t]$ filtered by powers of $t$.
	Its associated graded 
	\[
		\gr\,\ZZ[t]
			= \bigoplus_{n \in \mathbb N} \frac{t^n \ZZ[t]}{t^{n+1} \ZZ[t]}
			= \bigoplus_{n \in \mathbb N} \ZZ \{t^n\}
	\]
	can be identified with $\ZZ[t]$ as an~abelian group
	(note that the~action of $t$ annihilates $\mathrm{gr}\,\ZZ[t]$),
	whereas the~completion is exactly $\ZZ[[t]]$.
	Consider $M:=\ZZ[t]/(1 - tp )$ for some $p\in \ZZ[t]$ with the~induced filtration.		 
	Then $F^n\!M = M$ for any $n$, because $t$ is invertible in $M\ric$.
	This forces $\gr^n M = 0$ for all $n\geq 0$
	and $\widehat M = 0$ as a~result.
	Note that the~filtration is not weakly convergent and also $M[[t]] = 0$.
\end{example}	
	
\begin{example}
	Consider the~ring of Laurent polynomials $\ZZ[t,t^{-1}]$ as a~$\ZZ[t]$-module
	and let $F^n$ by generated by $t^n$ for $n\in\ZZ$.
	Contrary to the~previous case, this filtration is not bounded from below.
	The~associated graded can be identified as an~abelian group with $\ZZ[t,t^{-1}]$
	as in the~previous example and the~completion
	\[
		\widehat{\ZZ[t,t^{-1}]}
			= \lim_{\leftarrow s} \frac{\ZZ[t^{-1},t]}{t^s\ZZ[t]}
			= \ZZ[t^{-1},t]]
	\]
	is filtered by $\widehat F^n = t^n \ZZ[[t]]$ for $t\in \ZZ$.
	Note that the~filtration is exhaustive and weakly convergent.
	As before, consider $M := \ZZ[t,t^{-1}]/(1-tp)$ for some
	polynomial $p\in\ZZ[t]$ with the~induced filtration.
	Again, $t$ is invertible in $M\ric$, so that both $\gr M$ and $\widehat M$ vanish.
\end{example}
 \subsection{The \texorpdfstring{mod-$p$}{mod-p} Bockstein spectral sequence.}
\label{sec:modp}

The aim of  this section is to recall the classical Bockstein
sequence in the context of $\ZZ$-modules,
following \cite{MayPrimerSS}.
This is generalized in section \ref{sec:qmapsto-1-bockstein}
to the~case of modules over Laurent polynomials
and specializing the~value of the~formal variable.

Let $C$ be a~chain complex of $\ZZ$-modules and $p$ a~prime number.
The short exact sequence
\[
	0 \lra \ZZ \stackrel{\cdot p}{\lra} \ZZ
   \stackrel{\pi}{\lra} \ZZ/p\ZZ \lra 0
\]
induces a~long exact sequence of homology groups\footnote{
	Recall that for us a~differential in a~chain complex
	has degree $+1$, i.e.\ it increases the~homological degree.
}
\[
	\ldots
		\xrightarrow{\ \partial\ }   H_{\bullet}(C; \ZZ)
		\xrightarrow{\,H(\cdot p)\,} H_{\bullet}(C; \ZZ)
		\xrightarrow{\,H(\pi)\,}     H_{\bullet}(C; \ZZ/p\ZZ)
		\xrightarrow{\ \partial\ }   H_{\bullet+1}(C; \ZZ)
		\xrightarrow{\,H(\cdot p)\,} \ldots
\]
which can be thought of as an~exact triangle
\[
  \begin{tikzpicture}
    \node (A) at (-1.5,0) {$H(C; \ZZ)$};
    \node (B) at (1.5,0) {$H(C; \ZZ)$};
    \node (C) at (0,-1.5) {$H(C; \ZZ/p\ZZ)$};
    \draw[-to] (A) -- (B) node [midway, above, font=\scriptsize] {$H(\cdot p)$};
    \draw[-to] (B) -- (C) node [midway, right, font=\scriptsize] {$H(\pi)$};
    \draw[-to] (C) -- (A) node[midway, left, font=\scriptsize]
    {$\partial$};\end{tikzpicture}
\]
and described in terms of exact couples as we explain below.

Recall from \cite{MasseyExactCouples} that an~\emph{exact couple}
is a~tuple $(A,B, f,g,h)$ consisting of objects $A$ and $B$ from an~abelian category
and morphisms $f\co A \to A$, $g\co A \to B$ and $h\co B \to A$
satisfying $\image f = \ker g$, $\image g = \ker h$ and $\image h= \ker f\ric$.
Defining
\begin{itemize}
  \item $A' = \image f\ric$,
  \item $B'= \ker (g \circ h) / \image(g \circ h)$,
  \item $f'\co A' \to A'$ as the restriction of $f$ to $A'$,
  \item $h'\co B' \to A'$ as induced by $h$, and
  \item $g'\co A' \to B'$ by declaring that $a' = f(a) \in A'$ is
        mapped on $g(a') = g(a) \in B'$
\end{itemize}
yields another exact couple $(A',B', f',g', h')$.
Inductively one constructs a~sequence of exact couples
$(A^{(n)},B^{(n)}, f^{(n)},g^{(n)}, h^{(n)})_{n \in \NN}$
and checks that
$(B^{(n)}, g^{(n)} \circ f^{(n)})$ is a spectral sequence\footnote{
	Not necessarily bigraded in general.
}. 
The~\emph{Bockstein spetral sequence} arises from the~exact couple
\[
	(H(C; \ZZ), H(C;\ZZ/p\ZZ), H(\cdot p), H(\pi), \partial).
\]

\begin{example}\label{ex:Bocksteinpk}
  Consider the chain complex $C = \ZZ \stackrel{\cdot p^k}{\lra} \ZZ$ 
  for some $k \geq 1$. The first exact couple at stake  is:
\[
\NB{  \begin{tikzpicture}
    \node (A1) at (-1.5,0) {$\ZZ/p^k \ZZ$};
    \node (A2) at ( 1.5,0) {$\ZZ/p^k \ZZ$};
    \node (B) at (0, -1.7) {$ \ZZ/p\ZZ  \oplus \ZZ/p\ZZ $};
\draw[-to] (A1) -- (A2) node[midway, above, font= \scriptsize] {$\cdot p$};
    \draw[-to] (A2) -- (B) node[midway, right, font= \scriptsize]
    {\;$\displaystyle{\begin{pmatrix} 0 \\ 1 \end{pmatrix}} $};
    \draw[-to] (B) -- (A1) node[midway, left, font= \scriptsize]
    {$\left(p^{k-1}\,\,\, 0 \right)$};
  \end{tikzpicture}}.
\]
In general, for $1\leq  i\leq k$, the $i$th exact couple  is given by:
\[\NB{
  \begin{tikzpicture}
    \node (A1) at (-1.5,0) {$\ZZ/p^{k+1-i} \ZZ$};
    \node (A2) at ( 1.5,0) {$\ZZ/p^{k+1-i} \ZZ$};
    \node (B) at (0, -1.7) {$\ZZ/p\ZZ \oplus \ZZ/p\ZZ$};
\draw[-to] (A1) -- (A2) node[midway, above, font= \scriptsize] {$\cdot p$};
\draw[-to] (A2) -- (B) node[midway, right, font= \scriptsize]
    {\;$\displaystyle{\begin{pmatrix} 0 \\ 1 \end{pmatrix}} $};
    \draw[-to] (B) -- (A1) node[midway, left, font= \scriptsize]
    {$\left(p^{k-i}\,\,\, 0 \right)$};

  \end{tikzpicture}
  }
\]
and finally the $(k+1)$-st exact couple is identically $0$.
\end{example}

\begin{proposition}\label{prop:BocksteinConvergesModP}
  The first page of the~Bockstein spectral sequence of a~chain complex
  $C$ of abelian groups is $H(C; \ZZ/p\ZZ)$.
  If the chain complex $C$ is free and finitely generated,
  then the~Bockstein spectral sequence converges in finitely many steps
  and the infinite page is canonically isomorphic to
  the~quotient of\/ $H(C; \ZZ)$ by its torsion submodule, tensored with $\ZZ/p\ZZ$. 
\end{proposition}

\begin{proof}[Sketch of the~proof]
  This is a~very classical result and the proof is rather elementary.
  First, using Smith normal form of differentials,
  one obtains that every free and finitely generated complex of
  $\ZZ$-modules is a direct sum of shifted complexes of the form
  \begin{enumerate}\item\label{it:diffcomplex1} $0 \lra \ZZ \lra 0$,
    \item\label{it:diffcomplex2} $0\lra \ZZ \stackrel{\cdot r}{\lra}
      \ZZ \lra 0$ with $r$ an integer coprime with $p$,
    \item\label{it:diffcomplex3} $0\lra \ZZ \xrightarrow{\,\cdot
        p^kr\,} \ZZ \lra 0$ with $k\geq 1$ and $r$ an integer
      coprime with $p$.
  \end{enumerate}
  In case (\ref{it:diffcomplex1}), the spectral sequence converges
  immediately and its infinite page is equal to $\ZZ/p\ZZ$.
  In case (\ref{it:diffcomplex2}), the spectral sequence converges
  immediately and its infinite page is equal to $0$.
  Case (\ref{it:diffcomplex3}) is dealt with in
  Example~\ref{ex:Bocksteinpk}: it converges to $0$ at the $(k+1)$-st page. 
\end{proof}

\subsection{The \texorpdfstring{$(q\mapsto 1)$}{(q->1)} Bockstein sequence }
\label{sec:qmapsto-1-bockstein}

Let $\mathbb K$ be a~field and $\laurent:=\mathbb K[q, q^{-1}]$
be the~ring of Laurent polynomial over $\mathbb K$.
Note that $\laurent$ is a~principal ideal domain (PID),
so that Smith's normal form result applies.
We endow $\mathbb K$  with an~$\laurent$-module structure by letting
$q$ act by $1$. In other words, we have an~exact sequence of $\laurent$-modules
\[
	0 \lra \laurent
		\xrightarrow{\,\cdot (q-1)\,} \laurent
		\xrightarrow{\,q\mapsto 1\,} \mathbb K
		\lra 0.
\]
Let $C$ be a~chain complex of $\laurent$-modules.
Just like in subsection~\ref{sec:modp},
one can use the induced long exact sequence of homology
to construct an~exact couple
\[
	(H(C; \laurent), H(C;
  	\mathbb K), H(\cdot (q-1)), H(q\mapsto 1), \partial).
\]
This exact couple induces a~spectral sequence,
which we call the~\emph{$(q\mapsto 1)$ Bockstein spectral sequence}.

\begin{proposition}\label{prop:BocksteinConvergesq1}
	The first page of the~$(q\mapsto 1)$ Bockstein spectral sequence
	of a chain complex $C$ of\/ $\laurent$-modules is $H(C; \mathbb K)$.
	If the~chain complex $C$ is free and finitely generated,
	then the~$(q\mapsto 1)$ Bockstein spectral sequence converges
	in finitely many steps the infinite page is canonically isomorphic
	to the~quotient of\/ $H(C; \laurent)$ by its torsion submodule, tensored with\/ $\mathbb K$.
\end{proposition}

\begin{proof}[Sketch of the~proof]
  The~proof follows the~same line as the~one of
  Proposition~\ref{prop:BocksteinConvergesModP}.
  Every free and finitely generated complex of $\laurent$-module
  is a~direct sum of shifted complexes of the~form
  \begin{enumerate}
    \item\label{it:diffcomplex1q1}
      $0 \lra \laurent \lra 0$,
    \item\label{it:diffcomplex2q1}
      $0\lra \laurent \xrightarrow{\,\cdot p(q)\,} \laurent \lra 0$
      with $p(q)$ a~polynomial coprime with $q-1$, or
    \item\label{it:diffcomplex3q1}
      $0\lra \laurent \xrightarrow{\,\cdot (q-1)^k p(q)\,} \laurent \lra 0$
      with $k\geq 1$ and $p(q)$ a~polynomial coprime with $q-1$.
  \end{enumerate}
  In case (\ref{it:diffcomplex1q1}), the spectral sequence converges
  immediately and its infinite page is equal to $\mathbb K$.
  In case (\ref{it:diffcomplex2q1}), the spectral sequence converges
  immediately and its infinite page is equal to $0$.
  Case (\ref{it:diffcomplex3q1}) is similar to
  Example~\ref{ex:Bocksteinpk}. The first exact couple at stake is
  \[
   \NB{\begin{tikzpicture}
      \node (A1) at (-2,0) {$\laurent/ \langle(q-1)^k p(q)\rangle $};
      \node (A2) at ( 2,0) {$\laurent/ \langle(q-1)^k p(q)\rangle $};
      \node (B) at (0, -1.7) {$\mathbb K \oplus \mathbb K$};
      \draw[-to] (A1) -- (A2) node[midway, above, font= \scriptsize] {$\cdot (q-1)$};
      \draw[-to] (A2) -- (B) node[midway, right, font= \scriptsize]
        {\;\;$\displaystyle{\begin{pmatrix} 0 \\ 1 \end{pmatrix}} $};
      \draw[-to] (B) -- (A1) node[midway, left, font= \scriptsize]
      {$\left((q-1)^{k-1} p(q)\,\,\, 0 \right)\;$};
    \end{tikzpicture}}.
  \]
  In general, for $1\leq  i\leq k$, the $i$th exact couple  is given by
  \[\NB{
    \begin{tikzpicture}
      \node (A1) at (-2.5,0) {$\laurent/\langle(q-1)^{k+1-i} p(q)\rangle $};
      \node (A2) at ( 2.5,0) {$\laurent/\langle (q-1)^{k+1-i} p(q) \rangle $};
      \node (B) at (0, -1.7) {$\mathbb K \oplus \mathbb K$};
      \draw[-to] (A1) -- (A2) node[midway, above, font= \scriptsize] {$\cdot (q-1)$};
      \draw[-to] (A2) -- (B) node[midway, right, font= \scriptsize]
        {\;\;\; $\displaystyle{\begin{pmatrix} 0 \\ 1 \end{pmatrix}} $};
      \draw[-to] (B) -- (A1) node[midway, left, font= \scriptsize]
        {$\left((q-1)^{k-i} p(q)\,\,\, 0 \right)\;$};
    \end{tikzpicture}}.
  \]
  Finally, the~$(k+1)$-st exact couple is identically $0$.
  Hence, in all three cases the $(q\mapsto 1)$ Bockstein spectral
  sequence converges to the~quotient of $H(C; \laurent)$
  by its torsion submodule, tensored with $\mathbb K$ as desired.
\end{proof}

\section{Cyclicity of the~quantum Hochschild homology}
\label{sec:qHH-is-cyclic}
For this section we fix a~graded algebra $A$ and consider its quantum
Hochschild complex $\qCHoHom_\bullet(A)$ with the~differential denoted
by $\partial$.  The~complex arises actually from a~\emph{simplicial
module},\footnote{
	For more details about simplicial and cyclic module see \cite{LodayBook}.
}
which means that each chain group $\qCHoHom_n(A)$ admits
two families of homomorphisms: the~family of \emph{face maps}
$\{d_i\colon M_n\to M_{n-1}\}_{0\leqslant i\leqslant n}$ and of
\emph{degeneracy maps}
$\{s_j\colon M_n\to M_{n+1}\}_{0\leqslant j\leqslant n}$, which
satisfy the~equalities
\begin{align}\label{rels:simplicial}
	d_id_j &= d_{j-1}d_i \quad \textrm{for } i<j,\\
	s_is_j &= s_js_{i-1} \quad \textrm{for } i>j,\\
	d_is_j &= \begin{cases}
		s_{j-1}d_i & \textrm{for } i<j,\\
		\id        & \textrm{for } i=j,j+1,\\
		s_jd_{i-1} & \textrm{for } i>j+1.
	\end{cases}
\end{align}
Indeed, the~face maps are the~components of the~quantum Hochschild differential,
\begingroup
\def\qntterm{q^{-|a_n|} a_n a_0}\begin{align*}
	d_i(a_0\otimes \dots \otimes a_n) &:= \begin{cases}
		\hphantom{\qntterm}\mathllap{a_0 a_1}\otimes a_2 \otimes \dots \otimes a_n
			& \textrm{ if }i=0,\\
		\hphantom{\qntterm}\mathllap{a_0} \otimes \dots \otimes a_ia_{i+1}
			\otimes \dots \otimes a_n
			& \textrm{ if }0<i<n,\\
		\qntterm \otimes a_1 \otimes \dots \otimes a_{n-1} & \textrm{ if }i=n,
	\end{cases}
\intertext{whereas the~degeracy map $s_j$ inserts $1\in A$ after $j$-th factor:
}
	s_j(a_0\otimes \dots \otimes a_n) &:= a_0 \otimes \dots \otimes
			a_j \otimes 1 \otimes a_{j+1} \otimes \dots \otimes a_n.
\intertext{In addition to that, there is a~family of component-wise endomorphisms
}
	t_n(a_0 \otimes \dots \otimes a_n) &:=
		q^{-|a_n|} a_n \otimes a_0 \otimes \dots \otimes a_{n-1},
\end{align*}\endgroup
which satisfy the~equalities
\begin{equation}\label{rels:semicyclic}
	d_it_n = \begin{cases}
		\phantom{t_{n-1}}d_n & \textrm{ for }i=0,\\
		t_{n-1}d_{i-1}       & \textrm{ for } i>0,
	\end{cases}
	\qquad
	s_jt_n = \begin{cases}
		t_{n+1}^2s_n   & \textrm{ for }j=0,\\
		t_{n+1}s_{j-1} & \textrm{ for }j>0.
	\end{cases}
\end{equation}
Consider the~endomorphism $T$ of $\qCHoHom_\bullet(A)$ defined
by $T_n := t_n^{n+1}$. It is the~identity map when $q=1$,
which means that the~classical Hochschild homology is a~cyclic module,
but in general case it scales a~homogeneous degree $d$ Hochschild
chain by $q^d$. However, it is not far from the~identity map.

\begin{lemma}\label{lem:T-htpic-to-id}
	The~endomorphism $T$ is chain homotopic to the~identity map.
\end{lemma}
\begin{proof}
	Define $\sigma_n := t_{n+1}s_n$, so that
	\begin{equation}
		d_i \sigma_n = \begin{cases}
			\id         & \textrm{for }i=0,\\
			\sigma_{n-1} d_{i-1} & \textrm{for }0<i<n,\\
			t_n         & \textrm{for }i=n.
		\end{cases}
	\end{equation}
	We claim that $h_n=\sum_{j=0}^n(-1)^{jn} \sigma_n t_n^j$ is
	a~desired chain homotopy. First, write
	\begin{align}
		\label{htpy:hd}
		h_{n-1} \partial_n &= \sum_{j=0}^{n-1} \sum_{i=0}^n (-1)^{i+j(n-1)} \sigma_{n-1} t_{n-1}^j d_i\\
		\label{htpy:dh}
		\partial_{n+1} h_n &= \sum_{i=0}^{n+1} \sum_{j=0}^n (-1)^{i+jn} d_i \sigma_n t_n^j
	\end{align}
	and notice the~following cancellation in \eqref{htpy:dh}:
	\begin{equation}
		   (-1)^{n+1+jn} d_{n+1} \sigma_n t_n^j
		= -(-1)^{(j+1)n} t_n^{j+1} \\
		= -(-1)^{(j+1)n} d_0 \sigma_n t_n^{j+1}.
	\end{equation}
	Hence,
	\begin{align}
		\sum_{j=0}^n (-1)^{jn} (d_0 - (-1)^nd_{n+1}) \sigma_n t_n^j
		= d_0 \sigma_n - d_{n+1} \sigma_n t_n^n = \id - t_n^{n+1}.
	\end{align}
	Put the~remaining terms of $\partial h$ as well as the~terms of $h\partial$ in the~lexicographic order with respect to $i$ then $j$, to create $n(n+1)$ pairs:
	\begin{equation}\label{eq:the-order-of-terms}
		\setlength\arraycolsep{0.7ex}
		\begin{array}{ccccccccccccc}
			d_1 \sigma_n     &<& d_2 \sigma_n     &<&\cdots&<& d_n\sigma_n         &<& d_1\sigma_n t_n  &<& d_2\sigma_n t_n &<&\cdots
			\\[1ex]
			\updownarrow     & & \updownarrow     & &      & & \updownarrow        & & \updownarrow     & & \updownarrow
			\\[1ex]
			\sigma_{n-1} d_0 &<& \sigma_{n-1} d_1 &<&\cdots&<& \sigma_{n-1}d_{n-1} &<& \sigma_{n-1} d_n &<& \sigma_{n-1} t_{n-1} d_0 &<&\cdots
		\end{array}
	\end{equation}
	It is enough to show that none of the~pair contributes
	to $\partial h + h\partial$.
	
	The~term $d_{i+1} \sigma_n t_n^j$ is at the~position $jn+i+1$ in the~upper sequence of \eqref{eq:the-order-of-terms} and it appears in \eqref{htpy:dh} with sign $(-1)^{jn+i+1}$. We compute
	\begin{equation}
		d_{i+1} \sigma_n t_n^j = \sigma_{n-1} d_i t_n^j = \begin{cases}
			\sigma_{n-1} t_{n-1}^{j-1} d_{i-j+n+1} & \textrm{if }0\leqslant i<j,\\
			\sigma_{n-1} t_{n-1}^j d_{i-j}         & \textrm{if }j\leqslant i<n,
		\end{cases}
	\end{equation}
	obtaining a~term at the~position $jn+i+1$ in the~lower sequence of
	\eqref{eq:the-order-of-terms}, which appears in \eqref{htpy:hd} with
	sign $(-1)^{j(n-1)+i-j} = (-1)^{jn+i}$.
	Hence, the~two terms cancel each other and the~thesis follows.
\end{proof}

We are now ready to prove the~statement about quantum Hochschild homology
for a~polynomial algebra $R^{\underline k}$. In fact, Proposition~\ref{prop:qHH-vanishes}
is a~special case of the~following result.

\begin{proposition}
	Suppose that $A$ is supported in nonnegative degrees and that\/ $1-q^d$ is invertible
	for $d\neq 0$. Then the~inclusion of the~degree zero subalgebra $A_0 \subset A$
	induces a~homotopy equivalence
	of chain complexes $\qCHoHom_{\!\bullet}(A_0) \to \qCHoHom_{\!\bullet}(A)$.
	In particular, $\qHoHom_{\!\bullet}(A) \cong \qHoHom_{\!\bullet}(A_0)$.
\end{proposition}
\begin{proof}
	Let $T$ be the~endomorphism of $\qCHoHom_{\!\bullet}(A)$
	that maps a~homogeneous chain $c$ to $q^{|c|}c$.
	The~map $T - \id$ is nullhomotopic by Lemma~\ref{lem:T-htpic-to-id},
	so that the~subcomplex generated by chains of positive degree is contractible,
	whereas the~degree 0 subcomplex coinsides with $\qCHoHom_{\!\bullet}(A_0)$.
\end{proof}
 
\section{Computations of \texorpdfstring{$\gll_0$}{gl(0)} homology}
\label{sec:computations-gl0}

This section provides details of computation
of the~$\gll_0$ homology of the~two trefoil knots,
the~figure eight knot and the~$(5,2)$-torus knot.
These computation is used in the~introduction to prove detection results.

\subsection{Trefoils}
\label{sec:trefoils}

We see the right-handed (resp.{} left-handed) trefoil $3_1$ (resp.{}
$\bar 3_1$) as the closure
of the braid $\sigma_1^3$ (resp.{} $\sigma_1^{-3}$) in  the braid
group on two strands.
We start with $3_1$.  Following the definition of $\homologyglzero$,
we consider the hypercube given on
Figure~\ref{fig:hypercubeK3plus}. On this figure bases of some
$\gll_0$-state spaces are given. The fact that they are indeed bases
follows directly from the digon relation and Theorem~\ref{thm:RW3} (\ref{it:rank-gl0}).
\begin{figure}[ht]
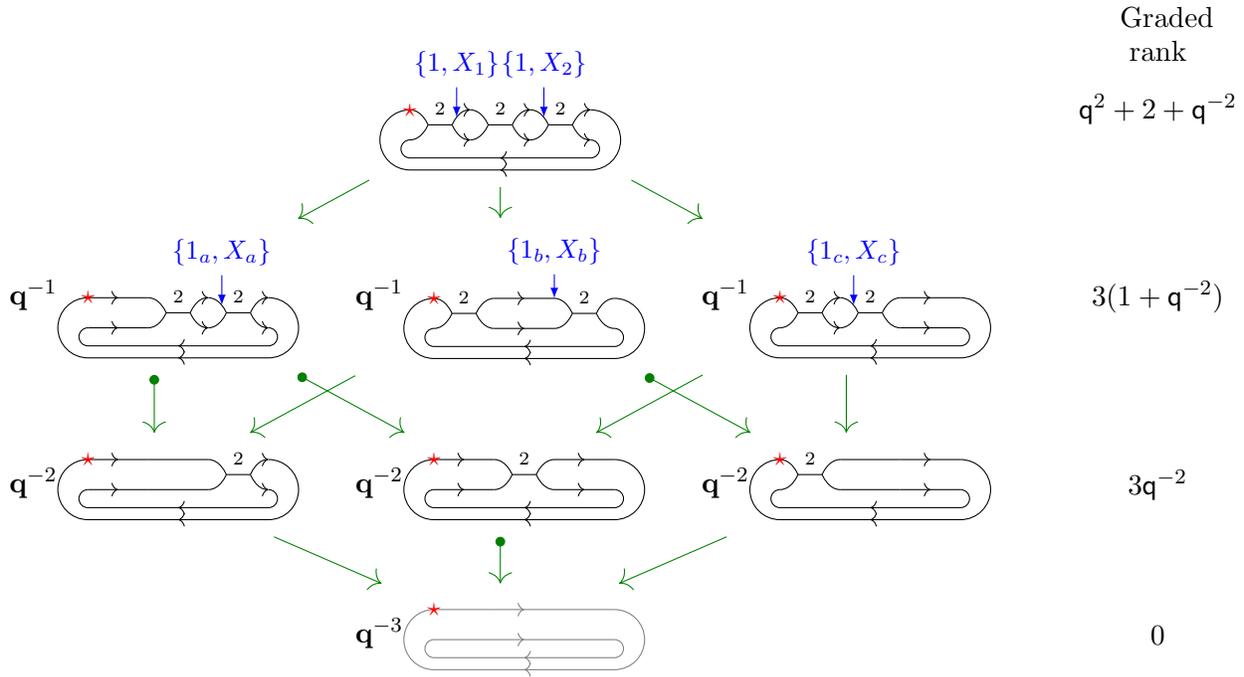

  \centering
  \begin{tikzpicture}[xscale = 2.3]
    \node (ddd) at (0,1) {\NB{\tikz[scale=0.8]{\begin{scope}
  \nicedumbell{D1}{0}{0};
  \nicedumbell{D2}{1}{0};
  \nicedumbell{D3}{2}{0};
  \draw[->-] (D3RT) arc (90:-90: 0.5) -- ++(-3, 0) arc (270:90:0.5cm);
  \draw[->-] (D3RB) arc (90:-90: 0.15) -- ++(-3, 0) arc (270:90:0.15cm);
  \node[red, font=\small] at (D1LT) {$\star$};
  \draw[blue, latex-] (D1rt) -- +(0, 0.5)  node[blue, above, font =\small]
  {$\{1, X_1\}$};
  \draw[blue, latex-] (D3lt) -- +(0, 0.5)  node[blue, above, font =\small]
  {$\{1, X_2\}$};
\end{scope}

}}}; 
    \node (dsd) at (0,-1.5) {$\qshift^{-1}$\NB{\tikz[scale=0.8]{\input{\imagesfolder/hfgl0_trefoildsd-base}}}};
    \node (sdd) at (-2,-1.5) {$\qshift^{-1}$\NB{\tikz[scale=0.8]{\input{\imagesfolder/hfgl0_trefoilsdd-base}}}};
    \node (dds) at (2,-1.5) {$\qshift^{-1}$\NB{\tikz[scale=0.8]{\input{\imagesfolder/hfgl0_trefoildds-base}}}} ;
    \node (sds) at (0,-4) {$\qshift^{-2}$\NB{\tikz[scale=0.8]{\input{\imagesfolder/hfgl0_trefoilsds}}}} ;
    \node (ssd) at (-2,-4) {$\qshift^{-2}$\NB{\tikz[scale=0.8]{\input{\imagesfolder/hfgl0_trefoilssd}}}} ;
    \node (dss) at (2,-4) {$\qshift^{-2}$\NB{\tikz[scale=0.8]{\input{\imagesfolder/hfgl0_trefoildss}}}};
    \node (sss) at (0,-6) {$\qshift^{-3}$\NB{\tikz[scale=0.8, gray]{\input{\imagesfolder/hfgl0_trefoilsss}}}};
    \node[text width =1cm, align=center] (dimension) at (3.8, 2) {Graded rank};
    \node (gd0) at (3.8, 1) {$\qvar^2+ 2 + \qvar^{-2}$};
    \node (gd0) at (3.8, -1.5) {$3(1 +\qvar^{-2})$};
    \node (gd0) at (3.8, -4) {$3\qvar^{-2}$};
    \node (gd0) at (3.8, -6) {$0$};
	\begin{scope}[green!50!black,  >={To[length=1.5mm]}]
	    \draw[->] (ddd) -- (dds);
	    \draw[->] (ddd) -- (dsd);
	    \draw[->] (ddd) -- (sdd);
	    \draw[->] (dds) -- (dss);
	    \draw[->] (dds) -- (sds);
	    \draw[Circle->] (dsd) -- (dss);
	    \draw[->] (dsd) -- (ssd);
	    \draw[Circle->] (sdd) -- (sds);
	    \draw[Circle->] (sdd) -- (ssd);
	    \draw[->] (ssd) -- (sss);
	    \draw[Circle->] (sds) -- (sss);
	    \draw[->] (dss) -- (sss);
	\end{scope}
  \end{tikzpicture}
  \caption{The~hypercube for computing the~$\gll_0$ homology of
  	the~right-handed trefoil. On the~upper four diagrams
    their homogeneous bases are given schematically in blue.
    The~diagram on top (resp.\ bottom) is in homological degree
    $0$ (resp.\ $3$) and all maps between diagrams are given by unzips.
    A~dot at the beginning of an~arrow indicates the~multiplication
    by $(-1)$ in the differential.
  }\label{fig:hypercubeK3plus}
\end{figure}

In the bases given in Figure~\ref{fig:hypercubeK3plus}, the two
non-trivial differentials are given by:
\[
  \begin{blockarray}{ccccc}\\
    1 & X_1& X_2 & X_1X_2& \\
    \begin{block}{(cccc)c}
    1 & 0 & 0 &0 & 1_a\\
    1 & 0 & 0 & 0& 1_b\\
    1 & 0 & 0 &0 &1_c\\
    0 & 0 & 1 & 0 &X_a\\
    0 & 1 & 1 &0 & X_b\\
    0 & 1 & 0 & 0 & X_c\\
  \end{block}
  \end{blockarray}
  \qquad \text{and} \qquad
  \begin{blockarray}{cccccc}\\
    1_a & 1_b & 1_c& X_a& X_b & X_c \\
  \begin{block}{(cccccc)}
    -1  & 1  & 0 & 0 & 0 & 0 \\
     -1 & 0  & 1 & 0 & 0 & 0 \\
    0   & -1 & 1 & 0 & 0 & 0 \\ 
  \end{block}
  \end{blockarray}
\]
Hence, the Poincar\'e polynomial of $\homologyglzero(3_1)$
with coefficients in either $\Ztwo$ or $\QQ$ is equal to:
\begin{equation*}
\tvar^0\qvar^2 + \tvar\qvar^0 + \tvar^2\qvar^{-2}.
\end{equation*}

For the~left-handed trefoil all arrows in the~hypercube are reversed
whereas the~homological degree and the~$q$-grading shifts are opposite
to that of Figure~\ref{fig:hypercubeK3plus}. The matrices of the~two
non-trivials differentials are
\[
  \begin{blockarray}{cccc}\\
  \begin{block}{(ccc)c}
     0 &  0 & 0 & 1_a \\
     0 &  0 & 0 & 1_b\\
     0 &  0 & 0 & 1_c\\ 
    -1 & -1 & 0 & X_a \\
     1 &  0 & -1 & X_b\\
     0 &  1 &  1 & X_c\\ 
  \end{block}
  \end{blockarray}
  \qquad \text{and} \qquad
  \begin{blockarray}{ccccccc}\\
    1_a & 1_b & 1_c & X_a& X_b & X_c& \\
    \begin{block}{(cccccc)c}
    0 &  0 &  0 & 0&  0 & 0 &  1\\
    1 &  1 &  0 & 0&  0 & 0 &X_1\\
    0 &  1 &  1 & 0&  0 & 0 & X_2\\
    0 &  0 &  0 & 1&  1 & 1 & X_1X_2\\
  \end{block}
  \end{blockarray}.
\]
Hence, the Poincar\'e polynomial of $\homologyglzero(\bar 3_1)$
with coefficients in either $\Ztwo$ or $\QQ$ is equal to:
\begin{equation*}
\tvar^{-2}\qvar^2 + \tvar^{-1}\qvar^0 + \tvar^0\qvar^{-2}.
\end{equation*}

For a~comparison we give the~Poincar\'e polynomials
of the~reduced triply graded homology of the trefoil knots below:
\begin{align*}
P_{3_1}(\tvar, \avar,\qvar) &=
  \tvar^2\avar^{-2}\qvar^{-2}+\tvar^{1}\avar^{-4}\qvar^0+\tvar^{0}\avar^{-2}\qvar^{2}, \\
  P_{\bar 3_1}(\tvar, \avar,\qvar) &= \tvar^{-2}\avar^{2}\qvar^{2}+\tvar^{-1}\avar^{4}\qvar^0+\tvar^{0}\avar^{2}\qvar^{-2} .
\end{align*}

\subsection{Figure-eight knot}
\label{sec:figure-eight}

We consider the~figure-eight knot $4_1$ as the~closure
of the~braid $\sigma_1\sigma_2^{-1}\sigma_1\sigma_2^{-1}$ on three strands.
Following the~definition of $\homologyglzero$
we build the~hypercube given on Figure~\ref{fig:hypercubeK4}
(disconnected diagrams are skipped, because the~associated spaces vanishes).
\begin{figure}[ht]
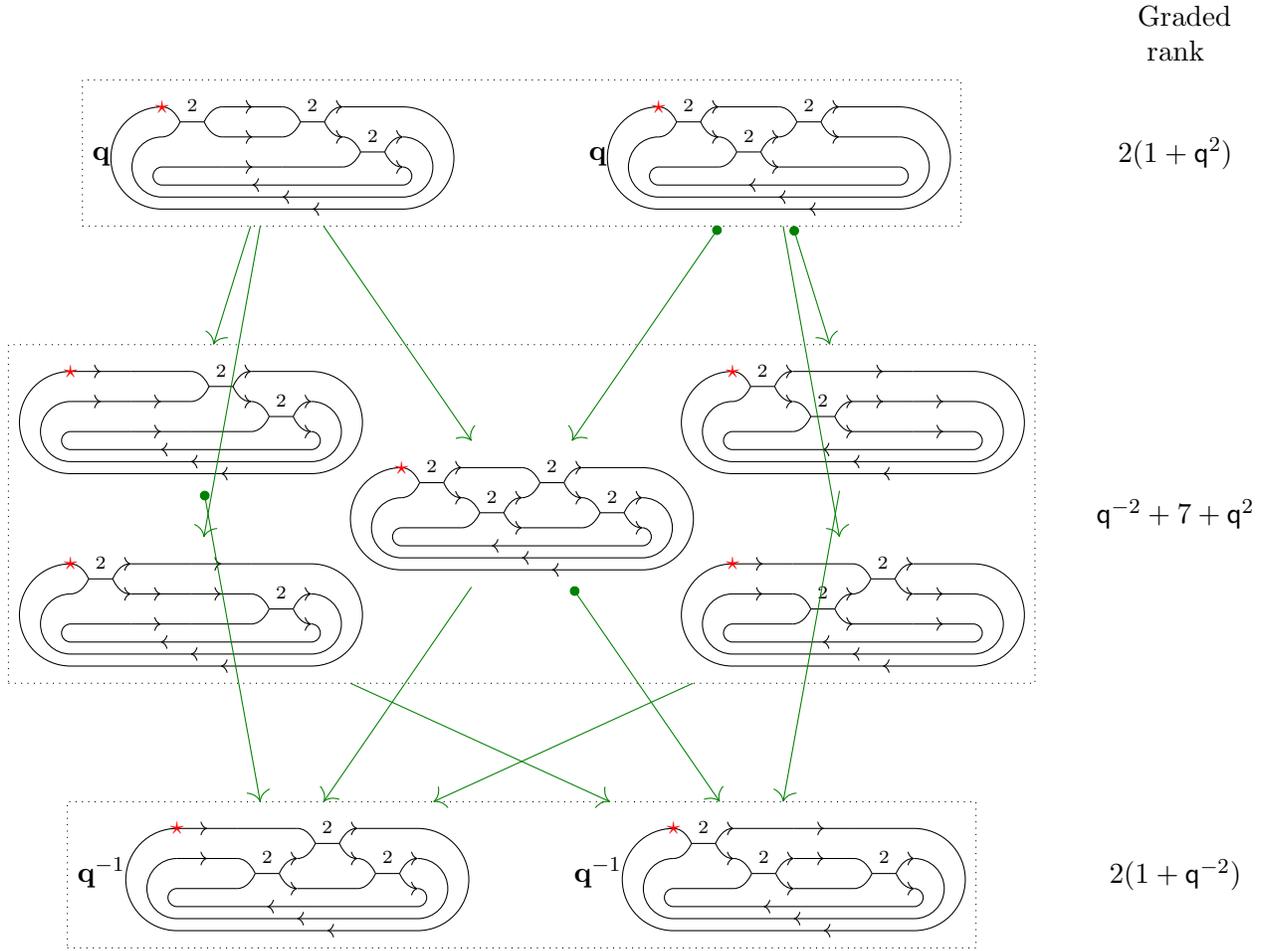

  \centering
  \begin{tikzpicture}[xscale = 2.2, yscale = 1.6]
    \node (dsdd) at (-1.5,0) {$\qshift$\NB{\tikz[scale=0.8]{\input{\imagesfolder/hfgl0_fig8dsdd}}}};
    \node (ddds) at (1.5,0) {$\qshift$\NB{\tikz[scale=0.8]{\input{\imagesfolder/hfgl0_fig8ddds}}}}; 
    \node (dssd) at (-2,-3.8) {\NB{\tikz[scale=0.8]{\input{\imagesfolder/hfgl0_fig8dssd}}}};
    \node (ssdd) at (-2,-2.2) {\NB{\tikz[scale=0.8]{\input{\imagesfolder/hfgl0_fig8ssdd}}}};
    \node (sdds) at (2,-3.8) {\NB{\tikz[scale=0.8]{\input{\imagesfolder/hfgl0_fig8sdds}}}};
    \node (ddss) at (2,-2.2) {\NB{\tikz[scale=0.8]{\input{\imagesfolder/hfgl0_fig8ddss}}}};
    \node (dddd) at (0,-3) {\NB{\tikz[scale=0.8]{\input{\imagesfolder/hfgl0_fig8dddd}}}};
    \node (sddd) at (-1.5,-6) {$\qshift^{-1}$\NB{\tikz[scale=0.8]{\input{\imagesfolder/hfgl0_fig8sddd}}}};
    \node (ddsd) at (1.5,-6) {$\qshift^{-1}$\NB{\tikz[scale=0.8]{\input{\imagesfolder/hfgl0_fig8ddsd}}}}; 
    \begin{scope}[xshift=3.95cm]
    \node[text width =1cm, align=center] (dimension) at (0, 1) {Graded rank};
    \node (gd0) at (0, 0) {$2(1+ \qvar^{2})$};
    \node (gd0) at (0, -3) {$\qvar^{-2} + 7 + \qvar^{2}$};
    \node (gd0) at (0, -6) {$2(1+\qvar^{-2})$};
  \end{scope}
  \begin{scope}[dotted]
    \draw (ssdd.north west) rectangle (sdds.south east);
    \draw (dsdd.north west) rectangle (ddds.south east);
    \draw (sddd.north west) rectangle (ddsd.south east);
  \end{scope}
    \begin{scope}[green!50!black, >={To[length=1.5mm]}]
    \draw[->] (dsdd) -- (dddd);
    \draw[->] (dsdd) -- (dssd);
    \draw[->] (dsdd) -- (ssdd);
    \draw[Circle->] (ddds) -- (dddd);
    \draw[Circle->] (ddds) -- (ddss);
    \draw[->] (ddds) -- (sdds);
    \draw[->] (dddd) -- (sddd);
    \draw[Circle->] (dddd) -- (ddsd);
    \draw[Circle->] (ssdd) -- (sddd);
    \draw[->] (dssd) -- (ddsd);
    \draw[->] (sdds) -- (sddd);
    \draw[->] (ddss) -- (ddsd);
    \end{scope}
  \end{tikzpicture}
  \caption{The~hypercube for computing the~$\gll_0$ homology of
    the~figure-eight knot.
    The two diagrams on top (resp.\ bottom) are in
    homological degree $-1$ (resp.\ $1$).
  }\label{fig:hypercubeK4}
\end{figure}
One could compute explicit bases for all diagrams.
However, this is not necessary.
Over $\Ztwo$ or $\QQ$ one easily obtains that the~graded rank
of $\homologyglzero(4_1)$ in homological degrees $-1$ and $1$ is
respectively $\qvar^2$ and $\qvar^{-2}$.
Using the~Euler characteristic argument,
we conclude that the~Poincar\'e polynomial of $\homologyglzero(4_1)$
with coefficients in either $\Ztwo$ or $\QQ$ is equal to:
\begin{equation*}
	\tvar^{-1}\qvar^2 + 3\tvar^{0}\qvar^{0} + \tvar^1\qvar^{-2}.
\end{equation*}
For comparison we give the~Poincar\'e polynomials of the~reduced triply
graded homology of the~figure-eight knot:
\begin{equation*}
  P_{4_1}(\tvar, \avar,\qvar)  =
  \tvar^0\avar^{2}\qvar^{0}+\tvar^{1}\avar^0\qvar^{-2}+\tvar^{0}\avar^{0}\qvar^{0}+\tvar^{-1}\avar^0\qvar^{2}+\tvar^{0}\avar^{-2}\qvar^0
\end{equation*}

\subsection{(5,2)-torus knot}
\label{sec:5-2-torus}

The (5,2)-torus knot $5_1$ can be presented as the closure
of the~braid $\sigma_1^5$ on two strands.
Computing its homology directly from the~hypercube of resolutions
is a~bit tedious, since it requires \emph{a priori} to
compute 32 state spaces (which are, however, quite simple).
Fortunately, the~complex of Soergel bimodules associated with
$\sigma_5$ is homotopic (see \cite{MR2339573}) to:

\[
  \begin{tikzpicture}[xscale=2.4]
    \node (A) at (0,0) {\NB{\tikz[scale=0.6]{}}};
    \node (B) at (1.2,0) {\NB{\tikz[scale=0.6]{}}};
    \node (C) at (2.4,0) {\NB{\tikz[scale=0.6]{}}};
    \node (D) at (3.6,0) {\NB{\tikz[scale=0.6]{}}};
    \node (E) at (4.8,0) {\NB{\tikz[scale=0.6]{}}};
    \node (F) at (6.0,0) {\NB{\tikz[scale=0.6]{}}};
    \begin{scope}[green!50!black, >={To[length=1.5mm]}]
      \draw[->] (A) -- (B) node[above, pos=0.5, scale=0.7] {$x-y'$};
      \draw[->] (B) -- (C) node[above, pos=0.5, scale=0.7] {$x-x'$};
      \draw[->] (C) -- (D) node[above, pos=0.5, scale=0.7] {$x-y'$};
      \draw[->] (D) -- (E) node[above, pos=0.5, scale=0.7] {$x-x'$};
      \draw[->] (E) -- (F) node[above, pos=0.5, scale=0.7] {$\piz$};
    \end{scope}
  \end{tikzpicture}
\]
where $x$ and $y$ act on the left and $x'$ and $y'$ on the right. We
ignore homological grading shifts for the moment. In this setting all arrows except
the last one have degree $2$, the last one has degree $1$.

We can use this simplification for computing $\gll_0$ homology. In
that context, all spaces have dimension $1$, the last one has
dimension $0$ and all maps are zero. Taking care of the $(q,t)$-grading,
we obtain that  the Poincar\'e polynomial for $\homologyglzero(5_1)$
with coefficients in either $\Ztwo$ or $\QQ$ is equal to:
\begin{equation*}
  \label{eq:poincare-polK51}
  \tvar^{0}\qvar^4 + \tvar^{1}\qvar^2 + \tvar^{2}\qvar^{0} +   \tvar^{3}\qvar^{-2} + \tvar^{4}\qvar^{-4}.
\end{equation*}
For comparison, here is the~Poincar\'e polynomials of the~reduced triply
graded homology:
\begin{equation*}
P_{5_1}(\tvar, \avar,\qvar)  =  \tvar^{0}\avar^{-4} \qvar^4 +
  \tvar^{1}\avar^{-6}\qvar^2  + \tvar^{2}\avar^{-4}\qvar^{0} +
  \tvar^{3}\avar^{-6}\qvar^{-2} + \tvar^{4}\avar^{-4} \qvar^{-4}.
\end{equation*}

\bibliographystyle{alphaurl}
\bibliography{biblio}

\end{document}